\documentclass[smallcondensed]{svjour3}

\usepackage[english]{babel}
\usepackage[latin1]{inputenc}
\usepackage[T1]{fontenc}
\usepackage{amsmath}
\usepackage{graphicx}
\usepackage{amssymb}
\usepackage{comment}
\usepackage{booktabs}
\usepackage{amsfonts}
\usepackage{mathcomp}
\usepackage{parskip}
\usepackage{boxedminipage}
\usepackage{setspace}
\usepackage{enumerate}
\usepackage{hyperref}
\usepackage{esvect}
\usepackage{ulem}
\usepackage{stmaryrd}
\usepackage{amsthm}
\usepackage{subfig}
\usepackage{bm}
\usepackage{tikz}
\usetikzlibrary{patterns}
\usetikzlibrary{decorations.markings}
\usetikzlibrary{arrows}
\usepackage{anyfontsize}
\usepackage[a4paper]{geometry}
\usepackage{setspace}
\usepackage{longtable}
\usepackage{pgfplots,xparse}
\pgfplotsset{compat=newest}

\newcommand{\pderivative}[2]{\frac{\partial #1}{\partial #2}}

\newcommand{\dV}{\,\mathrm{d}V}
\newcommand{\dS}{\,\mathrm{d}S}
\newcommand{\dt}{\,\mathrm{d}t}

\newcommand{\jump}[1]{\left\llbracket #1 \right\rrbracket}
\newcommand{\avg}[1]{\left\{\!\left\{#1\right\}\!\right\}}

\newcommand{\matx}[1]{\mathcal{\mathcal{#1}}} 
\newcommand{\blockmatx}[1]{\boldsymbol{\mathcal{#1}}}
\newcommand{\blockvec}[1]{\overset{\leftrightarrow}{#1}}
\newcommand{\blockveccon}[1]{\overset{\leftrightarrow}{\tilde{#1}}}

\newcommand{\interpolation}[2]{\mathbb{I}^{#1}\!#2}
\newcommand{\Dprojection}[2]{\vv{\mathbb{D}}^{#1}\!#2}

\newcommand{\rhoM}{\rho_{-}}

\newcommand{\vM}{v_{-}}

\newcommand{\betaM}{\beta_{-}}
\newcommand{\betaP}{\beta_{+}}

\newcommand{\voneM}{(v_1)_{-}}
\newcommand{\voneP}{(v_1)_{+}}
\newcommand{\vtwoM}{(v_2)_{-}}

\newcommand{\vthrM}{(v_3)_{-}}

\newcommand{\BoneM}{(B_1)_{-}}
\newcommand{\BoneP}{(B_1)_{+}}
\newcommand{\BtwoM}{(B_2)_{-}}
\newcommand{\BtwoP}{(B_2)_{+}}
\newcommand{\BthrM}{(B_3)_{-}}
\newcommand{\BthrP}{(B_3)_{+}}

\theoremstyle{theorem}
\newtheorem{thm}{Theorem}

\theoremstyle{remark}

\usepackage{mathtools}

\numberwithin{equation}{section}

\allowdisplaybreaks

\begin{document}

\title{Entropy Stable Space-Time Discontinuous Galerkin Schemes with Summation-by-Parts Property for Hyperbolic Conservation Laws}

\titlerunning{Entropy Stable Space-Time DG Schemes with SBP Property for Non-Linear Hyperbolic PDEs}
\author{Lucas Friedrich \and Gero Schn\"{u}cke \and Andrew R. Winters \and David C.~Del Rey Fern\'{a}ndez \and Gregor J. Gassner \and Mark H. Carpenter}
\institute{Lucas Friedrich (\email{lfriedri@math.uni-koeln.de})  \and Andrew R. Winters \and Gregor J. Gassner \at Mathematical Institute, University of Cologne, Cologne, Germany \\ David C.~Del Rey Fern\'{a}ndez \at National Institute of Aerospace and Computational AeroSciences Branch, NASA Langley Research Center, Hampton, VA, USA \\ Mark H. Carpenter \at Computational AeroSciences Branch, NASA Langley Research Center, Hampton, VA, USA}

\date{Received: date / Accepted: date}

\maketitle

\begin{abstract}
This work examines the development of an entropy conservative (for smooth solutions) or entropy stable (for discontinuous solutions) space-time discontinuous Galerkin (DG) method for systems of non-linear hyperbolic conservation laws. The resulting numerical scheme is fully discrete and provides a bound on the mathematical entropy at any time according to its initial condition and boundary conditions. The crux of the method is that discrete derivative approximations in space and time are summation-by-parts (SBP) operators. This allows the discrete method to mimic results from the continuous entropy analysis and ensures that the complete numerical scheme obeys the second law of thermodynamics. Importantly, the novel method described herein \textit{does not} assume any exactness of quadrature in the variational forms that naturally arise in the context of DG methods. Typically, the development of entropy stable schemes is done on the semi-discrete level ignoring the temporal dependence. In this work we demonstrate that creating an entropy stable DG method in time is similar to the spatial discrete entropy analysis, but there are important (and subtle) differences. Therefore, we highlight the temporal entropy analysis throughout this work. 
For the compressible Euler equations, the preservation of kinetic energy is of interest besides entropy stability. The construction of kinetic energy preserving (KEP) schemes is, again, typically done on the semi-discrete level similar to the construction of entropy stable schemes. We present a generalization of the KEP condition from Jameson to the space-time framework and provide the temporal components for both entropy stability and kinetic energy preservation. The properties of the space-time DG method derived herein is validated through numerical tests for the compressible Euler equations. Additionally, we provide, in appendices, how to construct the temporal entropy stable components for the shallow water or ideal magnetohydrodynamic (MHD) equations.
\end{abstract}


\section{Introduction}

In this work we focus on the numerical approximation of non-linear systems of hyperbolic conservation laws
\begin{equation}\label{eq:consLaw}
\pderivative{\vec{u}}{t} + \sum_{i=1}^3\pderivative{\vec{f}_i}{x_i} = \vec{0},
\end{equation}
where $\vec{u}$ is the vector of conserved variables, $\vec{f}_i,\,i=1,2,3$ are the physical flux vectors in each direction, and $(x_1,x_2,x_3) = (x,y,z)$ are the physical coordinates. State vectors, e.g. \vec{u}, are of size $p$ depending on the number of equations in the system under consideration. The conservation law is subject to appropriate initial and boundary conditions. It is well known that additional conserved quantities exist that are not explicitly built into the hyperbolic system of partial differential equations (PDEs) \eqref{eq:consLaw}. One such quantity is the mathematical \textit{entropy}. For gas dynamics a possible mathematical entropy is the scaled negative thermodynamic entropy which shows that the mathematical model correctly captures the second law of thermodynamics. For smooth solutions of \eqref{eq:consLaw} the entropy is conserved and decays for discontinuous solutions, e.g. \cite{harten1983,tadmor2003}.

On the mathematical level, we define a strongly convex entropy function $s=s(\vec{u})$. We then define a new set of variables
\begin{equation}\label{eq:etnVar}
\vec{w} := \pderivative{s}{\vec{u}},
\end{equation}
that provide a one-to-one mapping between conservative space and entropy space \cite{tadmor1984}. If we contract \eqref{eq:consLaw} from the left with the entropy variables we find that
\begin{equation}
\vec{w}^T\left(\pderivative{\vec{u}}{t} + \sum_{i=1}^3\pderivative{\vec{f}_i}{x_i}\right) = 0.
\end{equation}
From the definition of the entropy variables we know that
\begin{equation}\label{eq:timeChain}
\vec{w}^T\pderivative{\vec{u}}{t} = \pderivative{s}{t}.
\end{equation}
Next, an appropriate entropy flux is found from compatibility conditions between the entropy variables, the physical flux Jacobians, and the entropy flux Jacobians \cite{Barth1999,tadmor1984,tadmor1987} such that
\begin{equation}\label{EntropyVariables}
\vec{w}^T\left(\sum_{i=1}^3\pderivative{\vec{f}_i}{x_i}\right) = \left(\pderivative{s}{\vec{u}}\right)^T\sum_{i=1}^3\pderivative{\vec{f}_i}{\vec{u}}\pderivative{\vec{u}}{x_i}= \sum_{i=1}^3\left(\pderivative{f^s_i}{\vec{u}}\right)^T\pderivative{\vec{u}}{x_i} = \sum_{i=1}^3\pderivative{{f}^s_i}{x_i}.
\end{equation}
This generates appropriate entropy/entropy-flux pairs $(s\,,\,{f}^{s}_i),\,i=1,2,3$, e.g. \cite{harten1983,mock1980,tadmor1984} and gives the entropy conservation law
\begin{equation}
\pderivative{s}{t} +  \sum_{i=1}^3\pderivative{{f}^s_i}{x_i} = 0,
\end{equation}
for smooth solutions that becomes an inequality for discontinuous solutions. For scalar hyperbolic PDEs, the square entropy $s(u) = u^2/2$ is a possible choice for an entropy function and directly leads to $L^2$ stability, e.g. \cite{gassner_skew_burgers}. For gas dynamics, under appropriate physical assumptions like the positivity of density, the continuous entropy analysis provides insightful $L^2$ stability estimates on \eqref{eq:consLaw}, e.g. \cite{dutt1988,harten1983,tadmor1987}.

We see that a key part of the manipulations to contract into entropy space is the \textit{chain rule} \eqref{EntropyVariables}. Unfortunately, when we move to the discrete analysis the chain rule is lost and special care must be taken to recover it. To demonstrate the idea we first apply the product rule, on the continuous level, and rewrite the condition \eqref{EntropyVariables} to be
\begin{equation}\label{eq:prodRule}
\vec{w}^T\pderivative{\vec{f}}{x} = -\left(\pderivative{\vec{w}}{x}\right)^T\!\!\vec{f} + \pderivative{\left(\vec{w}^T\vec{f}\right)}{x} = \pderivative{f^s}{x}\implies \left(\pderivative{\vec{w}}{x}\right)^T\!\!\vec{f} = \pderivative{\left(\vec{w}^T\vec{f}\right)}{x} - \pderivative{f^s}{x}.
\end{equation}
Next, we consider a one-dimensional finite volume discretization of the form
\begin{equation}\label{eq:FVUpdate}
\pderivative{\vec{u}_k}{t} + \frac{1}{\Delta x}\left[\vec{f}^{n,*}_{k+\tfrac{1}{2}}-\vec{f}^{n,*}_{k-\tfrac{1}{2}}\right]=\vec{0}.
\end{equation}
Here, $\vec{u}_k$ is the mean value of the solution on a given cell $k$ (denoted with an integer index). The finite volume scheme also introduces numerical fluxes $\vec{f}^{*}_{k\pm{1}/{2}}$ at the cell interfaces (denoted with half-integer indices). For complete details on finite volume schemes see, e.g., LeVeque \cite{leveque2002}. In the pioneering work of Tadmor \cite{tadmor1987} it is precisely in the design of the numerical flux function where we can recover the chain rule as well as the entropic properties discretely. In the finite volume context a discrete version of the condition \eqref{eq:prodRule}, by replacing the derivatives with finite differences, is
\begin{equation}\label{eq:TadmorCondition}
\frac{\left(\vec{w}_{k+1} - \vec{w}_k\right)^T}{\Delta x}\vec{f}^{*}_{k+\tfrac{1}{2}} = \frac{\left(\left(\vec{w}_{k+1}\right)^T\vec{f}_{k+1} - \left(\vec{w}_{k}\right)^T\vec{f}_{k} \right)}{\Delta x} - \frac{f^s_{k+1}-f^s_{k}}{\Delta x}.
\end{equation}
If we multiply through by $\Delta x$ then \eqref{eq:TadmorCondition} becomes
\begin{equation}\label{eq:TadmorCondition!}
\left(\vec{w}_{k+1} - \vec{w}_k\right)^T\vec{f}^{*}_{k+\tfrac{1}{2}} = \left(\left(\vec{w}_{k+1}\right)^T\vec{f}_{k+1} - \left(\vec{w}_{k}\right)^T\vec{f}_{k} \right) - \left(f^s_{k+1}-f^s_{k}\right),
\end{equation}
which is exactly the discrete entropy conservation condition for a numerical flux originally developed by Tadmor \cite{tadmor1987}, albeit constructed from a different prospective. Interestingly, to discretely recover the chain rule we actually examine it in terms of the discrete product rule. Over the ensuing decades many researchers extended these low-order spatial methods to high-order spatial approximation with particular finite volume reconstruction techniques, e.g. \cite{fjordholm2012,lefloch2000}.


In 2013, the work of Fisher and Carpenter \cite{fisher2013} opened a new avenue for discrete entropy analysis in the context of high-order numerical methods. They demonstrated that as long as the discrete derivative matrix satisfies the summation-by-parts (SBP) property \cite{Fernandez2014,gassner_skew_burgers,kreiss1} (a discrete analogue of integration-by-parts) the low-order entropy flux of Tadmor is extended to high spatial order \cite{fisher2013,fisher2013_2}. These seminal works of Fisher and his collaborators opened a floodgate of new research into the construction of high-order entropy stable numerical schemes on quadrilateral/hexahedral elements, e.g. \cite{bohm2018,Gassner:2016ye,wintermeyer2017}, or on triangular/tetrahedral elements, e.g. \cite{chan2018,Chen2017,crean2018}. These recent publications feature a wide variety of physical applications such as oceanography, gas dynamics, and plasma flows.

However, there are still questions in the high-order discrete entropy community. Often, the entropy analysis of a numerical method is performed on the semi-discrete level, e.g. the simple finite volume scheme \eqref{eq:FVUpdate}. That is, the temporal approximation is left separate such that the entropy conservation of a scheme is stated in the limit of a shrinking time step for an explicit time integration method, e.g. \cite{bohm2018,carpenter_esdg,chan2018,Chen2017,crean2018,Fjordholm2011}. However, we already noted in the continuous entropy analysis the importance of the chain rule \eqref{eq:timeChain}. Thus, we perform a similar low-order analysis on the temporal term. First, we apply the product rule to rewrite \eqref{eq:timeChain}
\begin{equation}
\vec{w}^T\pderivative{\vec{u}}{t}  = -\left(\pderivative{\vec{w}}{t}\right)^T\vec{u} + \pderivative{\left(\vec{w}^T\vec{u}\right)}{t} = \pderivative{s}{t} \implies\left(\pderivative{\vec{w}}{t}\right)^T\vec{u} = \pderivative{\left(\vec{w}^T\vec{u}\right)}{t} - \pderivative{s}{t}.
\end{equation}
If we discretize the temporal part of \eqref{eq:FVUpdate} with finite differences and multiply through by $\Delta t$ the chain rule condition at a given temporal interface becomes
\begin{equation}\label{eq:timeConditionDiscrete}
\left(\vec{w}^{n+1}_{k} - \vec{w}^n_k\right)^T\left(\vec{u}^{*}\right)^{n+\tfrac{1}{2}}_{k} = \left(\left(\vec{w}^{n+1}_{k}\right)^T\vec{u}^{n+1}_{k} - \left(\vec{w}^n_{k}\right)^T\vec{u}^n_{k} \right) - \left(s^{n+1}_{k}-s^{n}_{k}\right),
\end{equation}
where the index $n$ denotes at which time level the variable lies. Notice, this introduces the need for a numerical state function for $\left(\vec{u}^{*}\right)^{n+1/2}_{k}$ at the time interfaces. 

The design of such numerical state functions in time for a given conservation law is a one focus of the current paper. The notation to perform this discrete analysis is simplified if we introduce the entropy/entropy flux potential pairs \cite{tadmor2003}
\begin{equation}\label{eq:phiPsi}
\left(\phi\,,\,{\psi}_i\right) = \left(\vec{w}^T\vec{u} - s\,,\,\vec{w}^T\vec{f}_i - f^s_i\right),\quad i = 1,2,3. 
\end{equation}
Due to the strong convexity of the entropy function $s(\vec{u})$ the pairs \eqref{eq:phiPsi} act as the Legendre transforms of the entropy/entropy flux pairs. In this way we find that the entropy conditions for the temporal and spatial variables are
\begin{equation}
\begin{aligned}\label{eq:timeSpaceConditionDiscrete}
\left(\vec{w}^{n+1}_{k} - \vec{w}^n_k\right)^T\left(\vec{u}^{*}\right)^{n+\tfrac{1}{2}}_{k} &= \left(\phi_k^{n+1}-\phi_k^n\right),\\[0.15cm]
\left(\vec{w}^n_{k+1} - \vec{w}^n_k\right)^T\left(\vec{f}^{*}_i\right)^n_{k+\tfrac{1}{2}} &= \left(\left(\psi_i\right)_{k+1}^n - \left(\psi_i\right)_k^n\right),\quad i = 1,2,3,
\end{aligned}
\end{equation}
respectively.

Herein we consider a nodal discontinuous Galerkin method in space. We already noted that the key to \textit{discrete} entropy conservation in the spatial components is to design DG operators that possess the SBP property needed to mimic the continuous analysis. Such operators are naturally obtained when using the Legendre-Gauss-Lobatto (LGL) nodes in the nodal DG approximation \cite{gassner_skew_burgers}. This enables the construction of high-order DG methods in space that are entropy conservative (for smooth solutions) or entropy stable (for discontinuous solutions) without the assumption of exact evaluation of the variational forms. Alternatively, there are high-order entropy conservative (or stable) space-time schemes available, e.g. \cite{diosady2015,fjordholm2012}, however they assume exact integration of variational forms in the space-time formulation. This is problematic because aliasing errors introduced by inexact quadratures are unavoidable (or at least cannot be avoided in practical simulations) for commonly examined conservation laws like the Euler equations, e.g. \cite{Gassner:2016ye,moura2017}. So, the design of numerical methods that are entropy conservative/stable in a full space-time domain is an important step in the development of thermodynamically aware numerical methods.

The main focus of this work is to apply a similar nodal DG ansatz with the SBP property to the temporal approximation and ensure that the fully discrete space-time discontinuous Galerkin spectral element method (DGSEM) remains entropy stable. The discrete entropy analysis for the spatial components has been studied by many authors and is well understood as previously discussed. In the temporal analysis we will derive appropriate numerical state functions for the vector of conservative variables $\vec{u}$ from the condition in \eqref{eq:timeConditionDiscrete}. Several authors have used SBP operators to construct temporal derivatives that lead to energy stability for linear problems \cite{boom2015,lundquist2014,nordstrom2013}. We generalize this SBP stability analysis to non-linear problems. Additionally, this work presents, for the first time, a fully discrete entropy analysis to approximate solutions of a general \textit{non-linear} conservation laws \eqref{eq:consLaw} with \textit{inexact numerical integration}. In particular, this work focuses on the temporal component of a space-time DG scheme, but we note that the proofs presented herein \textit{carry over} to any diagonal norm SBP method, such as those found in the finite difference community \cite{Fernandez2014,fisher2013}.

The space-time discontinuous Galerkin (DG) method that is the focus of this work is built from a variational formulation. Thus, we seek an integral statement of the continuous entropy analysis, such that we can clearly \textit{mimic} the continuous steps in the discrete entropy analysis. As such, we consider a spatial domain $\Omega\subset\mathbb{R}^3$ and time interval $[0,T]\subset\mathbb{R}^{+}$. Next, we integrate over the space-time domain to obtain
\begin{equation}
\int\limits_{\Omega}\int\limits_0^T \pderivative{s}{t}\dt\dV + \int\limits_0^T\int\limits_{\Omega}\sum_{i=1}^3\pderivative{{f}^s_i}{x_i}\dV\dt = 0.
\end{equation}
Next, we apply the fundamental theorem of calculus on the temporal term and apply the divergence theorem to the spatial terms
\begin{equation}\label{eq:ContinuousEntropyEquation}
\int\limits_{\Omega}\left(s(x,y,z,T) - s(x,y,z,0)\right)\dV + \int\limits_0^T\int\limits_{\partial\Omega}\sum_{i=1}^3f_i^sn_i\dS\dt = 0,
\end{equation}
where $n_i,\,i=1,2,3$ is the appropriate normal direction at the physical boundary. Rearranging terms and allowing for possibly discontinuous solutions we obtain the integral statement of the continuous entropy inequality 
\begin{equation}\label{eq:ourEndGoal}
\int\limits_{\Omega}s(x,y,z,T)\dV\leq \int\limits_{\Omega} s(x,y,z,0)\dV -\int\limits_0^T\int\limits_{\partial\Omega}\sum_{i=1}^3f_i^sn_i\dS\dt = 0.
\end{equation}
So, we see that the entropy at a given time $T$ is bounded by its initial value provided proper boundary conditions are considered.

The construction of discrete entropy stable schemes for the compressible Euler equations is particularly important to ensure the validity of the second law of thermodynamics. However, the discrete evolution of the entropy isn't the only auxiliary quantity of interest for these equations. In the turbulence community, the discrete kinetic energy evolution is also essential for the robustness of simulations, e.g. \cite{flad2017}. Numerical schemes are deemed kinetic energy preserving (KEP) when, ignoring boundary conditions, the discrete integral of the kinetic energy is not changed by the advective terms, but only by the pressure work \cite{jameson2008}. The development of such semi-discrete KEP schemes has been performed for low-order finite volume schemes \cite{jameson2008,pirozzoli2011} as well as in the high-order DG context \cite{Gassner:2016ye}. An additional result of the space-time DG analysis in this work is the generalization of Jameson's KEP condition for the temporal components of the approximation. In particular, we develop KEP conditions for the construction of numerical state function, $\left(\vec{u}^{*}\right)^{n+1/2}_{k}$, at a temporal interface.

The remainder of this work is organized as follows: We provide a brief introduction to the general form of the spatial DG approximation in Section \ref{sec:DG}. In particular, Sections \ref{sec:Spectral} and \ref{sec:STDG} introduce the most important operators of the scheme such as the discrete derivative matrix. The new SBP temporal analysis for non-linear systems is given in Sec. \ref{sec:Entropy}. This provides detailed derivations of the numerical state values in time and demonstrates their relation to the classical entropy analysis of Tadmor. The new entropy conservative/stable space-time DG method is developed in general. Thus, we present in Section \ref{Sec:Euler} details of the temporal analysis as well as the kinetic energy preservation properties of the approximation for the compressible Euler equations. Next, we present numerical results in Section \ref{sec:numResults} for the Euler equations to verify and validate the theoretical derivations described herein. Concluding remarks are given in Section \ref{sec:conc}. Finally, we also provide details for the temporal extensions of the shallow water and ideal MHD equations in the Appendices \ref{sec:App SW} and \ref{sec:App B}, respectively.

\subsection{Nomenclature}
The notation in this paper is motivated by the compact nomenclature presented in \cite{Gassner2017}. In particular, the following terminology are used:   
\begin{center}
\begin{longtable}{lllllll}
\toprule
$\mathbb{P}^{K}$ & & & &&& Space of polynomials of degree $\leq$ $K=M;N$   \tabularnewline[0.1cm]
$\interpolation{K}{}$ & & & &&& Polynomial Interpolation operator for $K = M;N$   \tabularnewline[0.1cm]
$\mathbb{D}^{M}$, $\Dprojection{N}$ & & & &&& Derivative projection operators defined in Section \ref{sec:STDG} \tabularnewline[0.1cm]
$t$ & & & &&&   Time coordinate \tabularnewline[0.1cm]
$\tau$ & & & &&& Time coordinate in the reference domain $\left[-1,1\right]$ \tabularnewline[0.1cm]
$\left(x,y,z\right)$ & & & &&& Physical spatial coordinates   \tabularnewline[0.1cm] 
$\left(\xi,\eta,\zeta\right)$ & & & &&& Spatial coordinates in the reference domain $\left[-1,1\right]^3$    \tabularnewline[0.1cm] 
$\vv{v}$ & & & &&&   Vector in three dimensional space   \tabularnewline[0.1cm] 
$\vv{n}=n_{1}\hat{x}+n_{2}\hat{y}+n_{3}\hat{z}$ & & & &&&  Physical space normal vector   \tabularnewline[0.1cm] 
$\hat{n}=\hat{n}^{1}\hat{\xi}+n^{2}\hat{\eta}+n^{3}\hat{\zeta}$ & & & &&&  Cartesian space normal vector   \tabularnewline[0.1cm] 
$\textbf{u}$ & & & &&&   Continuous quantity   \tabularnewline[0.1cm] 
$\textbf{U}$ & & & &&&   Polynomial approximation   \tabularnewline[0.1cm] 
$\blockvec{\mathbf{f}}$, $\blockveccon{\mathbf{f}}$ & & & &&&  Block vector of Cartesian and contravariant flux \tabularnewline[0.1cm] 
$\matx{B}$ & & & &&&  Matrix \tabularnewline[0.1cm] 
$\blockmatx{B}$ & & & &&&  Block matrix \tabularnewline[0.1cm]
\bottomrule
\end{longtable} 
\end{center}

\section{The space-time discontinuous Galerkin spectral element approximation}\label{sec:DG}
The compact \textit{block vector} nomenclature in \cite{Gassner2017} simplifies the analysis of the system \eqref{eq:consLaw} on curved hexahedral elements in three spatial dimensions. Thus, we translate the conservation law \eqref{eq:consLaw} in block vector notation. A block vector is highlighted by the double arrow    
\begin{equation}\label{BlockFlux}
\blockvec{\vec{f}}:=\left(\vec{f}\!_{1}^T,\vec{f}\!_{2}^T,\vec{f}\!_{3}^T\right)^T. 
\end{equation}
Two dot products can be defined for block vectors. The dot product of two block vectors is given by   
\begin{equation}\label{ProductBlock}
\blockvec{\vec{f}}\cdot\blockvec{\vec{g}}:=\sum_{i=1}^{3} \vec{f}_{i}^T\vec{g}_{i}. 
\end{equation}
Additionally, the dot product of a vector in the three dimensional space and a block vector is defined by    
\begin{equation}\label{ProductVecBlock}
\vv{v}\cdot\blockvec{\vec{f}}:=\sum_{i=1}^{3} v_{i}\vec{f}_{i}.
\end{equation}
The dot product \eqref{ProductBlock} is a scalar quantity and the dot product \eqref{ProductVecBlock} is a vector in a $p$ dimensional space, where the number $p$ corresponds to the number of conserved variables in the conservation law \eqref{eq:consLaw}. The spatial gradient of the conserved variables is defined by     
\begin{equation}\label{BlockGradient}
\vv{\nabla}_{x}\vec{u}:=\left(\vec{u}_{x}^T,\vec{u}_{y}^T,\vec{u}_{z}^T\right)^T. 
\end{equation}
The dot product \eqref{ProductBlock} and the spatial gradient \eqref{BlockGradient} are used to define the divergence of a block vector flux as  
\begin{equation}\label{BlockDivergence}
\vv{\nabla}_{x}\cdot\blockvec{\vec{f}}:=\sum_{i=1}^{3}\pderivative{\vec{f}\!_{i}}{x_{i}}.
\end{equation}
With these notations we state the conservation law \eqref{eq:consLaw} in the following compact form   
\begin{equation}\label{eq:consLawBlock}
\pderivative{\vec{u}}{t} + \vv{\nabla}_x\cdot\blockvec{\vec{f}} = 0.
\end{equation}

To set up the space-time spectral element approximation, we subdivide the time interval $\left[0,T\right]$ into $K_{T}$ non-overlapping intervals $I_{n}:=\left[t^{n},t^{n+1}\right]$ with cell length $\Delta t^{n}$, $n=1,\dots,K_{T}$. These time intervals can be mapped into the temporal computational domain $E=\left[-1,1\right]$ by the affine linear mapping $t=\pi\left(\tau\right)=\frac{\Delta t^{n}}{2}\left(\tau+1\right)+t^{n}$. The physical domain $\Omega$ is subdivided into $K_{S}$ non-overlapping and conforming hexahedral elements, $e_{k}$, $k=1,\dots,K_{S}$. These elements can have curved faces to accurately approximate the geometry. The temporal and spatial elements provide the space-time elements $\mathcal{E}_{n,k}:=I_{n}\times e_{k}$. Each spatial element $e_{k}$ is mapped into the spatial computational domain $E^{3}=\left[-1,1\right]^{3}$ with a mapping $\vv{x}=\vv{\chi}\left(\vv{\xi}\right)$, where $\vv{\chi} =X_{1}\hat{x}_{1}+X_{2}\hat{x}_{2}+X_{3}\hat{x}_{3}$ and the hats represent unit vectors. Likewise, the reference element space is represented by $\vv{\xi}=\xi \hat{\xi}+\eta \hat{\eta}+\zeta \hat{\zeta}$. The spatial mapping supplies the three covariant basis vectors   
\begin{equation}
\vv{a}\!_{i}:=\frac{\partial\vv{\chi}}{\partial\xi^{i}},\qquad i=1,2,3,
\end{equation}
and the (volume weighted) contravariant vectors 
\begin{equation}\label{ContravariantVectors}
\mathcal{J}\vv{a}^{i}:=\vv{a}\!_{j}\times\vv{a}\!_{k},\qquad\left(i,j,k\right)\ \text{cyclic},
\end{equation}
where the Jacobian determinate of the spatial mapping is given by 
\begin{equation}
\mathcal{J}:=\vv{a}_{i}\cdot\left(\vv{a}\!_{j}\times\vv{a}\!_{k}\right),\qquad\left(i,j,k\right)\ \text{cyclic}. 
\end{equation}
Additionally, the contravariant coordinate vectors satisfy the metric identities
\begin{equation}\label{MetricIdentities}
\sum_{i=1}^{3}\frac{\partial\left(Ja_{j}^{i}\right)}{\partial\xi^{i}}=0,\qquad j=1,2,3. 
\end{equation}

In \cite{Gassner2017}, the following block matrix has been introduced to transform the spatial gradient \eqref{BlockGradient} and the divergence \eqref{BlockDivergence}    
\begin{equation}\label{TransformationBlock}
\blockmatx{M}=\left[\begin{matrix}\mathcal{J}a_{1}^{1}\matx{I_{\textit{p}}} & \mathcal{J}a_{1}^{2}\matx{I_{\textit{p}}} & \mathcal{J}a_{1}^{3}\matx{I_{\textit{p}}}\\[0.1cm]
\mathcal{J}a_{2}^{1}\matx{I_{\textit{p}}} & \mathcal{J}a_{2}^{2}\matx{I_{\textit{p}}} & \mathcal{J}a_{2}^{3}\matx{I_{\textit{p}}}\\[0.1cm]
\mathcal{J}a_{3}^{1}\matx{I_{\textit{p}}} & \mathcal{J}a_{3}^{2}\matx{I_{\textit{p}}} & \mathcal{J}a_{3}^{3}\matx{I_{\textit{p}}}\\[0.1cm]
\end{matrix}\right],
\end{equation}
where the matrix $\matx{I_{\textit{p}}}$ is the $p\times p$ identity matrix. The transformation of the gradient becomes     
\begin{equation}\label{BlockGradientReference}
\vv{\nabla}_{x}\mathbf{u}=\frac{1}{\mathcal{J}}\blockmatx{M}\vv{\nabla}_{\xi}\mathbf{u}
\end{equation}  
by the block matrix \eqref{TransformationBlock}. Moreover, by taking advantage of \eqref{MetricIdentities}, the transformation of the divergence can be written as 
\begin{equation}\label{BlockDivergenceReference}
\vv{\nabla}_{x}\cdot\blockvec{\vec{f}}=\frac{1}{\mathcal{J}}\vv{\nabla}_{\xi}\cdot\blockmatx{M}^{T}\blockvec{\vec{f}}.
\end{equation}  
Hence, the contravariant block vector flux is given by 
\begin{equation}\label{ContravariantBlockVector}
\blockveccon{\vec{f}}= \blockmatx{M}^{T}\blockvec{\vec{f}}.
\end{equation} 
Finally, the chain rule formula provides the identity  
\begin{equation}
\pderivative{\vec{u}}{\tau}=\frac{\Delta t}{2}\pderivative{\vec{u}}{t}
\qquad \text{or} \qquad 
\pderivative{\vec{u}}{t}=\frac{2}{\Delta t}\pderivative{\vec{u}}{\tau}.
\end{equation}

Thus, for each space-time element $\mathcal{E}_{n,k}$ the system \eqref{eq:consLawBlock} transformed into the conservation law  
\begin{equation}\label{eq:consLawRD}
\mathcal{J}\pderivative{\vec{u}}{\tau}+\frac{\Delta t}{2}\vv{\nabla}_{\xi}\cdot\blockveccon{\vec{f}}=0.
\end{equation}
In the following sections, a space-time approximation for the system \eqref{eq:consLawRD} is derived.

\subsection{Modules for the spectral element approximation}\label{sec:Spectral}
In the space-time discontinuous Galerkin spectral element approximation, the solution and fluxes of the system \eqref{eq:consLawRD} are approximated by tensor product Lagrange polynomials of degree $N$ in space and Lagrange polynomials of degree $M$ in time, e.g.,
\begin{equation}\label{SpatialPolynomial}
\mathbf{u}\left(\tau,\cdot\right)\approx\mathbf{U}\left(\tau,\cdot\right)\in\mathbb{P}^{N}\!\!\left(E^{3}\right),\qquad\text{for fixed }\tau\in E,
\end{equation}
and additionally
\begin{equation}\label{TemporalPolynomial}
\mathbf{u}\left(\cdot,\vv{\xi}\right)\approx\mathbf{U}\left(\cdot,\vv{\xi}\right)\in\mathbb{P}^{M}\!\!\left(E\right),\qquad\text{for fixed }\vv{\xi}\in E^{3}.
\end{equation}
The one dimensional Lagrange polynomial basis for the temporal approximation is computed at $M+1$ Legendre-Gauss-Lobatto (LGL) points. The three dimensional tensor product basis for the spatial approximation is computed from an one dimensional Lagrange  polynomial basis with nodes at $N+1$ LGL points. The nodal values of the space-time approximation are represented by $\mathbf{u}\left(\tau_{\sigma},\xi_{i},\eta_{j},\zeta_{k}\right)\approx\mathbf{U}_{\sigma ijk}$, $\sigma=0,\dots,M$ and $i,j,k=0,\dots,N$. Moreover, we introduce for all $\tau,\xi,\eta,\zeta\in E$ the space-time interpolation operator   
\begin{equation}\label{InterpolationSpaceTime}
\interpolation{N\times M}{\left(\mathbf{F}\right)}\left(\tau,\xi,\eta,\zeta\right)
=\sum_{\sigma=0}^{M}\ell_{\sigma}\left(\tau\right)\sum_{i,j,k=0}^{N}\ell_{i}\left(\xi\right)\ell_{j}\left(\eta\right)\ell_{k}\left(\zeta\right)\mathbf{F}_{\sigma ijk}.
\end{equation} 
Derivatives are approximated by exact differentiation of the polynomial interpolants. It should be noted that differentiation and interpolation do not commute \cite{CHQZ:2006,Kopriva:2009nx}. In general it is $ \left(\interpolation{K}{\left(g\right)}\right)'\neq \interpolation{K-1}{\left(g'\right)}$, $K=M;N$. The geometry and metric terms are not time dependent. Hence these quantities are approximated using spatial polynomial basis, e.g.,   
\begin{equation}\label{ApproximationMetricterm}
\mathcal{J}\approx J\in\mathbb{P}^{N}\!\!\left(E^{3}\right).
\end{equation}    
However, the contravariant coordinate vectors need to be discretized in such a way that the metric identities \eqref{MetricIdentities} are satisfied on the discrete level, too. Kopriva \cite{kopriva2006metric} introduced the following discretization for the contravariant basis vectors    
\begin{equation}\label{DiscreteContravariantVectors}
Ja_{\alpha}^{i}=-\hat{x}_{i}\cdot\nabla_{\xi}\times\left(\interpolation{N}{\left(X_{\gamma}\nabla_{\xi}X_{\beta}}\right)\right),\quad i=1,2,3,\quad \alpha=1,2,3,\quad\left(\alpha,\beta,\gamma\right) \text{ cyclic}. 
\end{equation}
This discretization ensures the discrete metric identities  
\begin{equation}\label{DiscreteMetricIdentities}
\sum_{i=1}^{3}\frac{\partial\interpolation{N}{\left(Ja_{j}^{i}\right)}}{\partial\xi^{i}}=0,\qquad j=1,2,3.
\end{equation}
Thus, in order to guarantee that the spectral element approximation discretization respects the free stream preservation property, we need to discretize the contravariant basis vectors in the block matrix \eqref{TransformationBlock} by \eqref{DiscreteContravariantVectors}.      

Temporal integrals are approximated by a $2M-1$ accurate LGL quadrature formula and a tensor product extension of a $2N-1$ accurate LGL quadrature formula is used to approximate the spatial integrals. Hence, interpolation and quadrature nodes are \textit{collocated}. In one spatial dimension the LGL quadrature formula is given by   
\begin{equation}\label{GLQ}
\int\limits_{-1}^{1}g\left(\xi\right)\approx\sum_{i=0}^{K}\omega_{i}g\left(\xi_{i}\right)=\sum_{i=0}^{K}\omega_{i}g_{i},\qquad K=M;N,  
\end{equation}
where $\omega_{i}$, $i=0,\dots,K$, are the quadrature weights and $\xi_{i}$, $i=0,\dots,K$, are the LGL quadrature points. The formula \eqref{GLQ} motivates the definition of the inner product notation 
\begin{equation}\label{SpaceTimeInnerproduct}
\left\langle \mathbf{F},\mathbf{G}\right\rangle_{\!N\times M}
:=\sum_{i,j,k=0}^{N}\omega_{i}\omega_{j}\omega_{k}
\left(\sum_{\sigma=0}^{M}\omega_{\sigma}\mathbf{F}_{\sigma ijk}\mathbf{G}_{\sigma ijk}\right)
=\sum_{i,j,k=0}^{N}\sum_{\sigma=0}^{M}\omega_{\sigma}\omega_{ijk}\mathbf{F}_{\sigma ijk}\mathbf{G}_{\sigma ijk},
\end{equation}
for two functions $\mathbf{F}$ and $\mathbf{G}$. We note that the inner product \eqref{SpaceTimeInnerproduct} satisfies        
\begin{equation}
\left\langle \interpolation{N\times M}{\left(\mathbf{F}\right)},\boldsymbol{\varphi}\right\rangle_{\!N\times M}
=\left\langle \mathbf{F},\boldsymbol{\varphi}\right\rangle_{\!N\times M},
\end{equation}
for all test functions $\boldsymbol{\varphi}$ which are tensorial polynomials up to degree $N$ in space and polynomials up to degree $M$ in time. In addition, we introduce the notation    
\begin{equation}
\left.\left\langle \mathbf{F},\mathbf{G}\right\rangle _{\!N}\right|_{-1}^{\,1}=
\left.\left\langle \mathbf{F},\mathbf{G}\right\rangle _{\!N}\right|_{\tau=1}- 
\left.\left\langle \mathbf{F},\mathbf{G}\right\rangle _{\!N}\right|_{\tau=-1}
\end{equation}
with 
\begin{equation}
\left.\left\langle \mathbf{F},\mathbf{G}\right\rangle _{\!N}\right|_{\tau=1}:=\sum_{i,j,k=0}^{N}\omega_{ijk}\mathbf{F}_{Mijk}\mathbf{G}_{Mijk},
\qquad\left.\left\langle \mathbf{F},\mathbf{G}\right\rangle _{\!N}\right|_{\tau=-1}:=\sum_{i,j,k=0}^{N}\omega_{ijk}\mathbf{F}_{0ijk}\mathbf{G}_{0ijk}
\end{equation}
to approximate the spatial integrals at temporal interfaces. A further discrete quantity for the space-time discretization is defined by   
\begin{align}\label{DiscreteSpatialSurface}
\begin{split}
\int\limits_{\partial E^{3},N} 
\left\langle \left\{\blockvec{\vec{F}}\cdot \hat{n}\right\},1\right\rangle_{M}\dS
:=& \quad \sum_{\sigma=0}^{M} \omega_{\sigma}
\bigg[ 
\sum_{j,k=0}^{N}\omega_{jk}
\left(\left(\mathbf{F}_{1}\right)_{\sigma N jk}-\left(\mathbf{F}_{1}\right)_{\sigma 0 jk} \right)  \\
& \qquad \qquad  +\sum_{i,k=0}^{N}\omega_{ik}\left(\left(\mathbf{F}_{2}\right)_{\sigma iNk}-\left(\mathbf{F}_{2}\right)_{\sigma i0k} \right)  \\
& \qquad \qquad +\sum_{i,j=0}^{N}\omega_{ij}\left(\left(\mathbf{F}_{3}\right)_{\sigma ijN}-\left(\mathbf{F}_{3}\right)_{\sigma ij0} \right)\bigg], 
\end{split}
\end{align} 
where $\hat{n}$ is the reference space unit outward normal at the faces of $E^{3}$.

The spectral element approximation with LGL points for interpolation and quadrature  provides a summation-by-parts (SBP) operator  $\matx{Q}=\matx{M}\,\matx{D}$ with the mass matrix $\matx{M}$ and the derivative matrix $\matx{D}$ \cite{gassner_skew_burgers} where the LGL quadratures are $2K-1$, $K=M;N$, accurate and two of the LGL points match with the cell boundary points of the reference cell $\left[-1,1\right]$. The the mass matrix and the derivative matrix are given by  
\begin{equation}\label{SBPcoefficient}
\mathcal{M}_{ij}=\omega_{i}\delta_{ij}, \qquad 
\mathcal{D}_{ij}=\ell_{j}'\left(\xi_{i}\right) \qquad 
i,j=0,\dots,K=M;N,
\end{equation}
where $\delta_{ij}$ is the Kronecker symbol, $\left\{ \ell_{j}\right\} _{j=0}^{K=M;N}$ is the Lagrange basis and $\left\{ \xi_{i}\right\} _{i=0}^{K=M;N}$ are the LGL points. The important characteristic of this SBP operator is the property      
\begin{equation}\label{SBP}
\matx{Q}+\matx{Q}^T=\matx{B},  
\end{equation}  
where $\matx{B}=\text{diag}\left(-1,0,\dots,0,1\right)$, e.g. \cite{Fernandez2014,kreiss1}. The SBP property \eqref{SBP} is the essential ingredient to prove the entropy stability of our numerical method.            

Finally, it is worth noting that the spatial ansatz \eqref{SpatialPolynomial} \textit{without} the temporal terms \eqref{TemporalPolynomial} leads to a semi-discrete high-order nodal discontinuous Galerkin spectral element method (DGSEM). There are many recent publications which detail the construction of an entropy stable DGSEM for conservation laws in the semi-discrete framework, e.g., \cite{bohm2018,carpenter_esdg,chan2018,Chen2017,crean2018,Gassner:2016ye}.

\subsection{The space-time discontinuous Galerkin method}\label{sec:STDG}
Now, we apply the notation introduced in Section \ref{sec:Spectral} and construct a space-time DGSEM. First, we replace the solution, fluxes and spatial Jacobian 
$\mathcal{J}$ in the system \eqref{eq:consLawRD} by polynomial approximations and multiply the system with test functions $\boldsymbol{\varphi}$. The test functions are tensorial polynomials up to degree $N$ in space and polynomials up to degree $M$ in time. Next, we integrate the resulting system over a space-time element $\mathcal{E}_{n,k}$ and use integration by parts to separate boundary and volume contributions in time and space. The integrals in the variational form are approximated with high-order Legendre-Gauss-Lobatto (LGL) quadrature. Then, we insert numerical surface states $\mathbf{U}^{*}$ at the temporal element interfaces. Likewise, we insert numerical surface fluxes $\blockvec{\tilde{\mathbf{F}}}{}^{*}$ at the spatial element interfaces. Afterwards, we use the SBP property \eqref{SBP} for the temporal and spatial volume contribution to get the space-time DGSEM in strong form. Finally, since we are interested in the construction of an entropy stable method, we apply the same approach as for the semi-discrete DGSEM \cite{bohm2018,carpenter_esdg,chan2018,Chen2017,crean2018,Gassner:2016ye} and replace the interpolation operators in the temporal and spatial volume contribution by derivative projection operators. The derivative projection operators are defined by numerical volume states/fluxes denoted with a ``\#'' symbol. 

The temporal derivative projection operator is defined by    
\begin{equation}\label{TemporalDerivativeProjectionOperator}
\mathbb{D}^{M}\mathbf{U}_{\sigma ijk}^{\text{EC}}:=2\sum_{\theta=0}^{M}\mathcal{D}_{\sigma\theta}\mathbf{U}^{\#}\left(\mathbf{U}_{\sigma ijk},\mathbf{U}_{\theta ijk}\right),
\end{equation}
where the state $\mathbf{U}^{\#}$ is consistent, symmetric and satisfies the discrete entropy conservation condition \eqref{eq:timeSpaceConditionDiscrete} in the nodal values such that   
\begin{equation}\label{DiscreteEntropyConservation}
\jump{\mathbf{W}}_{\left(\sigma,\theta\right)ijk}^{T}
\mathbf{U}^{\#}\left(\mathbf{U}_{\sigma ijk},\mathbf{U}_{\theta ijk}\right)=
\jump{\varPhi}_{\left(\sigma,\theta\right)ijk},
\end{equation}
for $\sigma,\theta=0,\dots,M$ and $i,j,k=0,\dots,N$, where 
\begin{equation}\label{EntropyFunctional1}
\varPhi:=\mathbf{W}^T\mathbf{U}-S, 
\end{equation}  
and the volume jumps in \eqref{DiscreteEntropyConservation} are defined by  
\begin{equation}\label{TemporalVolumeJumps}
\jump{\cdot}_{\left(\sigma,\theta\right)ijk}:=\left(\cdot\right)_{\sigma ijk}-\left(\cdot\right)_{\theta ijk}.
\end{equation}   
The quantity $\mathbf{W}$ represents the polynomial approximation of the entropy variables \eqref{EntropyVariables}. 

For now, we keep the analysis general. In Section \ref{sec:EulerStuff}, we construct a temporal volume state function for the compressible Euler equations  
and in Appendices \ref{sec:App SW} and \ref{sec:App B}, we construct temporal volume state functions for the shallow water and ideal MHD equations which respect the condition \eqref{DiscreteEntropyConservation}.        

The spatial derivative projection operator is more complex, since the discretization of the metric terms must be taken into account. A suitable derivative operator was introduced in \cite{Gassner2017} and is given by 
\begin{align}\label{SpatialDerivativeProjectionOperator}
\begin{split}
\Dprojection{N}\cdot\blockvec{\tilde{\mathbf{F}}}{}_{\sigma ijk}^{\text{EC}}
:=2\sum_{m=0}^{N}\quad\,& 
\mathcal{D}_{im}\left(\blockvec{\mathbf{F}}{}^{\text{EC}}\left(\mathbf{U}_{\sigma ijk},\mathbf{U}_{\sigma mjk}\right)\cdot\avg{J\vv{a}^{1}}_{\left(i,m\right)jk}\right) \\
\quad +  & \,\mathcal{D}_{jm}\left(\blockvec{\mathbf{F}}{}^{\text{EC}}\left(\mathbf{U}_{\sigma ijk},\mathbf{U}_{\sigma imk}\right)\cdot\avg{J\vv{a}^{2}}_{ i\left(j,m\right)k}\right) \\
\quad  +& \, \mathcal{D}_{km}\left(\blockvec{\mathbf{F}}{}^{\text{EC}}\left(\mathbf{U}_{\sigma ijk},\mathbf{U}_{\sigma ijm}\right)\cdot\avg{J\vv{a}^{3}}_{ ij\left(k,m\right)}\right), 
\end{split}
\end{align}
with the volume averages of the metric terms, e.g. \label{SpatialVolumeJumps}  
\begin{equation}\label{SpatialVolumeAverages}
\avg{\cdot}_{(i,m)jk}:=\frac{1}{2}\left[\left(\cdot\right)_{ ijk}+\left(\cdot\right)_{ mjk}\right],
\end{equation}
where the average is only done in the spatial directions as the metric terms are constant in time. The flux $\blockvec{\mathbf{F}}{}^{\text{EC}}=\left(\mathbf{F}_{1}^{\text{EC}},\mathbf{F}_{2}^{\text{EC}},\mathbf{F}_{3}^{\text{EC}}\right)^T$ is consistent with $\blockvec{\mathbf{F}}$ and symmetric such that, e.g.,
\begin{equation}\label{FluxSymmetric}
\blockvec{\mathbf{F}}{}^{\text{EC}}\left(\mathbf{U}_{\sigma ijk},\mathbf{U}_{\sigma mjk}\right)=\blockvec{\mathbf{F}}{}^{\text{EC}}\left(\mathbf{U}_{\sigma mjk},\mathbf{U}_{\sigma ijk}\right),
\end{equation}       
for $\sigma=0,\dots,M$ and $i,j,k,m=0,\dots,N$. Furthermore, as in the work of Gassner et al. \cite{Gassner2017}, the flux functions $\mathbf{F}_{l}^{\text{EC}}$, $l=1,2,3$, satisfy Tadmor's discrete entropy condition 
\begin{align}\label{EntropyFunctional2}
\begin{split}
\jump{\mathbf{W}}_{\sigma(i,m)jk}^{T}\mathbf{F}_{l}^{\text{EC}}\left(\mathbf{U}_{\sigma ijk},\mathbf{U}_{\sigma mjk}\right)
=&\jump{\Psi_{l}}_{\sigma(i,m)jk}, \\
\jump{\mathbf{W}}_{\sigma i(j,m)k}^{T}\mathbf{F}_{l}^{\text{EC}}\left(\mathbf{U}_{\sigma ijk},\mathbf{U}_{\sigma imk}\right)
=&\jump{\Psi_{l}}_{\sigma i(j,m)k}, \\
\jump{\mathbf{W}}_{\sigma ij(k,m)}^{T}\mathbf{F}_{l}^{\text{EC}}\left(\mathbf{U}_{\sigma ijk},\mathbf{U}_{\sigma ijm}\right)
=&\jump{\Psi_{l}}_{\sigma ij(k,m)},
\end{split}
\end{align}
where 
\begin{equation}\label{EntropyFunctional2a}
\Psi_{l}:=\mathbf{W}^{T}\mathbf{F}_{l}-F_{l}^{s}, \quad l=1,2,3, 
\end{equation} 
and the volume jumps, e.g. 
\begin{equation}\label{SpatialVolumeJumps}
\jump{\cdot}_{\sigma (i,m)jk}:=\left(\cdot\right)_{\sigma ijk}-\left(\cdot\right)_{\sigma mjk}.
\end{equation}
There are several available entropy conservative flux functions with these properties, e.g. \cite{Chandrashekar2012,IsmailRoe2009} for the Euler equations. In particular, if we take the test function to be the interpolant of the entropy variables, $\boldsymbol\varphi = \vec{W}$, we obtain from the SBP property \eqref{SBP} and the same arguments as in the semi-discrete case \cite[Appendix B.3]{Gassner2017} the identity   
\begin{equation}\label{EntropyVolumeContribution}
\left\langle \Dprojection{N}\cdot\blockvec{\tilde{\mathbf{F}}}{}^{\text{EC}},\mathbf{W}\right\rangle_{\!N\times M}=\int\limits _{\partial E^{3},N}\left\langle \tilde{F}_{\hat{n}}^{s},1\right\rangle_{\!M}\dS,
\end{equation}
where $\tilde{F}^{s}_{\hat{n}}:=\vv{\tilde{F}}^{s}\cdot\hat{n}$. That is, the volume contributions of the spatial entropy conservative DGSEM become the discrete entropy flux on the boundary.

Finally, for each space-time element $\mathcal{E}_{n,k}$ and all test functions $\boldsymbol{\varphi}$ the space-time DGSEM can be written in the following compact variational form:      
\begin{equation}\label{SpaceTimeDG1}
A_{T}\!\left(\mathbf{U},\boldsymbol{\varphi}\right)
+A_{S}\!\left(\blockvec{\mathbf{F}},\boldsymbol{\varphi}\right)=0,
\end{equation}
where the temporal part $A_{T}\left(\mathbf{U},\boldsymbol{\varphi}\right)$ and the spatial $A_{S}\!\left(\blockvec{\mathbf{F}},\boldsymbol{\varphi}\right)$ of the space-time DGSEM are defined by      
\begin{align}
A_{T}\!\left(\mathbf{U},\boldsymbol{\varphi}\right)
:=& \quad \left\langle J
\mathbb{D}^{M}\mathbf{U}^{\#},\boldsymbol{\varphi}\right\rangle_{\!N\times M} 
+\left.\left\langle J\left(\mathbf{U}^{*}-\mathbf{U}\right),\boldsymbol{\varphi}\right\rangle_{\!N}\right|_{-1}^{\,1}, \label{TemporalPart} \\
& \nonumber \\
A_{S}\!\left(\blockvec{\mathbf{F}},\boldsymbol{\varphi}\right)
:=& \quad\frac{\Delta t}{2}\left\langle \Dprojection{N}\cdot\blockvec{\tilde{\mathbf{F}}}{}^{\text{EC}},\boldsymbol{\varphi}\right\rangle_{\!N\times M} +\frac{\Delta t}{2}\int\limits_{\partial E^{3},N}\left\langle \boldsymbol{\varphi}^{T}\left(\tilde{\mathbf{F}}_{\hat{n}}^{*}-\tilde{\mathbf{F}}_{\hat{n}}\right),1\right\rangle_{\!M}\dS.  
\label{SpatialPart}  
\end{align}
The flux $\tilde{\mathbf{F}}_{\hat{n}}$ in \eqref{SpatialPart} is defined by 
\begin{equation}\label{BlockvectorNormal}
\tilde{\mathbf{F}}_{\hat{n}}=\left(\hat{s}\vv{n}\right)\cdot\blockvec{\mathbf{F}}
=\sum_{l=1}^{3}\hat{n}^{l}\left(Ja_{1}^{l}F_{1}+Ja_{2}^{l}F_{2}+Ja_{3}^{l}F_{3}\right)
=\left\{\blockmatx{M}\,\blockvec{\mathbf{F}}\right\}\cdot \hat{n}, 
\quad \hat{s}=\left|\sum_{l=1}^{3}\left(J\vv{a}^{l}\right)\hat{n}^{l}\right|.
\end{equation}
To specify the state functions in \eqref{TemporalPart} and the numerical surface flux functions \eqref{SpatialPart}, we introduce notation for states at the LGL nodes along an interface between two temporal intervals and two spatial elements to be a primary ``$-$'' and complement the notation with a secondary ``$+$'' to denote the value at the LGL nodes on the opposite side. Then the orientated jump and
the arithmetic mean at the interfaces are defined by  
\begin{equation}\label{SurfaceJumpMean}
\jump{\cdot}:=\left(\cdot\right)_{+}-\left(\cdot\right)_{-},
\quad \text{and} \quad 
\avg{\cdot}:=\frac{1}{2}\left[\left(\cdot\right)_{+}+\left(\cdot\right)_{-}\right].  
\end{equation}
When applied to vectors, the average and jump operators are evaluated separately for each vector component. Then the normal vector $\vv{n}$ is defined unique to point from the ``$-$'' to the ``$+$'' side. This notation allows for the consistent
presentation of the temporal state functions $\mathbf{U}^{*}$ and the spatial numerical surface flux functions. A temporal state function with the property \eqref{DiscreteEntropyConservation} is not a computationally tractable option, since it couples all the time slabs.  The only feasible choice as a temporal numerical state is the pure upwind state function $\mathbf{U}^{*}=\mathbf{U}_{-}$, as it is the only flux that decouples the time slaps. We assume that the upwind state satisfies     
\begin{equation}\label{TemporalStability1}
\jump{\mathbf{W}}^T\mathbf{U}^{*}\leq\jump{\varPhi}.
\end{equation} 
In Section \ref{sec:EulerStuff} (Lemma 1), we prove that the upwind state function for the compressible Euler equations satisfies this assumption. Furthermore, in the appendices \ref{sec:ShallowWater} and \ref{sec:UpwindMHD} it is proven that the upwind state functions for the shallow water equations and MHD equations satisfy \eqref{TemporalStability1}. The contravariant surface numerical fluxes are computed
from the entropy conserving Cartesian fluxes  $\mathbf{F}_{l}^{\text{EC}}$, $l=1,2,3$, as follows:  
\begin{equation}\label{ContravariantSurfaceFlux}
\tilde{\mathbf{F}}_{\hat{n}}^{*}=\hat{s}\left(n_{1}\mathbf{F}_{1}^{\text{EC}}+n_{2}\mathbf{F}_{2}^{\text{EC}}+n_{3}\mathbf{F}_{3}^{\text{EC}}\right).
\end{equation}
The definition of the numerical surface flux functions \eqref{ContravariantSurfaceFlux} produces the equality   
\begin{equation}\label{SurfaceEntropyCondition}
\sum_{n=1}^{K_{T}}\underset{\text{faces}}{\sum_{\text{Interior}}}\frac{\Delta t^{n}}{2}\int\limits_{\partial E^{3},N}\left\langle \jump{\mathbf{W}^{n,k}}^{T}\tilde{\mathbf{F}}_{\hat{n}}^{n,k,*}-\jump{\left(\mathbf{W}^{n,k}\right)^{T}\tilde{\mathbf{F}}_{\hat{n}}^{n,k}}+\jump{\tilde{F}_{\hat{n}}^{s}},1\right\rangle_{\!M}\dS = 0,
\end{equation}
by the same arguments as in \cite[Appendix B.2]{Gassner2017}. However, in order to obtain an entropy stable discretization for discontinuous solutions, an addition matrix dissipation operator needs to be added to the entropy conserving flux functions 
$\mathbf{F}_{l}^{\text{EC}}$, $l=1,2,3$. Then, the Cartesian fluxes are computed by             
\begin{equation}\label{CartesianFluxesES}
\mathbf{F}_{l}^{\text{ES}}:=\mathbf{F}_{l}^{\text{EC}}-\frac{1}{2}\matx{R}_{l}\left|\Lambda_{l}\right|\matx{T}_{l}\,\matx{R}_{l}^{T}\jump{\mathbf{W}},\qquad l=1,2,3,
\end{equation} 
the quantity $\matx{R}_{l}\left|\Lambda_{l}\right|\matx{T}\,\matx{R}_{l}^{T}$, $l=1,2,3$, represents a positive semidefinite matrix dissipation operator. The matrix $\matx{R}_{l}$ contains the averaged right eigenvectors of the flux Jacobian, the corresponding absolute values of the averaged eigenvalues are contained in the diagonal matrix $\Lambda_{l}$ and the diagonal matrix $\matx{T}_{l}$ is a scaling matrix (see \cite{Barth1999} for details).  Dissipation matrices for the Euler equations can be found in \cite[Appendix A]{Gassner2017}. Then, the contravariant surface numerical fluxes are computed by replacing the fluxes $\mathbf{F}_{l}^{\text{EC}}$ with $\mathbf{F}_{l}^{\text{ES}}$ in \eqref{SpatialPart}. When the fluxes $\mathbf{F}_{l}^{\text{ES}}$, $l=1,2,3$ are used to compute the contravariant surface numerical fluxes the identity \eqref{SurfaceEntropyCondition} becomes an inequality bounded above by zero.

\begin{remark}
For the MHD equations the spatial part \eqref{SpatialPart} must be altered because a non-conservative term is necessary to build an entropy stable method, see Bohm et al. for details \cite{bohm2018}. In Appendix \ref{sec:App B3} the correct spatial part for the MHD equations is presented. By replacing \eqref{SpatialPart} with the appropriate terms a space-time DGSEM to solve the ideal MHD equations can be constructed. 
\end{remark}

\subsection{Discrete entropy analysis}\label{sec:Entropy}
The spatial integral of the entropy is bounded in time on the continuous level. Thus, it is desirable that a numerical method is stable in the sense that a discrete version of this integral is bounded in time, too. Methods with this stability property are called entropy stable methods. 

In the context of the space-time DGSEM, we are interested to find an upper bound for the quantity 
\begin{equation}\label{DisreteEntropyIntegral1}
\bar{S}\left(\mathbf{U}\left(T\right)\right):=\sum_{k=1}^{K_{S}}\left.\left\langle s\left(\mathbf{U}^{K_{T},k}\right),J^{k}\right\rangle_{\!N}\right|_{\tau=1}.  
\end{equation}
We note that \eqref{DisreteEntropyIntegral1} is a discrete version of the spatial integral on the left of the continuous inequality \eqref{eq:ourEndGoal}. The discrete upper bound should depend on a discrete contribution from the boundary faces and the initial quantity of the entropy
\begin{equation}\label{DisreteEntropyIntegral2}
\bar{S}\left(\mathbf{u}\left(0\right)\right):=\sum_{k=1}^{K_{S}}\left.\left\langle s\left(\mathbf{u}^{1,k}\right),J^{k}\right\rangle_{\!N}\right|_{\tau=-1}, 
\end{equation}
where $\left.\mathbf{u}^{1,k}\right|_{\tau=-1}=\mathbf{u}^{k}\left(0\right)$ is the initial condition prescribed to the conservation law in the spatial cell $e_{k}$, $k=1,\dots,K_{S}$.

\begin{thm}[Entropy stability]
Consider the space-time DGSEM with Dirichlet boundary conditions in time and periodic boundary conditions in space. Assume that the temporal numerical states are upwind fluxes $\mathbf{U}^{*}=\mathbf{U}_{-}$ with the property \eqref{TemporalStability1} and assume that the spatial numerical surface fluxes $\tilde{\mathbf{F}}_{\hat{n}}^{*}$ are computed by \eqref{ContravariantSurfaceFlux}. 
Then the space-time DGSEM is entropy stable.  
\end{thm}

\begin{proof}
We choose $\boldsymbol{\varphi}=\mathbf{W}$ as test function in the equation \eqref{SpaceTimeDG1} and sum over all space-time elements. This results in the identity
\begin{equation}\label{EntropyStability1}
\sum_{n=1}^{K_{T}}\sum_{k=1}^{K_{S}}\left(A_{T}\!\left(\mathbf{U}^{n,k},\mathbf{W}^{n,k}\right)+A_{S}\!\left(\blockvec{\mathbf{F}}{}^{n,k},\mathbf{W}^{n,k}\right)\right)=0. 
\end{equation}

In Appendix \ref{sec:App C} we prove the following equation:
\begin{align}\label{TemporalStability2}
\begin{split}
\sum_{n=1}^{K_{T}}\sum_{k=1}^{K_{S}}A_{T}\!\left(\mathbf{U}^{n,k},\mathbf{W}^{n,k}\right)
=&  \quad \bar{S}\left(\mathbf{U}\left(T\right)\right)  
+  \sum_{k=1}^{K_{S}}\left.\left\langle \left(\mathbf{W}^{K_{T},k}\right)^{T}\left(\mathbf{U}^{K_{T},k,*}-\mathbf{U}^{K_{T},k}\right),J^{k}\right\rangle_{\!N}\right|_{\tau=1} \\
&  \quad \qquad \qquad -\sum_{n=2}^{K_{T}}\sum_{k=1}^{K_{S}}\left.\left\langle \jump{\mathbf{W}^{n,k}}^{T}\mathbf{U}^{n,k,*}-\jump{\varPhi^{n,k}},J^{k}\right\rangle_{\!N}\right|_{\tau=-1} \\
& -\bar{S}\left(\mathbf{U}\left(0\right)\right)-\sum_{k=1}^{K_{S}}\left.\left\langle \left(\mathbf{W}^{1,k}\right)^{T}\left(\mathbf{U}^{1,k,*}-\mathbf{U}^{1,k}\right),J^{k}\right\rangle_{\!N}\right|_{\tau=-1}, 
\end{split}
\end{align}
where the quantity $\bar{S}\left(\mathbf{U}\left(T\right)\right)$ is defined in \eqref{DisreteEntropyIntegral1} and the quantity $\bar{S}\left(\mathbf{U}\left(0\right)\right)$ is given by 
\begin{equation}\label{DisreteEntropyIntegral3}
\bar{S}\left(\mathbf{U}\left(0\right)\right):=\sum_{k=1}^{K_{S}}\left.\left\langle s\left(\mathbf{U}^{1,k}\right),J^{k}\right\rangle_{\!N}\right|_{\tau=-1}.
\end{equation}
The state $\mathbf{U}^{*}$ is an upwind flux and satisfies $\left.\mathbf{U}^{K_{T},k,*}\right|_{\tau=1}=\left.\mathbf{U}^{K_{T},k}\right|_{\tau=1}$. Thus, we obtain 
\begin{equation}\label{SummationTemporal1}
\sum_{k=1}^{K_{S}}\left.\left\langle \left(\mathbf{W}^{K_{T},k}\right)^{T}\left(\mathbf{U}^{K_{T},k,*}-\mathbf{U}^{K_{T},k}\right),J^{k}\right\rangle_{\!N}\right|_{\tau=1}=0.
\end{equation} 
Likewise, it follows from \eqref{TemporalStability1} that
\begin{equation}\label{SummationTemporal2}
-\sum_{n=2}^{K_{T}}\sum_{k=1}^{K_{S}}\left.\left\langle \jump{\mathbf{W}^{n,k}}^{T}\mathbf{U}^{n,k,*}-\jump{\varPhi^{n,k}},J^{k}\right\rangle_{\!N}\right|_{\tau=-1} \geq 0.
\end{equation}
Furthermore, we add and subtract $\bar{S}\left(\mathbf{u}\left(0\right)\right)$ and the quantity $\mathbf{W}^{T}\mathbf{U}^*$ from the outside of the first temporal element to the last line in equation \eqref{TemporalStability2}. The upwind flux is defined as $\left.\mathbf{U}^{1,k,*}\right|_{\tau=-1}=\left.\mathbf{u}^{1,k}\right|_{\tau=-1}$ in the first temporal element. Therefore, we obtain by \eqref{TemporalStability1} the inequality 
\begin{align}\label{SummationTemporal3}
\begin{split}
&  -\bar{S}\left(\mathbf{U}\left(0\right)\right)-\sum_{k=1}^{K_{S}}\left.\left\langle \left(\mathbf{W}^{1,k}\right)^{T}\left(\mathbf{U}^{1,k,*}-\mathbf{U}^{1,k}\right),J^{k}\right\rangle_{\!N}\right|_{\tau=-1}  \\
=& -\bar{S}\left(\mathbf{u}\left(0\right)\right)+\sum_{k=1}^{K_{S}}\left.\left\langle \jump{\varPhi^{1,k}}-\jump{\mathbf{W}^{1,k}}^{T}\mathbf{u}^{1,k},J^{k}\right\rangle _{\!N}\right|_{\tau=-1}\geq -\bar{S}\left(\mathbf{u}\left(0\right)\right),
\end{split}
\end{align}  
where the quantity $\bar{S}\left(\mathbf{u}\left(0\right)\right)$ is defined in \eqref{DisreteEntropyIntegral2}. 
Overall, we obtain the following inequality for the temporal part of the space-time DGSEM  
\begin{equation}\label{TemporalStability2.2}
\bar{S}\left(\mathbf{U}\left(T\right)\right)-\bar{S}\left(\mathbf{u}\left(0\right)\right)\leq\sum_{n=1}^{K_{T}}\sum_{k=1}^{K_{S}}A_{T}\!\left(\mathbf{U}^{n,k},\mathbf{W}^{n,k}\right).
\end{equation}

For the spatial part of the space-time DGSEM, we apply the identity \eqref{EntropyVolumeContribution} and obtain 
\begin{align}\label{SpatialStability1}
\begin{split}
A_{S}\!\left(\blockvec{\mathbf{F}},\mathbf{W}\right) 
=& \quad \frac{\Delta t}{2}\left\langle \Dprojection{N}\cdot\blockvec{\tilde{\mathbf{F}}}{}^{\text{EC}},\mathbf{W}\right\rangle_{\!N\times M}+\frac{\Delta t}{2}\int\limits _{\partial E^{3},N}\left\langle \mathbf{W}^{T}\left(\tilde{\mathbf{F}}_{\hat{n}}^{*}-\tilde{\mathbf{F}}_{\hat{n}}\right),1\right\rangle_{\!M}\dS \\
=& \quad \frac{\Delta t}{2}\int\limits _{\partial E^{3},N}\left\langle \mathbf{W}^{T}\tilde{\mathbf{F}}_{\hat{n}}^{*}-\mathbf{W}^{T}\tilde{\mathbf{F}}_{\hat{n}}+\tilde{F}_{\hat{n}}^{s},1\right\rangle_{\!M}\dS.
\end{split}
\end{align}
By summing the contribution \eqref{SpatialStability1} over all space-time elements, we obtain 
\begin{align}\label{SpatialStability2}
\begin{split}
& 
\quad \sum_{n=1}^{K_{T}}\sum_{k=1}^{K_{S}}A_{S}\!\left(\blockvec{\mathbf{F}}{}^{n,k},\mathbf{W}^{n,k}\right)
 \\  
 =& \quad \textbf{BC}-\sum_{n=1}^{K_{T}}\underset{\text{faces}}{\sum_{\text{Interior}}}\frac{\Delta t^{n}}{2}\int\limits_{\partial E^{3},N}\left\langle\jump{\mathbf{W}^{n,k}}^{T}\tilde{\mathbf{F}}_{\hat{n}}^{n,k,*}-\jump{\left(\mathbf{W}^{n,k}\right)^{T}\tilde{\mathbf{F}}_{\hat{n}}^{n,k}}+\jump{\tilde{F}_{\hat{n}}^{n,k,s}},1\right\rangle_{\!M}\dS,  
\end{split}
\end{align}
where the contribution from the physical boundary terms is compactly given by  
\begin{equation}
\textbf{BC}:=\sum_{n=1}^{K_{T}}\underset{\text{faces}}{\sum_{\text{Boundary}}}\frac{\Delta t^{n}}{2}\int\limits_{\partial E^{3},N}\left\langle \left(\mathbf{W}^{n,k}\right)^{T}\tilde{\mathbf{F}}_{\hat{n}}^{n,k,*}-\left(\mathbf{W}^{n,k}\right)^{T}\tilde{\mathbf{F}}_{\hat{n}}+\tilde{F}_{\hat{n}}^{n,k,s},1\right\rangle_{\!M}\dS.
\end{equation}
Next, we apply the equation \eqref{SurfaceEntropyCondition} and obtain 
\begin{equation}
\sum_{n=1}^{K_{T}}\underset{\text{faces}}{\sum_{\text{Interior}}}\frac{\Delta t^{n}}{2}\int\limits_{\partial E^{3},N}\left\langle\jump{\mathbf{W}^{n,k}}^{T}\tilde{\mathbf{F}}_{\hat{n}}^{n,k,*}-\jump{\left(\mathbf{W}^{n,k}\right)^{T}\tilde{\mathbf{F}}_{\hat{n}}^{n,k}}+\jump{\tilde{F}_{\hat{n}}^{n,k,s}},1\right\rangle_{\!M}\dS=0.
\end{equation}
Thus, it follows 
\begin{equation}\label{SpatialStability2.1}
\sum_{n=1}^{K_{T}}\sum_{k=1}^{K_{S}}A_{S}\!\left(\blockvec{\mathbf{F}}{}^{n,k},\mathbf{W}^{n,k}\right)=\textbf{BC}. 
\end{equation}

Moreover, when the contravariant surface numerical fluxes are computed with the entropy stable Cartesian fluxes \eqref{CartesianFluxesES}, the equation \eqref{SpatialStability2} provides 
\begin{equation}\label{SpatialStability2a}
\sum_{n=1}^{K_{T}}\sum_{k=1}^{K_{S}}A_{S}\!\left(\blockvec{\mathbf{F}}{}^{n,k},\mathbf{W}^{n,k}\right)
\geq \textbf{BC}, 
\end{equation}

Finally, we obtain the discrete entropy inequality        
\begin{equation}\label{eq:finalResultEC}
\bar{S}\left(\mathbf{U}\left(T\right)\right)
\leq
\bar{S}\left(\mathbf{u}\left(0\right)\right)-\textbf{BC},
\end{equation}  
by \eqref{EntropyStability1}, \eqref{TemporalStability2.2}  and \eqref{SpatialStability2.1} (or alternative \eqref{SpatialStability2a}). Here, we consider a periodic problem in the three spatial directions. So, the physical boundary terms cancel and $\textbf{BC}=0$ \cite{Gassner2017} yielding
\begin{equation}\label{eq:discreteECInq}
\bar{S}\left(\mathbf{U}\left(T\right)\right)\leq\bar{S}\left(\mathbf{u}\left(0\right)\right).
\end{equation}
\end{proof}
\begin{remark}
We directly see that the result \eqref{eq:finalResultEC} in the proof of Theorem 1 is the discrete analogue of the continuous analysis which gave \eqref{eq:ourEndGoal}.
\end{remark}

\subsection{Discrete entropy preservation}\label{sec:EntropyPreservation}
The particular choice of periodic boundary and entropy conservative spatial numerical fluxes 
provide the identity  
\begin{equation}\label{EC:Equation1}
\sum_{n=1}^{K_{T}}\sum_{k=1}^{K_{S}}A_{S}\!\left(\blockvec{\mathbf{F}},\mathbf{W}\right)=0
\end{equation}
as we can extract from the proof of Theorem 1.  
Next, we investigate the space-time DGSEM with periodic boundary conditions in time 
\begin{equation}\label{EC:Equation3}
\mathbf{U}_{+}^{K_{T},k}=\mathbf{U}_{-}^{1,k},\qquad\mathbf{U}_{+}^{1,k}=\mathbf{U}_{-}^{K_{T},k},\qquad k=1,\dots,K_{S} 
\end{equation}
and numerical state functions, denoted with a \#, that satisfy the property \eqref{DiscreteEntropyConservation}. Similar to the construction of \eqref{TemporalStability2} we derive the equality 
\begin{align}\label{TemporalStabilityEC}
\begin{split}
\sum_{n=1}^{K_{T}}\sum_{k=1}^{K_{S}}A_{T}\!\left(\mathbf{U}^{n,k},\mathbf{W}^{n,k}\right)
=&  \quad \bar{S}\left(\mathbf{U}\left(T\right)\right)  
+  \sum_{k=1}^{K_{S}}\left.\left\langle \left(\mathbf{W}^{K_{T},k}_{-}\right)^{T}\left(\mathbf{U}^{K_{T},k,\#}-\mathbf{U}^{K_{T},k}_{-}\right),J^{k}\right\rangle_{\!N}\right|_{\tau=1} \\
&  \quad \qquad \qquad -\sum_{n=2}^{K_{T}}\sum_{k=1}^{K_{S}}\left.\left\langle \jump{\mathbf{W}^{n,k}}^{T}\mathbf{U}^{n,k,\#}-\jump{\varPhi\left(\mathbf{U}^{n,k}\right)},J^{k}\right\rangle_{\!N}\right|_{\tau=-1} \\
& -\bar{S}\left(\mathbf{U}\left(0\right)\right)-\sum_{k=1}^{K_{S}}\left.\left\langle \left(\mathbf{W}^{1,k}_{+}\right)^{T}\left(\mathbf{U}^{1,k,\#}-\mathbf{U}^{1,k}_{+}\right),J^{k}\right\rangle_{\!N}\right|_{\tau=-1}.
\end{split}
\end{align}
From the property \eqref{DiscreteEntropyConservation} of the numerical state functions the middle term of \eqref{TemporalStabilityEC} vanishes, i.e., 
\begin{equation}\label{EC:Equation4}
-\sum_{n=2}^{K_{T}}\sum_{k=1}^{K_{S}}\left.\left\langle \jump{\mathbf{W}^{n,k}}^{T}\mathbf{U}^{n,k,\#}-\jump{\varPhi^{n,k}},J^{k}\right\rangle_{\!N}\right|_{\tau=-1}=0. 
\end{equation}
The remaining boundary terms in \eqref{TemporalStabilityEC} cancel by the prescription of the boundary conditions \eqref{EC:Equation3} and we find that
\begin{equation}
\sum_{n=1}^{K_{T}}\sum_{k=1}^{K_{S}}A_{T}\!\left(\mathbf{U}^{n,k},\mathbf{W}^{n,k}\right) = 0.
\end{equation}

Therefore, periodic boundary conditions in time are not an appropriate choice. Next, we consider Dirichlet boundary conditions in time and apply temporal numerical state $\mathbf{U}^{*}$ functions with the properties:
\begin{align}
\left.\mathbf{U}^{1,k,\#}\right|_{\tau=-1}=\left.\mathbf{U}_{-}^{1,k}\right|_{\tau=-1}=\left.\mathbf{u}^{1,k}\right|_{\tau=-1},
\qquad\left.\mathbf{U}^{K_{T},k,\#}\right|_{\tau=1}=\left.\mathbf{U}_{-}^{K_{T},k}\right|_{\tau=1}, \quad k=1,\dots,K_{S}, \label{EntropyPreservation1}  \\ 
\nonumber \\
\jump{\mathbf{W}^{n,k}}^{T}\mathbf{U}^{n,k,\#}=\jump{\varPhi^{n,k}}, \quad  n=2,\dots,K_{T}-1, \quad k=1,\dots,K_{S}.  \label{EntropyPreservation2}
\end{align}
Again, the middle term of \eqref{TemporalStabilityEC} vanishes because of the property \eqref{EC:Equation4} by \eqref{EntropyPreservation2}. Furthermore, we see that
\begin{equation}\label{eq:tempStabTEC}
 \sum_{k=1}^{K_{S}}\left.\left\langle \left(\mathbf{W}^{K_{T},k}_{-}\right)^{T}\left(\mathbf{U}^{K_{T},k,\#}-\mathbf{U}^{K_{T},k}_{-}\right),J^{k}\right\rangle_{\!N}\right|_{\tau=1}=0,
\end{equation}
due to the choice of the upwind state \eqref{EntropyPreservation1} and we obtain, 
similar to \eqref{SummationTemporal3}, that
\begin{align}\label{SummationTemporal4}
\begin{split}
&  -\bar{S}\left(\mathbf{U}\left(0\right)\right)-\sum_{k=1}^{K_{S}}\left.\left\langle \left(\mathbf{W}^{1,k}\right)^{T}\left(\mathbf{U}^{1,k,\#}-\mathbf{U}_{+}^{1,k}\right),J^{k}\right\rangle_{\!N}\right|_{\tau=-1}  \\
=& -\bar{S}\left(\mathbf{u}\left(0\right)\right)+\sum_{k=1}^{K_{S}}\left.\left\langle \jump{\varPhi^{1,k}}-\jump{\mathbf{W}^{1,k}}^{T}\mathbf{u}^{1,k},J^{k}\right\rangle _{\!N}\right|_{\tau=-1}.
\end{split}
\end{align} 
Thus, we obtain by \eqref{TemporalStabilityEC}, \eqref{eq:tempStabTEC} and \eqref{SummationTemporal4} the identity 
\begin{align}\label{EC:Equation5}
\begin{split}
\sum_{n=1}^{K_{T}}\sum_{k=1}^{K_{S}}A_{T}\!\left(\mathbf{U}^{n,k},\mathbf{W}^{n,k}\right)=& \quad 
\bar{S}\left(\mathbf{U}\left(T\right)\right)-\bar{S}\left(\mathbf{u}\left(0\right)\right) \\
& +\sum_{k=1}^{K_{S}}\left.\left\langle \jump{\varPhi^{1,k}}-\jump{\mathbf{W}^{1,k}}^{T}\mathbf{u}^{1,k},J^{k}\right\rangle _{\!N}\right|_{\tau=-1}.
\end{split}
\end{align}
The equations \eqref{EC:Equation1} and \eqref{EC:Equation5} provide the identity 
\begin{align}\label{EC:EntropyPreservation}
\begin{split}
0=&\quad \sum_{n=1}^{K_{T}}\sum_{k=1}^{K_{S}}\left(A_{T}\!\left(\mathbf{U}^{n,k},\mathbf{W}^{n,k}\right)+ A_{S}\!\left(\blockvec{\mathbf{F}},\mathbf{W}\right)\right) \\
=& \quad  \bar{S}\left(\mathbf{U}\left(T\right)\right)-\bar{S}\left(\mathbf{u}\left(0\right)\right)
+\sum_{k=1}^{K_{S}}\left.\left\langle \jump{\varPhi^{1,k}}-\jump{\mathbf{W}^{1,k}}^{T}\mathbf{u}^{1,k},J^{k}\right\rangle _{\!N}\right|_{\tau=-1}.
\end{split}
\end{align}
This equation is a discrete version of the continuous entropy equation \eqref{eq:ContinuousEntropyEquation}. In general the sum on the right hand side in equation \eqref{EC:EntropyPreservation} does not vanish, but the contribution can be interpreted as a projection error that is, in general, small \cite{CHQZ:2006}. Hence, we have proven the following statement.          
\begin{thm}[Entropy preservation]\label{Theorem:EntropyPreservation}
Consider the space-time DGSEM with Dirichlet boundary conditions in time and periodic boundary conditions in space. Assume that the temporal numerical states $\mathbf{U}^{*}$ have the properties \eqref{EntropyPreservation1} and \eqref{EntropyPreservation2}. The spatial numerical surface fluxes $\tilde{\mathbf{F}}_{\hat{n}}^{*}$ are computed by \eqref{ContravariantSurfaceFlux}. 
Then the space-time DGSEM is an entropy preserving method in the scene that the equation \eqref{EC:EntropyPreservation} is satisfied.    
\end{thm}
We note that with the temporal states defined by \eqref{EntropyPreservation2} it is possible to demonstrate entropy conservation of the space-time DG scheme. However, these temporal states are not a practicable choice for simulations because they fully couple the space-time slabs. Therefore, the only practical choice is the upwind temporal states $\mathbf{U}^{*}=\mathbf{U}_{-}$, which are entropy stable as shown in Theorem 1, as they provide the weakest coupling possible between the time slabs.

\section{The compressible Euler equations}\label{Sec:Euler}
As a flagship example for the temporal entropy analysis we consider the three dimensional compressible Euler equations
\begin{equation}\label{eq:Euler}
\pderivative{\vec{u}}{t} + \vv{\nabla} \cdot \blockvec{\vec{f}}=0, 
\end{equation}
with $\vec{u}:=\left(\rho,\rho\vv{v}, E\right)^{T}$, 
$\blockvec{\vec{f}}=\left(\vec{f}_{1},\vec{f}_{2},\vec{f}_{3}\right)^T$ and
\begin{equation}\label{eq:EulerFluxes}
\vec{f}_{1}=\left(\begin{array}{c}
\rho v_{1}\\
\rho v_{1}^{2}+p\\
\rho v_{1}v_{2}\\
\rho v_{2}v_{3}\\
\left(E+p\right)v_{1}
\end{array}\right),\qquad\vec{f}_{2}=\left(\begin{array}{c}
\rho v_{2}\\
\rho v_{1}v_{2}\\
\rho v_{2}^{2}+p\\
\rho v_{2}v_{3}\\
\left(E+p\right)v_{2}
\end{array}\right),\qquad\vec{f}_{3}=\left(\begin{array}{c}
\rho v_{3}\\
\rho v_{1}v_{3}\\
\rho v_{2}v_{3}\\
\rho v_{3}^2+p\\
\left(E+p\right)v_{3}
\end{array}\right).
\end{equation}   
The conserved states are the density $\rho$, the fluid velocities $\vv{v}=(v_1\,,\,v_2\,,\,v_3)$ and the total energy $E$. In order to close the system, we assume an ideal gas such that the pressure is defined as
\begin{equation}
p = (\gamma-1)\left(E - \frac{\rho}{2}\left|\vv{v}\right|^2\right),
\end{equation}  
where $\gamma$ is the adiabatic constant.
\subsection{Euler State Values in Time}\label{sec:EulerStuff}
Here we focus on the temporal entropy analysis. Complete details on the spacial entropy stability analysis for the Euler equations can be found in, e.g., \cite{carpenter_esdg,Chandrashekar2012,Gassner2017,IsmailRoe2009}. 
\begin{thm}[Entropy conservative temporal Euler state]
From the entropy conservation condition in time 
\begin{equation}\label{eq:condInThm}
\jump{\vec{W}}^T\vec{U}^{\#} = \jump{\Phi},
\end{equation}
we derive the temporal state for the Euler equations to be
\begin{equation}\label{Euler:ECStateInThm}
\vec{U}^\# =
\begin{bmatrix}
\rho^{\ln}\\[0.1cm]
\rho^{\ln}{\avg{v_1}}\\[0.1cm]
\rho^{\ln}{\avg{v_2}}\\[0.1cm]
\rho^{\ln}{\avg{v_3}}\\[0.1cm]
\frac{\rho^{\ln}}{2\beta^{\ln}(\gamma-1)}+\rho^{\ln}\left(\avg{v_1}^2+\avg{v_2}^2+\avg{v_3}^2-\frac{1}{2}\left(\avg{v_1^2}+\avg{v_2^2}+\avg{v_3^2}\right)\right)\\[0.1cm]
\end{bmatrix}
\end{equation}
with the arithmetic mean \eqref{SurfaceJumpMean} and introducing the logarithmic mean 
\begin{equation}\label{eq:logMean}
(\cdot)^{\ln} = \frac{\jump{\cdot}}{\jump{\ln(\cdot)}}.
\end{equation}
\end{thm}
\begin{proof}
First, we collect the necessary quantities for the discrete temporal entropy analysis of the Euler equations:
\begin{equation}
\begin{aligned}\label{Euler:Quantities}
\vec{U} &= (\rho\,,\,\rho v_1\,,\,\rho v_2\,,\,\rho v_3\,,\,E)^T,\\[0.1cm]
\vec{W} &= \left(\frac{\gamma-\varsigma}{\gamma-1}-\beta \left|\vv{v}\right|^2\,,\,2\beta v_1\,,\,2\beta v_2\,,\,2\beta v_3\,,\,-2\beta\right)^T,\\[0.1cm]
s &= -\frac{\rho \varsigma}{\gamma-1},\\[0.1cm]
\Phi &= \rho,\\[0.1cm]
\end{aligned}
\end{equation} 
where $\varsigma = -(\gamma-1)\ln(\rho)-\ln(\beta)-\ln(2)$ and $\beta=\rho/(2p)$. 
In order to compute the Euler state function $\vec{U}^\#$, we will rearrange the equation \eqref{eq:condInThm} by algebraic manipulations. Therefore, the following two important properties of the jump operator are necessary   
\begin{equation}\label{eq:jumpProperties}
\jump{ab} = \avg{a}\jump{b} + \avg{b}\jump{a}
\qquad \text{and} \qquad \jump{a^2} = 2\avg{a}\jump{a}, 
\end{equation}
where $a$ and $b$ are given quantities. Forgoing some algebra we apply the identities \eqref{eq:jumpProperties} several times and introduce the logarithmic mean \eqref{eq:logMean} to find
\begin{equation}
\resizebox{\textwidth}{!}{$
\begin{aligned}
\jump{\rho}:& \frac{U^\#_1}{\rho^{\ln}} = 1\\[0.1cm]
\jump{v_1}:& -2U^\#_1\avg{v_1}\avg{\beta} + 2U^\#_2\avg{\beta} = 0\\[0.1cm]
\jump{v_2}:& -2U^\#_1\avg{v_2}\avg{\beta} + 2U^\#_3\avg{\beta} = 0\\[0.1cm]
\jump{v_3}:& -2U^\#_1\avg{v_3}\avg{\beta} + 2U^\#_4\avg{\beta} = 0\\[0.1cm]
\jump{\beta}:& -2U^\#_5+\frac{U^\#_1}{\beta^{\ln}(\gamma-1)}-U^\#_1\left(\avg{v_1^2}+\avg{v_2^2}+\avg{v_3^2}\right)+2U_2^\#\avg{v_1}+2U_3^\#\avg{v_2}+2U_4^\#\avg{v_3}=0.
\end{aligned}
$}
\end{equation}
We solve to find the entropy conservative state function in time to be
\begin{equation}\label{Euler:ECState}
\vec{U}^\# =
\begin{bmatrix}
\rho^{\ln}\\[0.1cm]
\rho^{\ln}{\avg{v_1}}\\[0.1cm]
\rho^{\ln}{\avg{v_2}}\\[0.1cm]
\rho^{\ln}{\avg{v_3}}\\[0.1cm]
\frac{\rho^{\ln}}{2\beta^{\ln}(\gamma-1)}+\rho^{\ln}\left(\avg{v_1}^2+\avg{v_2}^2+\avg{v_3}^2-\frac{1}{2}\left(\avg{v_1^2}+\avg{v_2^2}+\avg{v_3^2}\right)\right)\\[0.1cm]
\end{bmatrix}.
\end{equation}
We note that $\vec{U}^\#$ is symmetric with respect to its arguments and is consistent, as taking the left and right states to be the same gives
\begin{equation}
\vec{U}^\# =
\begin{bmatrix}
\rho\\[0.1cm]
\rho v_1\\[0.1cm]
\rho v_2\\[0.1cm]
\rho v_3\\[0.1cm]
\frac{\rho}{2\beta(\gamma-1)}+\frac{\rho}{2}\left|\vv{v}\right|^2\\[0.1cm]
\end{bmatrix}
=
\begin{bmatrix}
\rho\\[0.1cm]
\rho v_1\\[0.1cm]
\rho v_2\\[0.1cm]
\rho v_3\\[0.1cm]
\frac{p}{(\gamma-1)}+\frac{\rho}{2}\left|\vv{v}\right|^2\\[0.1cm]
\end{bmatrix}
=
\begin{bmatrix}
\rho\\[0.1cm]
\rho v_1\\[0.1cm]
\rho v_2\\[0.1cm]
\rho v_3\\[0.1cm]
E\\[0.1cm]
\end{bmatrix}.
\end{equation}
Finally, a numerically stable procedure to compute the logarithmic mean \eqref{eq:logMean} is provided by Ismail and Roe \cite[Appendix B]{IsmailRoe2009}.
\end{proof}


The entropy conservative temporal state function $\vec{U}^\#$ couples all time slabs in the space-time DG scheme. Thus, it is not a computationally tractable option for approximating the solution of the Euler equations. As previously mentioned, the only feasible choice for a temporal numerical state is the pure upwind state, as it decouples the time slaps. Thus, we next show that the upwind numerical state satisfies \eqref{TemporalStability1} in the case of the Euler equations.

\begin{lemma}[Entropy stable temporal Euler state]
The upwind temporal state 
\begin{equation}\label{Euler:UpwindState}
\vec{U}^*=\begin{pmatrix}
\rhoM\\[0.1cm]
\rhoM\voneM\\[0.1cm]
\rhoM\vtwoM\\[0.1cm]
\rhoM\vthrM\\[0.1cm]
E_{-}\\[0.1cm]
\end{pmatrix},
\quad E_{-} = \frac{\rhoM}{2\betaM(\gamma-1)}+\frac{\rhoM}{2}\left(\voneM^2+\vtwoM^2+\vthrM^2\right),
\end{equation}
for the  compressible Euler equations satisfies the entropy stability condition in time
\begin{equation}\label{eq:EulerCondition}
\jump{\vec{W}}^T\vec{U}^*\leq \jump{\Phi}.
\end{equation}
\end{lemma}
\begin{proof}
First we compute the jump in the entropy variables
\begin{equation}
\resizebox{\textwidth}{!}{$
\jump{\vec{W}} 
=
	\begin{bmatrix}
	\frac{\jump{\rho}}{(\rho)^{\ln}}+\frac{\jump{\beta}}{\beta^{\ln}(\gamma-1)}-\Big(\avg{v_1^2}+\avg{v_2^2}+\avg{v_3^2}\Big)\jump{\beta}-2\avg{\beta}\Big( \avg{v_1}\jump{v_1} + \avg{v_2}\jump{v_2} + \avg{v_3}\jump{v_3} \Big) \\
	2 \avg{\beta}\jump{v_1} + 2 \avg{v_1}\jump{\beta} \\[0.1cm]
	2 \avg{\beta}\jump{v_2} + 2 \avg{v_2}\jump{\beta} \\[0.1cm]
	2 \avg{\beta}\jump{v_3} + 2 \avg{v_3}\jump{\beta} \\[0.1cm]
	-2 \jump{\beta} \\[0.1cm]
	\end{bmatrix}
	.$}
\end{equation}
Now we compute the left side of the condition \eqref{eq:EulerCondition} to be
\begin{equation}\label{eq:EulerFirst}
\resizebox{\textwidth}{!}{$
\begin{aligned}
\jump{\vec{W}}^T\vec{U}^* &= \rhoM\left(\frac{\jump{\rho}}{(\rho)^{\ln}}+\frac{\jump{\beta}}{\beta^{\ln}(\gamma-1)}-\Big(\avg{v_1^2}+\avg{v_2^2}+\avg{v_3^2}\Big)\jump{\beta}-2\avg{\beta}\Big( \avg{v_1}\jump{v_1} + \avg{v_2}\jump{v_2} + \avg{v_3}\jump{v_3} \Big)\right) + \rhoM\vM\left(2\avg{v}\jump{\beta}+2\avg{\beta}\jump{v}\right)-2E_{-}\jump{\beta}\\
&\quad+\rhoM\voneM\left(2 \avg{\beta}\jump{v_1} + 2 \avg{v_1}\jump{\beta}\right)+\rhoM\vtwoM\left(2 \avg{\beta}\jump{v_2} + 2 \avg{v_2}\jump{\beta}\right)+\rhoM\vthrM\left(2 \avg{\beta}\jump{v_3} + 2 \avg{v_3}\jump{\beta}\right)-2E_{-}\jump{\beta}\\[0.1cm]
&=\frac{\rhoM}{\rho^{\ln}}\jump{\rho} - 2\rhoM\avg{\beta}\jump{v_1}\left(\avg{v_1}-\voneM\right)- 2\rhoM\avg{\beta}\jump{v_2}\left(\avg{v_2}-\vtwoM\right)- 2\rhoM\avg{\beta}\jump{v_3}\left(\avg{v_3}-\vthrM\right) \\
&\quad + \jump{\beta}\left(\frac{\rhoM}{\beta^{\ln}(\gamma-1)}-\rhoM\avg{v_1^2}+2\rhoM\voneM\avg{v_1}-\rhoM\avg{v_2^2}+2\rhoM\vtwoM\avg{v_2}-\rhoM\avg{v_3^2}+2\rhoM\vthrM\avg{v_3}-2E_{-}\right).
\end{aligned}
$}
\end{equation}
First, we note for a given quantity $a$ we have
\begin{equation}\label{eq:avgProperty}
\avg{a} - a_{-} = \frac{1}{2}(a_{+}+a_{-}) - a_{-} = \frac{1}{2}\jump{a},
\end{equation}
and apply this finding on three terms in \eqref{eq:EulerFirst} as well as substitute the form of $E_{-}$ to obtain
\begin{equation}\label{eq:EulerFirstI}
\resizebox{\textwidth}{!}{$
\begin{aligned}
\jump{\vec{W}}^T\vec{U}^*&=\frac{\rhoM}{\rho^{\ln}}\jump{\rho} - \rhoM\avg{\beta}\left[\left(\jump{v_1}\right)^2 + \left(\jump{v_2}\right)^2 + \left(\jump{v_3}\right)^2\right] - \frac{\rhoM}{\beta_{-}\beta^{\ln}(\gamma-1)}\jump{\beta}\left(\beta^{\ln}-\beta_{-}\right) + \rhoM\jump{\beta}\left(-\avg{v_1^2}+2\voneM\avg{v_1}-\voneM^2\right.\\
&\quad \left.-\avg{v_2^2}+2\vtwoM\avg{v_2}-\vtwoM^2-\avg{v_3^2}+2\vthrM\avg{v_3}-\vthrM^2\right).
\end{aligned}
$}
\end{equation}
Now we do a manipulation on the velocity terms multiplied by $\jump{\beta}$, e.g., for $v_1$ we find
\begin{equation}\label{eq:velManip}
\begin{aligned}
- \avg{v_1^2}+2\voneM\avg{v_1}-\voneM^2 &= -\frac{1}{2}\voneP^2-\frac{1}{2}\voneM^2 + \voneP\voneM + \voneM^2 - \voneM^2\\
&= -\frac{1}{2}\left(\voneP^2-2\voneP\voneM+\voneM^2\right)\\
&= -\frac{1}{2}(\voneP-\voneM)^2\\
&= -\frac{1}{2}\left(\jump{v_1}\right)^2.
\end{aligned}
\end{equation}
This further simplifies \eqref{eq:EulerFirstI} to become
\begin{equation}\label{eq:EulerFirstII}
\resizebox{\textwidth}{!}{$
\begin{aligned}
\jump{\vec{W}}^T\vec{U}^*&=\frac{\rhoM}{\rho^{\ln}}\jump{\rho} - \rhoM\avg{\beta}\left[\left(\jump{v_1}\right)^2 + \left(\jump{v_2}\right)^2 + \left(\jump{v_3}\right)^2\right] - \frac{\rhoM}{\beta_{-}\beta^{\ln}(\gamma-1)}\jump{\beta}\left(\beta^{\ln}-\beta_{-}\right) -\frac{\rhoM}{2}\jump{\beta}\left[\left(\jump{v_1}\right)^2 + \left(\jump{v_2}\right)^2 + \left(\jump{v_3}\right)^2\right]\\[0.1cm]
&= \frac{\rhoM}{\rho^{\ln}}\jump{\rho} -\rhoM\left[\left(\jump{v_1}\right)^2 + \left(\jump{v_2}\right)^2 + \left(\jump{v_3}\right)^2\right]\left(\avg{\beta}+\frac{1}{2}\jump{\beta}\right) - \frac{\rhoM}{\beta_{-}\beta^{\ln}(\gamma-1)}\jump{\beta}\left(\beta^{\ln}-\beta_{-}\right).
\end{aligned}
$}
\end{equation}
From the definitions of the jump and the arithmetic mean \eqref{SurfaceJumpMean} we know
\begin{equation}\label{eq:betaAvgProp}
\avg{\beta}+\frac{1}{2}\jump{\beta} = \frac{1}{2}\left(\beta_{+} + \beta_{-} + \beta_{+} -\beta_{-}\right) = \beta_{+},
\end{equation}
and we add and subtract the value of $\jump{\rho}$ to \eqref{eq:EulerFirstII} to find
\begin{equation}\label{eq:EulerFirstIII}
\resizebox{\textwidth}{!}{$
\begin{aligned}
\jump{\vec{W}}^T\vec{U}^*&= \frac{\rhoM}{\rho^{\ln}}\jump{\rho} \pm \jump{\rho} -\rhoM\betaP\left[\left(\jump{v_1}\right)^2 + \left(\jump{v_2}\right)^2 + \left(\jump{v_3}\right)^2\right] - \frac{\rhoM}{\beta_{-}\beta^{\ln}(\gamma-1)}\jump{\beta}\left(\beta^{\ln}-\beta_{-}\right)\\[0.1cm]
&=\jump{\rho} - \frac{1}{\rho^{\ln}}\jump{\rho}\left(\rho^{\ln}-\rhoM\right)-\rhoM\betaP\left[\left(\jump{v_1}\right)^2 + \left(\jump{v_2}\right)^2 + \left(\jump{v_3}\right)^2\right] - \frac{\rhoM}{\beta_{-}\beta^{\ln}(\gamma-1)}\jump{\beta}\left(\beta^{\ln}-\beta_{-}\right).
\end{aligned}
$}
\end{equation}
From the definition of the logarithmic and arithmetic mean between two values of a given variable, $a$, we know that \cite{carlson1966}
\begin{equation}
a^{\ln} \leq \avg{a},
\end{equation}
and we can reuse the result \eqref{eq:avgProperty} to find
\begin{equation}\label{eq:logSimplify}
a^{\ln} - a_{-} \leq \avg{a} - a_{-} = \frac{1}{2}\jump{a}.
\end{equation}
We apply \eqref{eq:logSimplify} to the $\rho^{\ln}$ and $\beta^{\ln}$ terms in \eqref{eq:EulerFirstIII} to find
\begin{equation}\label{eq:EulerFirstIIII}
\begin{aligned}
\jump{\vec{W}}^T\vec{U}^* &\leq \jump{\rho} - \frac{1}{2\rho^{\ln}}\left(\jump{\rho}\right)^2-\rhoM\betaP\left[\left(\jump{v_1}\right)^2 + \left(\jump{v_2}\right)^2 + \left(\jump{v_3}\right)^2\right] - \frac{\rhoM}{2\beta_{-}\beta^{\ln}(\gamma-1)}\left(\jump{\beta}\right)^2\\[0.1cm]
&\leq\jump{\rho},
\end{aligned}
\end{equation}
under the physical assumptions of positive density and temperature. Thus, the upwind temporal state fulfills \eqref{eq:EulerCondition} and is entropy stable in time for the Euler equations.
\end{proof}

\begin{remark}
The result of Lemma 1 that the upwind state is entropy stable also holds for the shallow water (see \ref{sec:ShallowWater}) and the ideal MHD (see \ref{sec:UpwindMHD}) equations.
\end{remark}

\subsection{Kinetic energy preservation for the Euler equations}\label{sec:KineticEnergy}
For the Euler equations \eqref{eq:Euler}, it is also possible to recover a balance law for the kinetic energy such that the discrete integral of the kinetic energy is not changed by the advective terms, but only by the pressure work \cite{jameson2008}. Jameson \cite{jameson2008} analyzed finite volume methods with respect to the kinetic energy and constructed conditions on the numerical surface flux functions to generate kinetic energy preserving (KEP) schemes. Gassner et al. \cite{Gassner:2016ye} generalized the KEP scheme into the high-order DG context on Cartesian meshes. In this section, we extend these results to the space-time DGSEM on curvilinear hexahedral meshes and introduce similar conditions on the numerical state function to guarantee KEP in time.

\subsubsection{Continuous kinetic energy evolution}\label{Kinetic:ContinuousAnalysis}
We want the space-time DG scheme to mimic the continuous analysis. Therefore, much as we did for the entropy, we first examine and analyze the continuous kinetic energy balance and determine the steps that the discretization must capture. First, we define the following set of variables  
\begin{equation}\label{KineticEnergyVariables}
\vec{v}:=\left(-\frac{1}{2}\left|\vv{v}\right|^2,\vv{v},0\right)^T.
\end{equation}
Then, we obtain the kinetic energy by 
\begin{equation}
\vec{v}^{T}\vec{u}=-\frac{1}{2}\rho\left|\vv{v}\right|^{2}+\rho\vv{v}^{T}\vv{v}=\frac{1}{2}\rho\left|\vv{v}\right|^{2}=:\kappa.
\end{equation} 
We note that $\vec{v}^T=\pderivative{\kappa}{\vec{u}}$. Furthermore, it follows  
\begin{equation}
\vec{v}^{T}\pderivative{\vec{u}}{t}
=\pderivative{\kappa}{t}
\qquad \text{and} \qquad 
\vec{v}^{T}\left(\vv{\nabla}_x \cdot \blockvec{\vec{f}}\right)=\vv{\nabla}_x\cdot\vv{f}^{\kappa} + \vv{v}\cdot\vv{\nabla}_x p,
\end{equation}
where $\vv{f}^{\kappa}=\frac{1}{2}\rho\vv{v}\left|\vv{v}\right|^{2}$ can be interpreted as a kinetic energy flux and $\vv{v}\cdot\vv{\nabla}_x p$ is the pressure work. Thus, we obtain the equation for the kinetic energy balance 
\begin{equation}\label{KineticEnergyBalance}
\pderivative{\kappa}{t} + \vv{\nabla}_x \cdot \vv{f}^{\kappa}+\vv{v}^T\cdot\vv{\nabla}_x p=0.   
\end{equation} 
  
Next, we integrate the equation \eqref{KineticEnergyBalance} over a space-time domain $\Omega\times [0,T]$ to obtain
\begin{equation}\label{KineticEnergyBalanceIntegral1}
\int\limits_{\Omega}\int\limits_0^T \pderivative{\kappa}{t}\dt\dV + 
\int\limits_0^T\int\limits_{\Omega}\vv{\nabla}_x \cdot \vv{f}^{\kappa}\dV\dt 
+ \int\limits_0^T\int\limits_{\Omega} \vv{v}^T\cdot\vv{\nabla}_x p \dV\dt= 0.
\end{equation}
The terms on left hand side in \eqref{KineticEnergyBalanceIntegral1} are evaluated as follows: The fundamental theorem of calculus is used for the temporal integral in the first term, the divergence theorem is used for the spatial integral in the second term and the spatial integral in the last term is evaluated with the integration-by-parts formula. This results in 
\begin{equation}\label{KineticEnergyBalanceIntegral2}
\int\limits_{\Omega}\left(\kappa(x,y,z,T) - \kappa(x,y,z,0)\right)\dV  
- \int\limits_0^T\int\limits_{\Omega} (\vv{\nabla}_x\cdot\vv{v}) p\dV\dt  \\ 
+\int\limits_0^T\int\limits_{\partial\Omega} \left(\vv{f}^{\kappa}+p\vv{v}\right) \cdot\vv{n}\dS\dt= 0,
\end{equation}
where $\vv{n}$ is the normal at the physical boundary. Rearranging terms provides 
\begin{equation}\label{KineticEnergyBalanceIntegral3}
\int\limits_{\Omega}\kappa(x,y,z,T)\dV 
= 
\int\limits_{\Omega}\kappa(x,y,z,0)\dV
+ \int\limits_0^T\int\limits_{\Omega} (\vv{\nabla}_x\cdot\vv{v}) p\dV\dt
-\int\limits_0^T\int\limits_{\partial\Omega} \left(\vv{f}^{\kappa}+p\vv{v}\right) \cdot\vv{n}\dS\dt.  
\end{equation}
We see that in the incompressible case $\vv{\nabla}_x\cdot\vv{v}=0$. Thus, for incompressible flows, the spatial integral of the kinetic energy at time $T$ is bounded by its initial value provided proper boundary conditions are considered. 
In many situations, e.g., compressible flows, discontinuous solutions, or specific boundary conditions, it is only possible to obtain the inequality      
\begin{equation}\label{Goal:KineticEnergy}
\int\limits_{\Omega}\kappa(x,y,z,T)\dV 
\leq 
\int\limits_{\Omega}\kappa(x,y,z,0)\dV
+ \int\limits_0^T\int\limits_{\Omega} (\vv{\nabla}_x\cdot\vv{v}) p\dV\dt
-\int\limits_0^T\int\limits_{\partial\Omega} \left(\vv{f}^{\kappa}+p\vv{v}\right) \cdot\vv{n}\dS\dt.  
\end{equation}
Thus, the kinetic energy at a given time $T$ is bounded by its initial value, the volume integral over the non-conservative term $(\vv{\nabla}_x\cdot\vv{v}) p$ and a contribution from the boundary faces which can be controlled by suitable boundary conditions.

\subsubsection{Discrete kinetic energy analysis}
In the context of the space-time DGSEM, we are interested in finding an upper bound for the quantity
\begin{equation}\label{DicreteKineticIntegrals1}
\bar{K}\left(\mathbf{U}\left(T\right)\right):=
\sum_{k=1}^{K_{S}}\left.\left\langle \kappa\left(\mathbf{U}^{K_{T},k}\right),J^{k}\right\rangle _{N}\right|_{\tau=1}.
\end{equation}
The quantity \eqref{DicreteKineticIntegrals1} is a discrete version of the spatial integral of $\kappa(x,y,z,T)$. Likewise, the quantity  
\begin{equation}\label{DicreteKineticIntegrals2}
\bar{K}\left(\mathbf{u}\left(0\right)\right):=\sum_{k=1}^{K_{S}}\left.\left\langle \kappa\left(\mathbf{u}^{1,k}\right),J^{k}\right\rangle _{N}\right|_{\tau=-1},
\end{equation}
is a discrete version of the spatial integral of $\kappa(x,y,z,0)$. The quantity $\left.\mathbf{u}^{1,k}\right|_{\tau=-1}=\mathbf{u}^{k}\left(0\right)$ in \eqref{DicreteKineticIntegrals2} is the initial condition prescribed to the conservation law in the spatial cell $e_{k}$, $k=1,\dots,K_{S}$. In order to construct a space-time KEP DGSEM, we replace the derivative projection operator in the temporal part \eqref{TemporalPart} by    
\begin{equation}\label{TemporalDerivativeProjectionOperatorKinetic}
\mathbb{D}^{M}\mathbf{U}_{\sigma ijk}^{\text{KEP}}:=2\sum_{\theta=0}^{M}\mathcal{D}_{\sigma\theta}\mathbf{U}^{\text{KEP}}\left(\mathbf{U}_{\sigma ijk},\mathbf{U}_{\theta ijk}\right),
\end{equation}
where the state $\mathbf{U}^{\text{KEP}}$ is consistent, symmetric, and satisfies the conditions
\begin{equation}\label{DiscreteKineticFlux}
\mathbf{U}_{2}^{\text{KEP}}=\avg{v_{1}}_{\left(\sigma,\theta\right)ijk}\mathbf{U}_{1}^{\text{KEP}},\quad 
\mathbf{U}_{3}^{\text{KEP}}=\avg{v_{2}}_{\left(\sigma,\theta\right)ijk}\mathbf{U}_{1}^{\text{KEP}}, \quad 
\mathbf{U}_{4}^{\text{KEP}}=\avg{v_{3}}_{\left(\sigma,\theta\right)ijk}\mathbf{U}_{1}^{\text{KEP}},
\end{equation}
for $\sigma,\theta=0,\dots,M$ and $i,j,k=0,\dots,N$. 
The volume averages in \eqref{DiscreteKineticFlux} are given by
\begin{equation}
\avg{v_{l}}_{\left(\sigma,\theta\right)ijk}=\frac{1}{2}
\left(\left(v_{l}\right)_{\sigma ijk}+\left(v_{l}\right)_{\theta ijk}\right), 
\qquad l=1,2,3.
\end{equation} 
We note that the entropy preserving temporal Euler state $\mathbf{U}^{\#}$ given by \eqref{Euler:ECState} also satisfies the conditions \eqref{DiscreteKineticFlux}. The derivative projection operator in the spatial part \eqref{SpatialPart} is replaced by 
\begin{align}\label{SpatialDerivativeProjectionOperator}
\begin{split}
\Dprojection{N}\cdot\blockvec{\tilde{\mathbf{F}}}{}_{\sigma ijk}^{\text{KEP}}
:=2\sum_{m=0}^{N}\quad\,& 
\mathcal{D}_{im}\left(\blockvec{\mathbf{F}}{}^{\text{KEP},1}\left(\mathbf{U}_{\sigma ijk},\mathbf{U}_{\sigma mjk}\right)\cdot\avg{J\vv{a}^{1}}_{\left(i,m\right)jk}\right) \\
\quad +  & \,\mathcal{D}_{jm}\left(\blockvec{\mathbf{F}}{}^{\text{KEP},2}\left(\mathbf{U}_{\sigma ijk},\mathbf{U}_{\sigma imk}\right)\cdot\avg{J\vv{a}^{2}}_{i\left(j,m\right)k}\right) \\
\quad  +& \, \mathcal{D}_{km}\left(\blockvec{\mathbf{F}}{}^{\text{KEP},3}\left(\mathbf{U}_{\sigma ijk},\mathbf{U}_{\sigma ijm}\right)\cdot\avg{J\vv{a}^{3}}_{ij\left(k,m\right)}\right), 
\end{split}
\end{align}
where the volume averages of the metric terms are given by 
\eqref{SpatialVolumeAverages} .
The fluxes 
\begin{equation}\label{BlockKEP:Flux}
\blockvec{\mathbf{F}}{}^{\text{KEP},l}=\left(\mathbf{F}_{1}^{\text{KEP},l},\mathbf{F}_{2}^{\text{KEP},l},\mathbf{F}_{3}^{\text{KEP},l}\right)^T, \qquad l=1,2,3,
\end{equation}
are consistent with $\blockvec{\mathbf{F}}$ and symmetric such that e.g. 
\begin{equation}\label{FluxSymmetric}
\blockvec{\mathbf{F}}{}^{\text{KEP,l}}\left(\mathbf{U}_{\sigma ijk},\mathbf{U}_{\sigma mjk}\right)=\blockvec{\mathbf{F}}{}^{\text{KEP,l}}\left(\mathbf{U}_{\sigma mjk},\mathbf{U}_{\sigma ijk}\right),
\end{equation}      
for $\sigma=0,\dots,M$ and $i,j,k,m=0,\dots,N$. Furthermore the flux functions 
$\mathbf{F}_{s}{}^{\text{KEP},l}$, $l=1,2,3$ and $s=1,2,3$, satisfy Jameson's conditions \cite{jameson2008}   
\begin{equation}\label{Jameson}
{
\scriptsize
\begin{array}{ccccccccccc}
\mathbf{F}_{1}^{2,\text{KEP},l}= & \avg{v_{1}}\mathbf{F}_{1}^{1,\text{KEP},l} & +p_{1l}^{\star}, &  & \mathbf{F}_{2}^{2,\text{KEP},l}= & \avg{v_{1}}\mathbf{F}_{2}^{1,\text{KEP},l}, &  &  & \mathbf{F}_{2}^{3,\text{KEP},l}= & \avg{v_{1}}\mathbf{F}_{3}^{1,\text{KEP},l},\\[0.15cm]
\mathbf{F}_{1}^{3,\text{KEP},l}= & \avg{v_{2}}\mathbf{F}_{1}^{1,\text{KEP},l}, &  &  & \mathbf{F}_{2}^{3,\text{KEP},l}= & \avg{v_{2}}\mathbf{F}_{2}^{1}{}^{\text{KEP},l} & +p_{2l}^{\star}, &  & \mathbf{F}_{3}^{3,\text{KEP},l}= & \avg{v_{2}}\mathbf{F}_{3}^{1,\text{KEP},l},\\[0.15cm]
\mathbf{F}_{1}^{4,\text{KEP},l}= & \avg{v_{3}}\mathbf{F}_{1}^{1,\text{KEP},l}, &  &  & \mathbf{F}_{2}^{4,\text{KEP},l}= & \avg{v_{3}}\mathbf{F}_{2}^{1,\text{KEP},l}, &  &  & \mathbf{F}_{3}^{4,\text{KEP},l}= & \avg{v_{3}}\mathbf{F}_{3}^{1,\text{KEP},l} & +p_{3l}^{\star}, 
\end{array}
}
\end{equation}
where 
\begin{equation}\label{Jameson1} 
p_{sl}^{\star}=2\avg{p}-\frac{\avg{pJa_{l}^{s}}}{\avg{Ja_{l}^{s}}},\qquad s=1,2,3, \text{ and } l=1,2,3.
\end{equation}
We note that the quantities \eqref{Jameson1} become $p_{sl}=\avg{p}$ for all $s=1,2,3$ and $l=1,2,3$ in the case of a Cartesian mesh with non-curved elements. Moreover, it is enough to compute the spatial derivative projection 
operator \eqref{SpatialDerivativeProjectionOperator} from one Cartesian block flux  $\blockvec{\mathbf{F}}{}^{\text{KEP}}$ which components satisfy Jameson's conditions  
\begin{equation}\label{Jameson2}
{
\scriptsize
\begin{array}{ccccccccccc}
\mathbf{F}_{1}^{2,\text{KEP}}= & \avg{v_{1}}\mathbf{F}_{1}^{1,\text{KEP}} & +\avg{p}, &  & \mathbf{F}_{2}^{2,\text{KEP}}= & \avg{v_{1}}\mathbf{F}_{2}^{1,\text{KEP}}, &  &  & \mathbf{F}_{2}^{3,\text{KEP}}= & \avg{v_{1}}\mathbf{F}_{3}^{1,\text{KEP}},\\[0.15cm]
\mathbf{F}_{1}^{3,\text{KEP}}= & \avg{v_{2}}\mathbf{F}_{1}^{1,\text{KEP}}, &  &  & \mathbf{F}_{2}^{3,\text{KEP}}= & \avg{v_{2}}\mathbf{F}_{2}^{1,}{}^{\text{KEP}} & +\avg{p}, &  & \mathbf{F}_{3}^{3,\text{KEP}}= & \avg{v_{2}}\mathbf{F}_{3}^{1,\text{KEP}},\\[0.15cm]
\mathbf{F}_{1}^{4,\text{KEP}}= & \avg{v_{3}}\mathbf{F}_{1}^{1,\text{KEP}}, &  &  & \mathbf{F}_{2}^{4,\text{KEP}}= & \avg{v_{3}}\mathbf{F}_{2}^{1,\text{KEP}}, &  &  & \mathbf{F}_{3}^{4,\text{KEP}}= & \avg{v_{3}}\mathbf{F}_{3}^{1,\text{KEP}} & +\avg{p}. 
\end{array}
}
\end{equation}
In Ranocha's PHD thesis \cite[Chapter 7, Section 4]{ranocha2018} an example of an entropy conserving two point flux with the property \eqref{Jameson2} was developed. This flux is described in the Appendix \ref{sec:App E}. Moreover, the conditions \eqref{Jameson2} were used in \cite{Gassner:2016ye} to prove that a semi-discrete high-order DGSEM with non-curved elements is a KEP method. For a space-time DGSEM with curved hexahedral elements, we have the following result.      
\begin{thm}[Kinetic energy preservation]\label{Therorem:KEP}
A space-time DGSEM \eqref{SpaceTimeDG1} for the Euler equations with:
\begin{enumerate}
\item Dirichlet boundary conditions in time and periodic boundary conditions in space,   
\item  temporal numerical state functions $\mathbf{U}^{*}$ with the property  \eqref{EntropyPreservation1} at the exterior temporal boundary points and the property \eqref{DiscreteKineticFlux} at the interior temporal boundary points,   
\item surface fluxes $\tilde{\mathbf{F}}_{\hat{n}}$ computed from Cartesian fluxes $\mathbf{F}_{l}^{\text{KEP}}$, $l=1,2,3$, with the property \eqref{Jameson2}, 
\end{enumerate}
is a kinetic energy preserving method such that
\begin{align}\label{eq:DiscreteKEP}
\begin{split}
\bar{K}\left(\mathbf{U}\left(T\right)\right)
&=
\bar{K}\left(\mathbf{U}\left(0\right)\right)
+\sum_{n=1}^{K_{T}}\sum_{k=1}^{K_{S}}\frac{\Delta t^{n}}{2}\left\langle \vv{\nabla}_{\xi}\cdot\interpolation{N}{\left(\vv{\tilde{v}}^{n,k}\right)},p^{n,k}\right\rangle _{N\times M} \\ 
& \quad \qquad \qquad \ \ +\sum_{n=1}^{K_{T}}\underset{\text{faces}}{\sum_{\text{Interior}}}\frac{\Delta t^{n}}{2}\int\limits _{\partial E^{3},N}\left\langle \avg{p^{n,k}}\jump{\tilde{v}_{\hat{n}}^{n,k}},1\right\rangle _{M}\dS \\
&\quad \qquad \qquad \ \ +
\sum_{k=1}^{K_{S}}\left.\left\langle \jump{\mathbf{V}^{1,k}}^{T}\mathbf{u}^{1,k},J^{k}\right\rangle _{N}\right|_{\tau=-1}.
\end{split}
\end{align} 
\end{thm}
\begin{remark}
We note that the equation \eqref{eq:DiscreteKEP} is a discrete equivalent of the continuous equation \eqref{KineticEnergyBalanceIntegral3}. In general the last sum on the right side in \eqref{eq:DiscreteKEP} does not vanish, but the contribution of this quantity is small. In addition, we note that due to the discontinuous solution space ansatz the numerical approximation of $\vv{\nabla}_x\cdot\vv{v}$ generates volume and surface contributions.
\end{remark}
The proof for the Theorem \ref{Therorem:KEP} is similar to the proofs of the entropy stability and preservation in the sections \ref{sec:Entropy} and \ref{sec:EntropyPreservation}. Nevertheless, for the sake of completeness, we present a proof in Appendix \ref{sec:TimeProofsKEP1}.   

We note that it is not a computationally tractable option to compute the state functions $\vec{U}^{*}$ as given in point 2 of Theorem \ref{Therorem:KEP}, since a state function with the property \eqref{DiscreteKineticFlux} couples the time slaps. It is more convenient to choose the state functions $\vec{U}^{*}$ by the upwind state \eqref{Euler:UpwindState}. The upwind state does not have the KEP property \eqref{DiscreteKineticFlux}. However, by the properties \eqref{eq:jumpProperties} of the jump operator we know  
\begin{equation}\label{KineticUpwindState1}
\jump{\left|\vv{v}\right|^{2}}=2\left(\avg{v_{1}}\jump{v_{1}}+\avg{v_{2}}\jump{v_{2}}+\avg{v_{3}}\jump{v_{3}}\right). 
\end{equation}  
Thus, it follows directly by the definition of the temporal upwind state   
\begin{align}\label{KEPConditionUPWINDFLUX}
\begin{split}
\jump{\vec{V}}^T\vec{U}^{*} & =
  -\frac{1}{2}\jump{\left|\vv{v}\right|^{2}}\rho_{-} 
 +\jump{v_{1}}\rho_{-}\left(v_{1}\right)_{-} 
 +\jump{v_{2}}\rho_{-}\left(v_{2}\right)_{-} 
 +\jump{v_{3}}\rho_{-}\left(v_{3}\right)_{-}  \\
& =  \rho_{-}\left(\jump{v_{1}}\left(\left(v_{1}\right)_{-}-\avg{v_{1}}\right)+\jump{v_{2}}\left(\left(v_{2}\right)_{-}-\avg{v_{2}}\right)+\jump{v_{3}}\right)\left(\left(v_{3}\right)_{-}-\avg{v_{3}}\right) \\ 
 & =  -\frac{1}{2}\rho_{-}\left(\jump{v_{1}}^{2}+\jump{v_{2}}^{2}+\jump{v_{3}}^{2}\right) \\
 & \leq 0. 
\end{split}
\end{align}
If, we apply this inequality instead of the identity \eqref{DiscreteKineticFlux} in the proof of the Theorem \ref{Therorem:KEP} (cf. Appendix \ref{sec:TimeProofsKEP1}), we obtain the following inequality for the space-time DGSEM     
\begin{align}\label{eq:DiscreteKEPInequality}
\begin{split}
\bar{K}\left(\mathbf{U}\left(T\right)\right)
&\leq
\bar{K}\left(\mathbf{U}\left(0\right)\right)
+\sum_{n=1}^{K_{T}}\sum_{k=1}^{K_{S}}\frac{\Delta t^{n}}{2}\left\langle \vv{\nabla}_{\xi}\cdot\interpolation{N}{\left(\vv{\tilde{v}}^{n,k}\right)},p^{n,k}\right\rangle _{N\times M} \\ 
& \quad \qquad \qquad \ \ +\sum_{n=1}^{K_{T}}\underset{\text{faces}}{\sum_{\text{Interior}}}\frac{\Delta t^{n}}{2}\int\limits _{\partial E^{3},N}\left\langle \avg{p^{n,k}}\jump{\tilde{v}_{\hat{n}}^{n,k}},1\right\rangle _{M}\dS,
\end{split}
\end{align} 
since the definition of the upwind state \eqref{Euler:UpwindState} and the inequality \eqref{eq:DiscreteKEPInequality} provide   
\begin{equation}\label{eq:DiscreteKEPInequality1}
\left.\left\langle \jump{\mathbf{V}^{1,k}}^{T}\mathbf{u}^{1,k},J^{k}\right\rangle _{N}\right|_{\tau=-1}=\left.\left\langle \jump{\mathbf{V}^{1,k}}^{T}\mathbf{U}^{1,k,*}\right\rangle _{N}\right|_{\tau=-1}\leq 0,\qquad k=1,\dots,K_{S}.
\end{equation} 
We note that the inequality \eqref{eq:DiscreteKEPInequality} is the discrete equivalent of the continuous inequality \eqref{Goal:KineticEnergy}.


\section{Numerical results}\label{sec:numResults}
In this section, we present numerical tests for the one dimensional compressible Euler equations 
\begin{equation}\label{eq:Euler1D}
\pderivative{\vec{u}}{t} + \pderivative{\vec{f}}{x}=0, 
\end{equation}
\begin{equation}
\vec{u}:=\left(\rho,\rho v, E\right)^{T},\qquad 
\vec{f}=\left(\rho v, \rho v^{2}+p, \left(E+p\right)v\right)^T,
\end{equation}
to evaluate the theoretical findings of the previous sections. The density $\rho$, the fluid velocities $v$ and the total energy $E$ are conserved states in the Euler system. Furthermore, to close the system, we assume an ideal gas such that the pressure is defined as
\begin{equation}
p = (\gamma-1)\left(E - \frac{1}{2}\rho v^2\right),
\end{equation}  
where $\gamma$ is the adiabatic constant. In all the numerical examples is $\gamma=1.4$ which is the adiabatic constant for air. The primary concern is the numerical verification of the high-order accuracy, entropy stability, entropy preservation and  kinetic energy preservation of the space-time DGSEM. In particular, we show that the temporal upwind state \eqref{Euler:UpwindState} is entropy stable and the temporal state \eqref{Euler:ECState} 
preserves entropy and kinetic energy. 

\subsection{Convergence test}\label{sec:convergence}

We start with a demonstration of the high-order accuracy by a manufactured solution 
\begin{equation}\label{MFS:Euler}
\mathbf{U}\left(x,t\right)
=
\begin{pmatrix}
\rho\left(x,t\right)\\[0.15cm] \rho v\left(x,t\right) \\[0.15cm] E\left(x,t\right)\\[0.15cm]
\end{pmatrix}
=
\begin{pmatrix}
2+\sin\left(2\pi(x-t)\right)\\[0.15cm]
2+\sin\left(2\pi(x-t)\right) \\[0.15cm]
\left(2+\sin\left(2\pi(x-t)\right)\right)^2\\[0.15cm]
\end{pmatrix}.
\end{equation}
The state \eqref{MFS:Euler} is a smooth analytical solution of the compressible Euler equations, if we consider the Euler system with the source term 
\begin{equation}
\vec{Q}\left(x,t\right)
=
\begin{pmatrix}
0\\[0.1cm] \frac{\pi(\gamma-1)}{2}\cos\left(2\pi(x-t)\right) \\[0.1cm] \frac{\pi(\gamma-1)}{2}\cos\left(2\pi(x-t)\right)\\[0.1cm]
\end{pmatrix}.
\end{equation}
The spatial domain is $\left[0,1\right]$ and periodic boundary conditions in space are used. The final time of the simulation is chosen to be $T=1$ and Dirichlet Boundary conditions are used in time. The space-time elements are uniformly distributed with the length  $\Delta x= \frac{1}{K_{S}}$ for the spatial elements and the length  $\Delta t= \frac{1}{K_{T}}$ for the temporal elements, where $K_{S}$ denotes the number of spatial elements and $K_{T}$ denotes the number of temporal elements. We apply the space-time DGSEM with the temporal derivative projection operator \eqref{TemporalDerivativeProjectionOperator} computed by \eqref{Euler:ECState} and the upwind state \eqref{Euler:UpwindState} is used at the temporal interfaces. The spatial derivative projection operator \eqref{SpatialDerivativeProjectionOperator} is computed by the entropy conservative (EC) flux developed by Ranocha \cite{ranocha2018} given in Appendix \ref{sec:App E} and the EC flux with the dissipation matrix of the form \eqref{CartesianFluxesES} is used as spatial surface numerical flux in this example. 

The initial condition is determined by interpolation of $\mathbf{U}\left(x,0\right)$ at $N+1$ LGL nodes. Table \ref{tab:P2P2} shows the $L^{2}$ errors of the conserved quantities and order of convergence for temporal polynomials of degree $M=2$ and spatial polynomials of degree $N=2$. 
 The observed experimental convergence rates agree with the expected optimal order three for the space time scheme, e.g. \cite{gassner2015space}.
\begin{center}
\begin{tabular}{c|c|cccccc}\toprule[1.5pt]
$K_{T}$ &  $K_{S}$
&  ${L}^{2}\left(\rho\right)$ &  $EOC(\rho)$ &  ${L}^{2}\left(\rho u \right)$  &   $EOC(\rho u)$  &  ${L}^{2}\left(E\right)$   & $EOC(E)$  \\\midrule
2 & 2  &  3.73E-02     & -     & 5.84E-02   & - & 1.43E-01 & -    \\
4 & 4   &6.27E-03 &2.6	&8.55E-03 &2.8	&1.91E-02 &2.9  \\   
8 & 8 &7.79E-04 &3.0	&6.78E-04 &3.7	&2.39E-03			&3.0     \\
16 & 16 	&1.20E-04 &2.7	&6.65E-05 & 3.3	&3.35E-04			&2.8\\
32 & 32 & 1.55E-05 & 3.0 & 6.95E-06 & 3.3 & 3.84E-05  &3.1 \\ 
\bottomrule[1.5pt]
\end {tabular}\par
\captionof{table}{Experimental order of convergence (EOC) for manufactured solution test \eqref{MFS:Euler}. 
The space-time DGSEM is used with temporal polynomial degree $M=2$ and spatial polynomial degree $N=2$.}\label{tab:P2P2}
\end{center}
The number for the spatial and temporal elements was chosen equal for the test in Table \ref{tab:P2P2}. Thus, we repeat this test example, but this time the number for the temporal elements is chosen to be twice as large as the number of the spatial elements. The results are given in Table \ref{tab:P2P2Diffgrid}. 
We the same order of convergence for the conserved quantities 
 as in Table \ref{tab:P2P2}. 
\begin{center}
\begin{tabular}{c|c|cccccc}\toprule[1.5pt]
$K_{T}$ &  $K_{S}$ &  $L^{2}\left(\rho\right)$ &  $EOC(\rho)$ &  $L^{2}\left(\rho u \right)$  &   $EOC(\rho u)$  &  $L^{2}\left(E\right)$   & $EOC(E)$  \\\midrule
 4 &  2  &  2.90E-02   & -     & 4.52E-02   & - & 1.05E-01  & -    \\
  8 &  4  &6.41E-03 &2.2	&8.56E-03 &2.4	&1.83E-02 &2.5  \\    
16 &  8 &8.10E-04 &3.0	& 7.20E-04&3.6	&2.40E-03			&2.9     \\
32 & 16 &1.21E-04 &2.7	& 6.86E-05 &3.4	&3.35E-04		&2.8\\
64 & 32  & 1.56E-05 & 3.0 & 7.03E-06 &3.3 & 3.84E-05  &3.1 \\
\bottomrule[1.5pt]
\end {tabular}\par
\captionof{table}{Experimental order of convergence (EOC) for manufactured solution test \eqref{MFS:Euler}. 
The space-time DGSEM is used with temporal polynomial degree $M=2$ and spatial polynomial degree $N=2$. The grid cell number $K_{T}$ for the temporal cells is chosen twice big as the spatial grid cell number $K_{S}$.}\label{tab:P2P2Diffgrid}
\end{center}
In Table \ref{tab:P3P3} the behavior of the space-time DGSEM for polynomials with an odd degree is shown. The $L^{2}$ errors of the conserved quantities and order of convergence for temporal polynomials of degree $M=3$ and spatial polynomials of degree $N=3$ are presented. The experimental convergence order is suboptimal.   
\begin{center}
\centering
\begin{tabular}{c|c|cccccc}\toprule[1.5pt]
$K_{T}$ &  $K_{S}$ &  ${L}^{2}\left(\rho\right)$ &  $EOC(\rho)$ &  ${L}^{2}\left(\rho u \right)$  &   $EOC(\rho u)$  &  ${L}^{2}\left(E\right)$   & $EOC(E)$  \\\midrule
  2 &  2  &  6.85E-03  & -     & 1.17E-02   & - & 2.31E-02  & -    \\
  4 &  4   &3.83E-04 &4.2	&3.67E-04 &5.0	&1.37E-03 &4.1 \\   
  8 &  8  &5.70E-05 &2.7	&3.99E-05 &3.2	&1.49E-04			&3.2     \\
  16 &  16  &7.08E-06 &3.0	&2.75E-06 &3.9	&1.40E-05			&3.4     \\
  32 &  32 &5.44E-07 &3.7	&2.54E-07 &3.4	&1.22E-06			&3.5    \\
\bottomrule[1.5pt]
\end {tabular}\par
\captionof{table}{Experimental order of convergence (EOC) for manufactured solution test \eqref{MFS:Euler}. 
The space-time DGSEM is used with temporal polynomial degree $M=3$ and spatial polynomial degree $N=3$.}\label{tab:P3P3}
\end{center}
Next, we choose polynomials of different degree in time and space and compute the convergence order of the space-time DGSEM. The $L^{2}$ errors of the conserved quantities and order of convergence for temporal polynomials of degree $M=3$ and spatial polynomials of degree $N=2$ are presented in Table \ref{tab:P3P2}. The error of the spatial approximation dominates compared to the temporal approximation errors such that an experimental convergence order of three is obtained. 
\begin{center}
\begin{tabular}{c|c|cccccc}\toprule[1.5pt]
$K_{T}$ &  $K_{S}$  &  $L^{2}\left(\rho\right)$ &  $EOC(\rho)$ &  $L^{2}\left(\rho u \right)$  &   $EOC(\rho u)$  &  $L^{2}\left(E\right)$   & $EOC(E)$  \\\midrule
  2 &  2 &  2.98E-02    & -     & 4.48E-02   & - & 1.11E-01  & -    \\
  4 &  4   & 6.58E-03 &2.2	&8.78E-03 &2.4	&1.84E-02 &2.6  \\   
  8 &  8  &8.17E-04 &3.0	&7.32E-04 &3.6	&2.40E-03			&2.9     \\
  16 &  16 	&1.22E-04 &2.7	&6.91E-05 &3.4	&3.35E-04		&2.8\\
  32 &  32 & 1.56E-05 & 3.0 & 7.05E-06 &3.3 & 3.84E-05  &3.1 \\
\bottomrule[1.5pt]
\end {tabular}\par
\captionof{table}{Experimental order of convergence (EOC) for manufactured solution test \eqref{MFS:Euler}. 
The space-time DGSEM is used with temporal polynomial degree $M=3$ and spatial polynomial degree $N=2$.}\label{tab:P3P2}
\end{center}
We observe a similar effect when the degree of the spatial polynomials is larger than the degree of the temporal polynomials. In Table \ref{tab:P2P3}, we present the $L^{2}$ errors of the conserved quantities and order of convergence for temporal polynomials of degree $M=2$ and spatial polynomials of degree $N=3$. 
We reach the optimal order four for the three conserved quantities as in the Table \ref{tab:P3P3}.     
\begin{center}
\begin{tabular}{c|c|cccccc}\toprule[1.5pt]
$K_{T}$ &  $K_{S}$ & $L^{2}\left(\rho\right)$ &  $EOC(\rho)$ &  $L^{2}\left(\rho u \right)$  &   $EOC(\rho u)$  &  $L^{2}\left(E\right)$   & $EOC(E)$  \\\midrule
2 &  2    &2.73E-02  & -     & 3.58E-02    & - & 6.28E-02 & -    \\
4 & 4  &4.39E-03 &2.6	&5.69E-03 &2.7	&7.58E-03 &3.1  \\   
8 &  8 &3.29E-04 &3.7	&4.09E-04 &3.8	&4.72E-04			&4.0     \\
16 &  16	&2.34E-05 &3.8	&2.91E-05 &3.8	&3.69E-05		&3.7\\
32 &  32 & 1.51E-06 & 4.0 & 1.86E-06 &4.0 & 2.50E-06  &3.9 \\
\bottomrule[1.5pt]
\end {tabular}\par
\captionof{table}{Experimental order of convergence (EOC) for manufactured solution test \eqref{MFS:Euler}. 
The space-time DGSEM is used with temporal polynomial degree $M=2$ and spatial polynomial degree $N=3$. The grid cell number $K_{T}$ for the temporal cells is chosen twice large as the spatial grid cell number $K_{S}$.}\label{tab:P2P3}
\end{center}
\subsection{Entropy stability check}\label{sec:ESCheck}
We consider the one-dimensional Euler equations \eqref{eq:Euler1D} with the initial discontinuous data 
\begin{equation}\label{Test:EulerECStability}
\rho_{0}=   \begin{cases}
     1 & \text{for } x \le 0.3 \\
     1.125 & \text{for } x>0.3 \\
   \end{cases}, \quad 
 v_{0} =0, \quad 
 p_{0} =   \begin{cases}
     1 & \text{for } x \le 0.3 \\
     1.1 & \text{for } x>0.3 \\
   \end{cases}\\
\end{equation}
on the spatial domain $\left[0,1\right]$. Periodic boundary conditions are used in space and Dirichlet Boundary conditions are used in time. We note that \eqref{Test:EulerECStability} provides a shock solution. To test the entropy stability property, we  measure the error   
\begin{equation}\label{ErrorTortalEntropy}
\Delta_{S}(T):=\bar{S}\left(\mathbf{U}\left(T\right)\right)-\bar{S}\left(\mathbf{u}\left(0\right)\right),
\end{equation} 
where $\bar{S}\left(\mathbf{U}\left(T\right)\right)$ is computed by \eqref{DisreteEntropyIntegral1} and $\bar{S}\left(\mathbf{u}\left(0\right)\right)$ is computed by \eqref{DisreteEntropyIntegral2}. The quantity represents the difference between the discrete total entropy at time $T$ and initial time and should decrease in time. The spatial domain is decomposed into $K_{S}=4$ spatial elements. The final time of the simulation is chosen to be $T=1$. Thus, we partition the temporal domain $\left[0,1\right]$ in $K_{T}=4,16,128$ elements. 
The temporal polynomial degree is $M=3$ and the spatial polynomial degree is $N=2$. Furthermore, we apply the space-time DGSEM with the same temporal derivative projection operator \eqref{TemporalDerivativeProjectionOperator} as in Section \ref{MFS:Euler} and the spatial derivative projection operator \eqref{SpatialDerivativeProjectionOperator} as in Section \ref{MFS:Euler}. 
The evolution of the error \eqref{ErrorTortalEntropy} until a final time of $T=1$ is presented in Fig. \ref{fig:EntropyStability} for both EC flux functions. We observe that the quantity \eqref{ErrorTortalEntropy} decreases in time as predicted by the theoretical result of Theorem 1.    
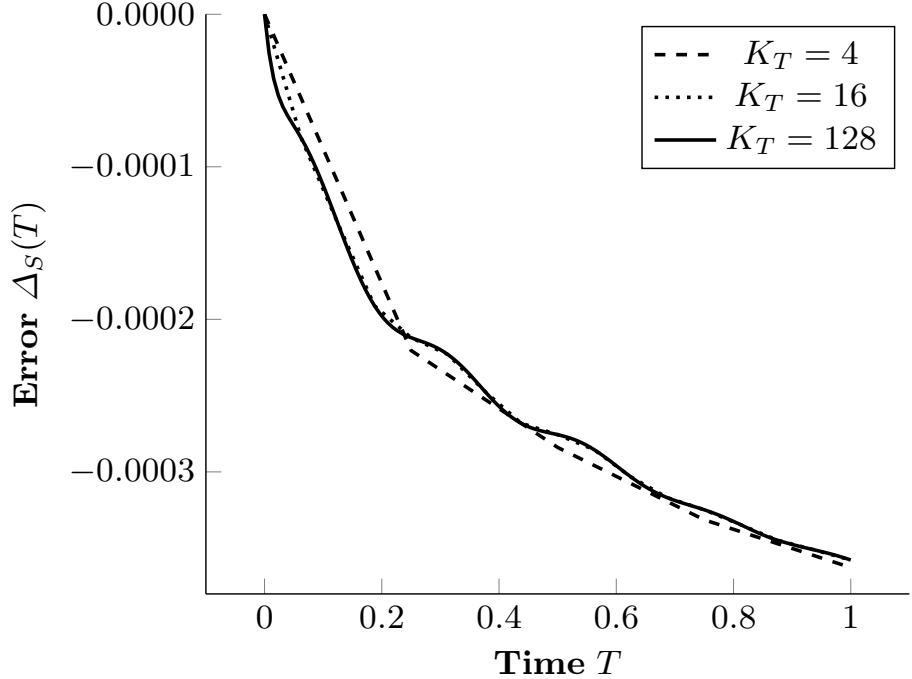
\begin{figure}[h]
\centering
\begin{tikzpicture}[scale=1.35]
\begin{axis}
[ymin=-0.00038,
 ymax=0,
    yticklabel style={
            /pgf/number format/fixed,
            /pgf/number format/precision=4,
            /pgf/number format/fixed zerofill
        }, xlabel={\textbf{Time} $T$}, ylabel={\textbf{Error} $\Delta_{S}(T)$ }, 
        scaled y ticks=false,
    axis x line*=bottom,
    axis y line*=left]
\addplot [dashed,line width=1.000] table[x index=0,y index=1] {DataEntropyStabilityKT=4.txt};
\addplot [dotted, line width=1.000] table[x index=0,y index=1] {DataEntropyStabilityKT=16.txt};
\addplot [line width=1.000] table[x index=0,y index=1] {DataEntropyStabilityKT=128.txt};
\legend{$K_{T}=4$, $K_{T}=16$, $K_{T}=128$}
\end{axis}
\end{tikzpicture}
\caption{Temporal evolution of the difference between the discrete total entropy at initial time and the discrete total energy at time $T$ represented by of the error $\Delta_{S}(T)$ given by \eqref{ErrorTortalEntropy} for the initial data \eqref{Test:EulerECStability} with $K_{T}=4,16,128$ temporal elements and $K_{S}=4$ spatial elements. The temporal polynomial degree is $M=3$ and the spatial polynomial degree $N=2$.}
	\label{fig:EntropyStability}
\end{figure}

\subsection{Entropy conservation check}\label{sec:ECCheck}
In this section, we investigate numerically the theoretical findings from Theorem~2. Therefore, we consider the same initial value problem \eqref{Test:EulerECStability}. As in the previous section, periodic boundary conditions are used in space and Dirichlet Boundary conditions are used in time. The spatial domain is decomposed into $K_{S}$ spatial elements and the temporal domain $\left[0,1\right]$ in $K_{T}$ elements. The space-time DGSEM is applied with the temporal derivative projection operator \eqref{TemporalDerivativeProjectionOperator} computed by \eqref{Euler:ECState} and the spatial derivative projection operator \eqref{SpatialDerivativeProjectionOperator} is computed by the entropy conservative flux from Appendix \ref{sec:App E}. 
Furthermore, the upwind state \eqref{Euler:UpwindState} is used to compute the temporal numerical state functions at the temporal interfaces of the first and last time points. At the interior temporal interfaces the numerical state functions are computed by the state function \eqref{Euler:ECState}. We are interested in measuring the quantity  
\begin{equation}\label{FancyError}
\Xi_{S}\left(T\right):=\Delta_{S}(T)+\sum_{k=1}^{K_{S}}\left.\left\langle \jump{\varPhi\left(\mathbf{U}^{1,k}\right)}-\jump{\mathbf{W}^{1,k}}^{T}\mathbf{u}^{1,k},J^{k}\right\rangle _{\!N}\right|_{\tau=-1},
\end{equation}
where $\Delta_{S}(T)$ is given by \eqref{ErrorTortalEntropy}. In order to reach entropy stability in the sense of the equation \eqref{EC:EntropyPreservation}, the quantity $\Xi_{S}\left(T\right)$ should be zero at final $T=1$. We apply the space-time DGSEM with various different values for the number of temporal cells, number of spatial cells, temporal polynomial degree $M$, spatial polynomial degree $N$, 
and compute the quantity \eqref{FancyError}. The results are given in Table \ref{tab:PreservationTest}, we observe that the quantity \eqref{FancyError} is on the order of machine precision in all the tested configurations varying the number of space-time elements as well as the two polynomial approximation orders. This fits to the theoretical result proven in Theorem 2.       
\begin{center}
\centering
\begin{tabular}{cccc|c}\toprule[1.5pt]
\bf $K_{T}$ &\bf $K_{S}$ & \bf $\mathbb{P}^{M}$ & \bf $\mathbb{P}^{N}$ & \bf $\Xi_{S}\left(T\right)$    \\\midrule
5 & 4  &  3    & 2     & -2.54e-15       \\[0.1cm]
4 & 5  &  2    & 3     & -5.56e-16       \\[0.1cm]
2 & 2  &  3    & 4     & -1.98e-14       \\[0.1cm]
2 & 3  &  6    & 5     &  8.86e-16      \\[0.1cm]
2 & 2  &  5    & 3     &  9.65e-15       \\[0.1cm]
1 & 8  &  6    & 4     &  1.34e-15       \\[0.1cm]
\bottomrule[1.5pt]
\end {tabular}\par
\captionof{table}{The error $\Xi_{S}\left(T\right)$ given by \eqref{FancyError} at final time $T=1$ on uniform grids and for polynomials with various degree in time and space. We observe that the error is always on the order of machine precision regardless of the chosen configuration.}
\label{tab:PreservationTest}
\end{center}

\subsection{Kinetic energy preservation check}\label{sec:KEPCheck}
In this section, we investigate numerically the theoretical findings from Theorem~4. Therefore, we consider the initial value problem 
\begin{equation}\label{LinearWaveTest:Euler}
\rho_{0}=2+\sin\left(2\pi(x-t)\right),\quad,v_{0}=1,\quad p=1.
\end{equation}
As in the previous section, periodic boundary conditions are used in space and Dirichlet Boundary conditions are used in time. The spatial domain is decomposed into $K_{S}$ spatial elements and the temporal domain $\left[0,1\right]$ in $K_{T}$ elements. The space-time DGSEM is applied with the temporal derivative projection operator \eqref{TemporalDerivativeProjectionOperator} computed by \eqref{Euler:ECState} and the spatial derivative projection operator \eqref{SpatialDerivativeProjectionOperator} is computed by the entropy conservative flux from Appendix \ref{sec:App E}. We note that this flux satisfies also Jameson's conditions \eqref{Jameson2} for kinetic energy preservation. Furthermore, the upwind state \eqref{Euler:UpwindState} is used to compute the temporal numerical state functions at the temporal interfaces of the first and last time points. At the interior temporal interfaces the numerical state functions are computed by the state function \eqref{Euler:ECState}. We are interested in measuring the quantity  
\begin{align}\label{FancyError1}
\begin{split}
\Theta_{K}\left(T\right)=&
\bar{K}\left(\mathbf{U}\left(T\right)\right)-\bar{K}\left(\mathbf{U}\left(0\right)\right)-\sum_{n=1}^{K_{T}}\sum_{k=1}^{K_{S}}\frac{\Delta t^{n}}{2}\left\langle \vv{\nabla}_{\xi}\cdot\interpolation{N}{\left(\vv{\tilde{v}}^{n,k}\right)},p^{n,k}\right\rangle _{N\times M} \\
& \qquad \qquad \qquad \qquad \qquad \   -\sum_{n=1}^{K_{T}}\underset{\text{faces}}{\sum_{\text{Interior}}}\frac{\Delta t^{n}}{2}\int\limits _{\partial E^{3},N}\left\langle \avg{p^{n,k}}\jump{\tilde{v}_{\hat{n}}^{n,k}},1\right\rangle _{M}\dS \\
& \qquad \qquad \qquad \qquad \qquad \  -\sum_{k=1}^{K_{S}}\left.\left\langle \jump{\mathbf{V}^{1,k}}^{T}\mathbf{u}^{1,k},J^{k}\right\rangle _{N}\right|_{\tau=-1}, 
\end{split}
\end{align}
where $\bar{K}\left(\mathbf{U}\left(T\right)\right)$ is given by \eqref{DicreteKineticIntegrals1} and $\bar{K}\left(\mathbf{U}\left(0\right)\right)$ is given by \eqref{DicreteKineticIntegrals2}. The quantity $\Theta_{K}\left(T\right)$ should be zero at final $T=1$. We apply the space-time DGSEM with various different values for the number of temporal cells, number of spatial cells, temporal polynomial degree $M$, spatial polynomial degree $N$, 
and compute the quantity \eqref{FancyError}. The results are given in Table \ref{tab:PreservationTestKEP}, we observe that the quantity \eqref{FancyError} is on the order of machine precision in all the tested configurations varying the number of space-time elements as well as the two polynomial approximation orders. This fits to the theoretical result proven in Theorem 4.       
\begin{center}
\centering
\begin{tabular}{cccc|c}\toprule[1.5pt]
\bf $K_{T}$ &\bf $K_{S}$ & \bf $\mathbb{P}^{M}$ & \bf $\mathbb{P}^{N}$ & \bf $\Theta_{K}\left(T\right)$    \\\midrule
5 & 4  &  3    & 2     & 1.04e-15       \\[0.1cm]
4 & 5  &  2    & 3     & 2.42e-15      \\[0.1cm]
2 & 2  &  3    & 4     & 9.82e-16       \\[0.1cm]
2 & 3  &  6    & 5     &  2.44e-15       \\[0.1cm]
2 & 2  &  5    & 3     &  -1.26e-15      \\[0.1cm]
1 & 8  &  6    & 4     &  -5.20e-15       \\[0.1cm]
\bottomrule[1.5pt]
\end {tabular}\par
\captionof{table}{The error $\Theta_{K}\left(T\right)$ given by \eqref{FancyError1} at final time $T=1$ on uniform grids and for polynomials with various degree in time and space. We observe that the error is always on the order of machine precision regardless of the chosen configuration.}
\label{tab:PreservationTestKEP}
\end{center}

\section{Conclusion}\label{sec:conc}

In this work we developed a novel space-time discontinuous Galerkin spectral element approximation for non-linear conservation laws that was entropy stable. On the discrete level we constructed a numerical state for the vector of conservative variables and numerical flux functions such that the chain rule still holds and we could directly mimic the continuous entropy analysis of the system. Here, special attention was given to the temporal term as discrete entropy analysis is typically done on the semi-discrete level, e.g. \cite{carpenter_esdg}. All derivatives in space and time are approximated with a high-order derivative matrix that are summations-by-parts (SBP) operators. The SBP property was required to mimic integration-by-parts to move derivatives back and forth in the discrete variational forms that naturally arose in the DG approximation. The variational forms are approximated with high-order Legendre-Gauss-Lobatto quadrature, which was critical for the entropy stability proof to hold for a general system of conservation laws because there is \textit{no assumption} of exact integration made on the variational forms.

We examined the Euler equations in detail and derived the necessary numerical state functions in time needed by the space-time scheme. Further, we generalized the concept of kinetic energy preservation (KEP), first proposed by Jameson \cite{jameson2008}, to include the influence of the temporal approximation. This meant creating constraints on the form of the numerical state in time and found they were similar to the constraints put on the spatial numerical flux functions. We used the Euler equations in the numerical results to verify the proven properties of the space-time DG scheme. Additionally, the temporal analysis for the shallow water and ideal MHD equations are provided in appendices.


\acknowledgement{Gregor Gassner and Gero Schn\"{u}cke have been supported by the European Research Council (ERC) under the European Union's Eights Framework Program Horizon 2020 with the research project \textit{Extreme}, ERC grant agreement no. 714487. This work was partially performed on the Cologne High Efficiency Operating Platform for Sciences (CHEOPS) at the Regionales Rechenzentrum K\"{o}ln (RRZK) at the University of Cologne.}

\bibliographystyle{spmpsci}
\bibliography{mybibfile}

\begin{thebibliography}{10}
\providecommand{\url}[1]{{#1}}
\providecommand{\urlprefix}{URL }
\expandafter\ifx\csname urlstyle\endcsname\relax
  \providecommand{\doi}[1]{DOI~\discretionary{}{}{}#1}\else
  \providecommand{\doi}{DOI~\discretionary{}{}{}\begingroup
  \urlstyle{rm}\Url}\fi

\bibitem{Barth1999}
Barth, T.J.: Numerical methods for gasdynamic systems on unstructured meshes.
\newblock In: D.~Kr\"{o}ner, M.~Ohlberger, C.~Rohde (eds.) {An Introduction to
  Recent Developments in Theory and Numerics for Conservation Laws},
  \emph{Lecture Notes in Computational Science and Engineering}, vol.~5, pp.
  195--285. Springer Berlin Heidelberg (1999)

\bibitem{bohm2018}
Bohm, M., Winters, A.R., Gassner, G.J., Derigs, D., Hindenlang, F., Saur, J.:
  An entropy stable nodal discontinuous {G}alerkin method for the resistive
  {MHD} equations. {Part I}: {T}heory and numerical verification.
\newblock Journal of Computational Physics
  \textbf{doi.org/10.1016/j.jcp.2018.06.027} (2018)

\bibitem{boom2015}
Boom, P.D., Zingg, D.W.: High-order implicit time-marching methods based on
  generalized summation-by-parts operators.
\newblock {SIAM} Journal on Scientific Computing \textbf{37}(6), A2682--A2709
  (2015)

\bibitem{CHQZ:2006}
Canuto, C., Hussaini, M.Y., Quarteroni, A., Zang, T.A.: Spectral Methods:
  Fundamentals in Single Domains.
\newblock Springer (2006)

\bibitem{carlson1966}
Carlson, B.C.: Some inequalities for hypergeometric functions.
\newblock Proceedings of the American Mathematical Society \textbf{17}(1),
  32--39 (1966)

\bibitem{carpenter_esdg}
Carpenter, M.H., Fisher, T.C., Nielsen, E.J., Frankel, S.H.: Entropy stable
  spectral collocation schemes for the {N}avier--{S}tokes equations:
  Discontinuous interfaces.
\newblock SIAM Journal on Scientific Computing \textbf{36}(5), B835--B867
  (2014)

\bibitem{chan2018}
Chan, J.: On discretely entropy conservative and entropy stable discontinuous
  {G}alerkin methods.
\newblock Journal of Computational Physics \textbf{362}, 346--374 (2018)

\bibitem{Chandrashekar2012}
Chandrashekar, P.: Kinetic energy preserving and entropy stable finite volume
  schemes for compressible {E}uler and {N}avier-{S}tokes equations.
\newblock Communications in Computational Physics \textbf{14}, 1252--1286
  (2013)

\bibitem{Chandrashekar2015}
{Chandrashekar}, P., Klingenberg, C.: Entropy stable finite volume scheme for
  ideal compressible {MHD} on 2-{D} cartesian meshes.
\newblock {SIAM} Journal of Numerical Analysis \textbf{54}(2), 1313--1340
  (2016)

\bibitem{Chen2017}
Chen, T., Shu, C.W.: Entropy stable high order discontinuous {G}alerkin methods
  with suitable quadrature rules for hyperbolic conservation laws.
\newblock Journal of Computational Physics \textbf{345}, 427--461 (2017)

\bibitem{crean2018}
Crean, J., Hicken, J.E., Fern{\'a}ndez, D.c.D.R., Zingg, D.W., Carpenter, M.H.:
  Entropy-stable summation-by-parts discretization of the {E}uler equations on
  general curved elements.
\newblock Journal of Computational Physics \textbf{356}, 410--438 (2018)

\bibitem{Fernandez2014}
{Del Rey Fern\'andez}, D.C., Hicken, J.E., Zingg, D.W.: Review of
  summation-by-parts operators with simultaneous approximation terms for the
  numerical solution of partial differential equations.
\newblock Computers \& Fluids \textbf{95}(22), 171--196 (2014)

\bibitem{Derigs2017}
Derigs, D., Winters, A.R., Gassner, G.J., Walch, S., Bohm, M.: Ideal {GLM-MHD}:
  About the entropy consistent nine-wave magnetic field divergence diminishing
  ideal magnetohydrodynamics equations.
\newblock Journal of Computational Physcis \textbf{364}, 420--467 (2018)

\bibitem{diosady2015}
Diosady, L.T., Murman, S.M.: Higher-order methods for compressible turbulent
  flows using entropy variables.
\newblock In: 53rd {AIAA} Aerospace Science Meeting, p. 0294 (2015)

\bibitem{dutt1988}
Dutt, P.: Stable boundary conditions and difference schemes for
  {N}avier-{S}tokes equations.
\newblock {SIAM} Journal of Numerical Analysis \textbf{25}(2), 245--267 (1988)

\bibitem{fisher2013}
Fisher, T.C., Carpenter, M.H.: High-order entropy stable finite difference
  schemes for nonlinear conservation laws: {F}inite domains.
\newblock Journal of Computational Physics \textbf{252}, 518--557 (2013)

\bibitem{fisher2013_2}
Fisher, T.C., Carpenter, M.H., Nordstr\"{o}m, J., Yamaleev, N.K., Swanson, C.:
  Discretely conservative finite-difference formulations for nonlinear
  conservation laws in split form: {T}heory and boundary conditions.
\newblock Journal of Computational Physics \textbf{234}, 353--375 (2013)

\bibitem{Fjordholm2011}
Fjordholm, U.S., Mishra, S., Tadmor, E.: Well-blanaced and energy stable
  schemes for the shallow water equations with discontiuous topography.
\newblock Journal of Computational Physics \textbf{230}(14), 5587--5609 (2011).
\newblock \doi{10.1016/j.jcp.2011.03.042}

\bibitem{fjordholm2012}
Fjordholm, U.S., Mishra, S., Tadmor, E.: {Arbitrarily High-order Accurate
  Entropy Stable Essentially Nonoscillatory Schemes for Systems of Conservation
  Laws}.
\newblock SIAM Journal on Numerical Analysis \textbf{50}(2), 544--573 (2012).
\newblock \doi{10.1137/110836961}

\bibitem{flad2017}
Flad, D., Gassner, G.J.: On the use of kinetic energy preserving {DG}-schemes
  for large eddy simulations.
\newblock {Journal of Computational Physics} \textbf{350}, 782--795 (2017)

\bibitem{gassner2015space}
Gassner, G., Staudenmaier, M., Hindenlang, F., Atak, M., Munz, C.D.: A
  space--time adaptive discontinuous {G}alerkin scheme.
\newblock Computers \& Fluids \textbf{117}, 247--261 (2015)

\bibitem{gassner_skew_burgers}
Gassner, G.J.: A skew-symmetric discontinuous {Galerkin} spectral element
  discretization and its relation to {SBP-SAT} finite difference methods.
\newblock SIAM Journal on Scientific Computing \textbf{35}(3), A1233--A1253
  (2013)

\bibitem{Gassner2017}
Gassner, G.J., Winters, A.R., Hindenlang, F.J., Kopriva, D.A.: The {BR1} scheme
  is stable for the compressible {N}avier-{S}tokes equations.
\newblock {Journal of Scientific Computing}
  \textbf{doi.org/10.1007/s10915-018-0702-1} (2018)

\bibitem{Gassner:2016ye}
Gassner, G.J., Winters, A.R., Kopriva, D.A.: Split form nodal discontinuous
  {G}alerkin schemes with summation-by-parts property for the compressible
  {E}uler equations.
\newblock Journal of Computational Physics \textbf{327}, 39--66 (2016)

\bibitem{gassner2015}
Gassner, G.J., Winters, A.R., Kopriva, D.A.: A well balanced and entropy
  conservative discontinuous {G}alerkin spectral element method for the shallow
  water equations.
\newblock Applied Mathematics and Computation \textbf{272}(2), 291--308 (2016)

\bibitem{harten1983}
Harten, A.: On the symmetric form of systems of conservation laws with entropy.
\newblock {Journal of Computational Physics} \textbf{49}, 151--164 (1983)

\bibitem{IsmailRoe2009}
Ismail, F., Roe, P.L.: Affordable, entropy-consistent {E}uler flux functions
  {II}: Entropy production at shocks.
\newblock Journal of Computational Physics \textbf{228}(15), 5410--5436 (2009)

\bibitem{jameson2008}
Jameson, A.: Formulation of kinetic energy preserving conservative schemes for
  gas dynamics and direct numerical simulation of one-dimensional viscous
  compressible flow in a shock tube using entropy and kinetic energy preserving
  schemes.
\newblock Journal of Scientific Computing \textbf{34}(3), 188--208 (2008)

\bibitem{kopriva2006metric}
Kopriva, D.A.: Metric identities and the discontinuous spectral element method
  on curvilinear meshes.
\newblock The Journal of Scientific Computing \textbf{26}(3), 301--327 (2006)

\bibitem{Kopriva:2009nx}
Kopriva, D.A.: Implementing Spectral Methods for Partial Differential
  Equations.
\newblock Scientific Computation. Springer (2009)

\bibitem{kreiss1}
Kreiss, H.O., Olliger, J.: Comparison of accurate methods for the integration
  of hyperbolic equations.
\newblock Tellus \textbf{24}, 199--215 (1972)

\bibitem{lefloch2000}
LeFloch, P.G., Rohde, C.: High-order schemes, entropy inequalities, and
  nonclassical shocks.
\newblock {SIAM} Journal on Numerical Analysis \textbf{37}(6), 2023--2060
  (2000)

\bibitem{leveque2002}
Le{V}eque, R.J.: Finite Volume Methods for Hyperbolic Problems, vol.~31.
\newblock Cambridge University Press (2002)

\bibitem{Liu2017}
Liu, Y., Shu, C.W., Zhang, M.: Entropy stable high order discontinuous
  {G}alerkin methods for ideal compressible {MHD} on structured meshes.
\newblock Journal of Computational Physics \textbf{354}, 163--178 (2017)

\bibitem{lundquist2014}
Lundquist, T., Nordstr\"{o}m, J.: The {SBP}-{SAT} technique for initial value
  problems.
\newblock Journal of Computational Physics \textbf{270}, 86--104 (2014)

\bibitem{mock1980}
Mock, M.S.: Systems of conservation laws of mixed type.
\newblock Journal of Differential Equations \textbf{37}(1), 70--88 (1980)

\bibitem{moura2017}
Moura, R.C., Mengaldo, G., Peiro, J., Sherwin, S.J.: An {LES} setting for
  {DG}-based implicit {LES} with insights on dissipation and robustness.
\newblock In: Spectral and High Order Methods for Partial Differential
  Equations {ICOSAHOM} 2016, pp. 161--173. Springer (2017)

\bibitem{nordstrom2013}
Nordstr\"{o}m, J., Lundquist, T.: Summation-by-parts in time.
\newblock Journal of Computational Physics \textbf{251}, 487--499 (2013)

\bibitem{pirozzoli2011}
Pirozzoli, S.: Numerical methods for high-speed flows.
\newblock Annual Review of Fluid Mechanics \textbf{43}, 163--194 (2011)

\bibitem{Powell1999}
Powell, K.G., Roe, P.L., Linde, T.J., Gombosi, T.I., Zeeuw, D.L.D.: A
  solution-adaptive upwind scheme for ideal magnetohydrodynamics.
\newblock Journal of Computational Physics \textbf{154}, 284--309 (1999)

\bibitem{ranocha2018}
Ranocha, H.: Generalised summation-by-parts operators and entropy stability of
  numerical methods for hyperbolic balance laws.
\newblock Ph.D. thesis, TU~Braunschweig (2018)

\bibitem{tadmor1984}
Tadmor, E.: Skew-selfadjoint form for systems of conservation laws.
\newblock Journal of Mathematical Analysis and Applications \textbf{103}(2),
  428--442 (1984)

\bibitem{tadmor1987}
Tadmor, E.: The numerical viscosity of entropy stable schemes for systems of
  conservation laws.
\newblock Mathematics of Computation \textbf{49}(179), 91--103 (1987).
\newblock \doi{10.2307/2008251}

\bibitem{tadmor2003}
Tadmor, E.: Entropy stability theory for difference approximations of nonlinear
  conservation laws and related time-dependent problems.
\newblock Acta Numerica \textbf{12}, 451--512 (2003)

\bibitem{wintermeyer2017}
Wintermeyer, N., Winters, A.R., Gassner, G.J., Kopriva, D.A.: An entropy stable
  discontinuous {G}alerkin method for the two dimensional shallow water
  equations with discontinuous bathymetry.
\newblock {Journal of Computational Physics} \textbf{340}, 200--242 (2017)

\end{thebibliography}

\appendix

\section{Shallow Water State Values in Time}\label{sec:App SW}

The shallow water equations are
\begin{equation}
\pderivative{}{t}\begin{bmatrix}
h\\[0.1cm]
h\vv{v}\\[0.1cm]
\end{bmatrix}
+
\vv{\nabla}_x\cdot\begin{bmatrix}
h\vv{v}\\[0.1cm]
h\vv{v}\vv{v}^T + \frac{g}{2}h^2\matx{I}\\[0.1cm]
\end{bmatrix}
=\begin{bmatrix}
0\\[0.1cm]
\vv{0}\\[0.1cm]
\end{bmatrix},
\end{equation}
where $h$ is the water height, $\vv{v} = (v_1,v_2)$ are the fluid velocities, $g$ is the gravitational constant, and $\matx{I}$ is the $2\times2$ identity matrix. 
Here we ignore the bottom topography, $b$, as it is constant in time. Thus, any jumps that arise in the current analysis with simply vanish. However, complete details on the development of the spacial entropy stability analysis for the shallow water equations \textit{with} a bottom topography can be found in, e.g., \cite{Fjordholm2011,gassner2015,wintermeyer2017}.

Here we focus on the temporal entropy analysis. The condition to design a discretely entropy conservative temporal state is \eqref{DiscreteEntropyConservation}. We collect the necessary shallow water quantities for the discrete temporal entropy analysis
\begin{equation}
\begin{aligned}
\vec{U} &= (h\,,\,hv_1\,,\,hv_2)^T,\\[0.1cm]
\vec{W} &= \left(gh-\frac{1}{2}\left(v_1^2+v_2^2\right)\,,\,v_1\,,\,v_2\right)^T,\\[0.1cm]
s &= \frac{hv^2}{2} + \frac{gh^2}{2} + ghb,\\[0.1cm]
\Phi &= \frac{gh^2}{2},\\[0.1cm]
\end{aligned}
\end{equation} 
with the gravitational constant $g$. Now, from the properties of the jump operator \eqref{eq:jumpProperties} we compute the components of $\vec{U}^{\#}$ from 
\begin{equation}
\jump{\vec{W}}^T\vec{U}^{\#} = g\jump{h}U_1^\# - \avg{v_1}\jump{v_1}U_1^\#- \avg{v_2}\jump{v_2}U_1^\# + \jump{v_1}U_2^\# + \jump{v_2}U_3^{\#} = g\avg{h}\jump{h}.
\end{equation}
Solving we determine the entropy conservative temporal state to be
\begin{equation}
\vec{U}^\# = \begin{bmatrix}
\avg{h}\\[0.1cm]
\avg{h}\avg{v_1}\\[0.1cm]
\avg{h}\avg{v_2}\\[0.1cm]
\end{bmatrix}.
\end{equation}
Is it easy to verify that the temporal state $\vec{U}^{\#}$ is consistent and symmetric with respect to its arguments.

\subsection{Upwind temporal state for shallow water equations}\label{sec:ShallowWater}
Just as in the Euler case, we want to use the upwind temporal state in practice to make the space-time DG numerical scheme computationally tractable. As such, we want to show that the upwind state in time is entropy stable. For the shallow this means we must satisfy the condition
\begin{equation}\label{eq:condintionL}
\jump{\vec{W}}^T\vec{U}^* \leq \jump{\Phi} = \jump{\frac{gh^2}{2}} = g\avg{h}\jump{h}.
\end{equation}
The upwind temporal state is given by
\begin{equation}\label{eq:upwindSW}
\vec{U}^* = 
\begin{bmatrix}
h_{-}\\[0.1cm]
h_{-}\voneM\\[0.1cm]
h_{-}\vtwoM\\[0.1cm]
\end{bmatrix}.
\end{equation}
From the upwind ansatz \eqref{eq:upwindSW} it is straightforward to show that \eqref{eq:condintionL} is satisfied
\begin{equation}
\begin{aligned}
\jump{\vec{W}}^T\vec{U}^* &= gh_{-}\jump{h}-h_{-}\avg{v_1}\jump{v_1}-h_{-}\avg{v_2}\jump{v_2} + h_{-}\voneM\jump{v_1}+ h_{-}\vtwoM\jump{v_2}\\
&= gh_{-}\jump{h} - h_{-}\jump{v_1}\left(\avg{v_1}-\voneM\right) - h_{-}\jump{v_2}\left(\avg{v_2}-\vtwoM\right)\\
&= gh_{-}\jump{h} - \frac{h_{-}}{2}\left(\jump{v_1}\right)^2 - \frac{h_{-}}{2}\left(\jump{v_2}\right)^2 \pm \frac{g}{2}h_{+}\jump{h}\\
&= g\avg{h}\jump{h} + \frac{g}{2}h_{-}\jump{h} - \frac{g}{2}h_{+}\jump{h} - \frac{h_{-}}{2}\left(\jump{v_1}\right)^2 - \frac{h_{-}}{2}\left(\jump{v_2}\right)^2\\
&= g\avg{h}\jump{h} - \frac{g}{2}\left(\jump{h}\right)^2 - \frac{h_{-}}{2}\left(\jump{v_1}\right)^2 - \frac{h_{-}}{2}\left(\jump{v_2}\right)^2 \\
&\leq g\avg{h}\jump{h} 
\end{aligned}
\end{equation}
where we use the identity \eqref{eq:avgProperty} twice, add and subtract a scaled value of $\jump{h}$, and use the physical assumption of positive water height, $h>0$. Thus, the upwind temporal state for the shallow water equations is entropy stable.

\section{Ideal magnetohydrodynamics time state evaluation}\label{sec:App B}

Here we consider a slightly modified version of the ideal magnetohydrodynamic (MHD) equations 
\begin{equation}\label{eq:modifiedMHD}
\pderivative{}{t}\begin{bmatrix}
\rho\\[0.2cm]
\rho\vv{v}\\[0.2cm]
E\\[0.2cm]
\vv{B}\\[0.2cm]
\end{bmatrix}
+
\vv{\nabla}_x\cdot\begin{bmatrix}
\rho\vv{v}\\[0.2cm]
\rho\vv{v}\vv{v}^T + \left(p + \frac{1}{2}|\vv{B}|^2\right)\matx{I} - \vv{B}\vv{B}^T\\[0.2cm]
\vv{v}\left(E+p+\frac{1}{2}|\vv{B}|^2\right)-\vv{B}\left(\vv{v}\cdot\vv{B}\right)\\[0.2cm]
\vv{v}\vv{B}^T - \vv{B}\vv{v}^T\\[0.2cm]
\end{bmatrix}
=
\left(\vv{\nabla}_x\cdot\vv{B}\right)\begin{bmatrix}
0\\[0.2cm]
\vv{B}\\[0.2cm]
\vv{v}\cdot\vv{B}\\[0.2cm]
\vv{v}\\[0.2cm]
\end{bmatrix}
=
(\vv{\nabla}_x\cdot\vv{B})\boldsymbol\Upsilon,
\end{equation}
where $\rho$ is the density, $\vv{v} = (v_1\,,\,v_2\,,\,v_3)^T$ are the fluid velocities, $\vv{B}=(B_1\,,\,B_2\,,\,B_3)^T$ are the magnetic fields, $E$ is the total energy, and $\matx{I}$ is the $3\times3$ identity matrix. With the assumption of an ideal gas the pressure is
\begin{equation}
p = (\gamma-1)\left(E - \frac{\rho}{2}|\vv{v}|^2-\frac{1}{2}|\vv{B}|^2\right).
\end{equation}
Also, the divergence-free constraint for magnetized fluids is incorporated into the model \eqref{eq:modifiedMHD} as a non-conservative term \cite{Barth1999,Powell1999}. This non-conservative term is an important aspect of the spatial entropy analysis particularly when $\vv{\nabla}_x\cdot\vv{B}\ne 0$ see, e.g. \cite{bohm2018,Chandrashekar2015,Derigs2017,Liu2017}. We address how this non-conservative term proportional to $\vv{\nabla}_x\cdot\vv{B}$ affects the spatial part of the DG approximation in Appendix \ref{sec:App B3}.

For now we concern ourselves with the temporal entropy analysis and collect the necessary quantities for the ideal MHD equations
\begin{equation}
\begin{aligned}
\vec{U} &= (\rho\,,\,\rho v_1\,,\,\rho v_2\,,\,\rho v_3\,,\,E\,,\,B_1\,,\,B_2\,,\,B_3)^T,\\[0.1cm]
\vec{W} &= \left(\frac{\gamma-\varsigma}{\gamma-1}-\beta |\vv{v}|^2\,,\,2\beta v_1\,,\,2\beta v_2\,,\,2\beta v_3\,,\,-2\beta\,,\,2\beta B_1\,,\,2\beta B_2\,,\,2\beta B_3\right)^T,\\[0.1cm]
s &= -\frac{\rho \varsigma}{\gamma-1},\\[0.1cm]
\Phi &= \rho + \beta\left|\vec{B}\right|^2\\[0.1cm]
\end{aligned}
\end{equation} 
where the physical entropy $\varsigma$ and $\beta$ are defined as in the Euler case in Section \ref{sec:EulerStuff}.

To determine an entropy conservative temporal state we must find a two-point function $\vec{u}^\#$ such that
\begin{equation}\label{eq:ECMHD}
\jump{\vec{W}}^T\vec{U}^\# = \jump{\Phi}= \jump{\rho + \beta\left|\vec{B}\right|^2}.
\end{equation}
To do so, we first compute the jump in the entropy variables
\begin{equation}\label{eq:entjumpMHD}
\jump{\vec{W}} = \begin{bmatrix}
\frac{\jump{\rho}}{\rho^{\ln}}+\frac{\jump{\beta}}{\beta^{\ln}(\gamma-1)} - 2\avg{\beta}\left(\avg{v_1}\jump{v_1}+\avg{v_2}\jump{v_2}+\avg{v_3}\jump{v_3}\right)-\jump{\beta}\left(\avg{v_1^2}+\avg{v_2^2}+\avg{v_3^2}\right)\\[0.1cm]
2\avg{v_1}\jump{\beta}+2\avg{\beta}\jump{v_1}\\[0.1cm]
2\avg{v_2}\jump{\beta}+2\avg{\beta}\jump{v_2}\\[0.1cm]
2\avg{v_3}\jump{\beta}+2\avg{\beta}\jump{v_3}\\[0.1cm]
-2\jump{\beta}\\[0.1cm]
2\avg{B_1}\jump{\beta}+2\avg{\beta}\jump{B_1}\\[0.1cm]
2\avg{B_2}\jump{\beta}+2\avg{\beta}\jump{B_2}\\[0.1cm]
2\avg{B_3}\jump{\beta}+2\avg{\beta}\jump{B_3}\\[0.1cm]
\end{bmatrix}.
\end{equation}
Then we compute the left hand side of \eqref{eq:ECMHD} to be
\begin{equation}
\begin{aligned}
\jump{\vec{W}}^T\vec{U}^\# =& U_1^\#\left(\frac{\jump{\rho}}{\rho^{\ln}}+\frac{\jump{\beta}}{\beta^{\ln}(\gamma-1)} - 2\avg{\beta}\left(\avg{v_1}\jump{v_1}+\avg{v_2}\jump{v_2}+\avg{v_3}\jump{v_3}\right)-\jump{\beta}\left(\avg{v_1^2}+\avg{v_2^2}+\avg{v_3^2}\right)\right)\\
&+U_2^\#\left(2\avg{v_1}\jump{\beta}+2\avg{\beta}\jump{v_1}\right) + U_3^\#\left(2\avg{v_2}\jump{\beta}+2\avg{\beta}\jump{v_2}\right) + U_4^\#\left(2\avg{v_3}\jump{\beta}+2\avg{\beta}\jump{v_3}\right)\\
&-2U_5^\#\jump{\beta} + U_6^\#\left(2\avg{B_1}\jump{\beta}+2\avg{\beta}\jump{B_1}\right) + U_7^\#\left(2\avg{B_2}\jump{\beta}+2\avg{\beta}\jump{B_2}\right)\\
&+ U_8^\#\left(2\avg{B_3}\jump{\beta}+2\avg{\beta}\jump{B_3}\right).
\end{aligned}
\end{equation}
Next, we expand the right hand side of \eqref{eq:ECMHD} 
\begin{equation}
\jump{\rho + \beta \left|\vec{B}\right|^2} = \jump{\rho} + \left(\avg{B_1^2}+\avg{B_2^2}+\avg{B_3^2}\right)\jump{\beta} + 2\avg{\beta}\left(\avg{B_1}\jump{B_1}+\avg{B_2}\jump{B_2}+\avg{B_3}\jump{B_3}\right).
\end{equation}
Next, we collect the individual jump terms to create conditions on the entropy conservative temporal state function $\vec{U}^\#$:
\begin{equation}
\begin{aligned}
\jump{\rho}:\quad& \frac{U_1^\#}{\rho^{\ln}} = 1\\[0.1cm]
\jump{v_1}:\quad& -2\avg{\beta}\avg{v_1}U_1^\# + 2\avg{\beta}U_2^\# = 0\\[0.1cm]
\jump{v_2}:\quad& -2\avg{\beta}\avg{v_2}U_1^\# + 2\avg{\beta}U_3^\# = 0\\[0.1cm]
\jump{v_3}:\quad& -2\avg{\beta}\avg{v_3}U_1^\# + 2\avg{\beta}U_4^\# = 0\\[0.1cm]
\jump{B_1}:\quad& 2\avg{\beta}U_6^\# = 2\avg{\beta}\avg{B_1}\\[0.1cm]
\jump{B_2}:\quad& 2\avg{\beta}U_7^\# = 2\avg{\beta}\avg{B_2}\\[0.1cm]
\jump{B_3}:\quad& 2\avg{\beta}U_8^\# = 2\avg{\beta}\avg{B_3}\\[0.1cm]
\jump{\beta}:\quad&U_1^\#\left(\frac{1}{\beta^{\ln}(\gamma-1)}-\avg{v_1^2}-\avg{v_2^2}-\avg{v_3^2}\right)+2U_2^\#\avg{v_1}+2U_3^\#\avg{v_2}+2U_4^\#\avg{v_3}\\
& -2U_5^\# + 2U_6^\#\avg{B_1}+ 2U_7^\#\avg{B_2}+ 2U_8^\#\avg{B_3} = \avg{B_1^2}+\avg{B_2^2}+\avg{B_3^2}
\end{aligned}
\end{equation}
It is then straightforward to determine
\begin{equation}
\vec{U}^\# = \begin{bmatrix}
\rho^{\ln}\\[0.1cm]
\rho^{\ln}\avg{v_1}\\[0.1cm]
\rho^{\ln}\avg{v_2}\\[0.1cm]
\rho^{\ln}\avg{v_3}\\[0.1cm]
U_5^\#\\[0.1cm]
\avg{B_1}\\[0.1cm]
\avg{B_2}\\[0.1cm]
\avg{B_3}\\[0.1cm]
\end{bmatrix},
\end{equation}
with the fifth term
\begin{equation}
\begin{aligned}
U_5^\# = \frac{\rho^{\ln}}{2\beta^{\ln}(\gamma-1)}&+\rho^{\ln}\left(\avg{v_1}^2+\avg{v_2}^2+\avg{v_3}^2-\frac{1}{2}\left(\avg{v_1^2}+\avg{v_2^2}+\avg{v_3^2}\right)\right)\\
&+\avg{B_1}^2+\avg{B_2}^2+\avg{B_3}^2-\frac{1}{2}\left(\avg{B_1^2}+\avg{B_2^2}+\avg{B_3^2}\right).
\end{aligned}
\end{equation}
We see that the function $\vec{U}^{\#}$ is symmetric with respect to its arguments and is consistent to the state vector if the left and right states are the same as
\begin{equation}
\vec{U}^{\#} = \begin{bmatrix}
\rho\\[0.1cm]
\rho v_1\\[0.1cm]
\rho v_2\\[0.1cm]
\rho v_3\\[0.1cm]
\frac{\rho}{2\beta(\gamma-1)} + \frac{\rho}{2}\|\vv{v}\|^2+\frac{1}{2}\|\vv{B}\|^2\\[0.1cm]
B_1\\[0.1cm]
B_2\\[0.1cm]
B_3\\[0.1cm]
\end{bmatrix}
= \begin{bmatrix}
\rho\\[0.1cm]
\rho v_1\\[0.1cm]
\rho v_2\\[0.1cm]
\rho v_3\\[0.1cm]
\frac{p}{\gamma-1} + \frac{\rho}{2}\|\vv{v}\|^2+\frac{1}{2}\|\vv{B}\|^2\\[0.1cm]
B_1\\[0.1cm]
B_2\\[0.1cm]
B_3\\[0.1cm]
\end{bmatrix}
= \begin{bmatrix}
\rho\\[0.1cm]
\rho v_1\\[0.1cm]
\rho v_2\\[0.1cm]
\rho v_3\\[0.1cm]
E\\[0.1cm]
B_1\\[0.1cm]
B_2\\[0.1cm]
B_3\\[0.1cm]
\end{bmatrix}
=
\vec{U}.
\end{equation}
Again, the entropy conservative temporal state fully couples all time levels. So, to make the entropy stable space-time DG scheme computationally attractive we next demonstrate that the upwind temporal state for the ideal MHD is entropy stable in time.

\subsection{Upwind temporal ideal MHD state}\label{sec:UpwindMHD}

To demonstrate entropy stability of the upwind temporal state the condition \eqref{eq:EulerCondition} becomes 
\begin{equation}\label{eq:MHDEnt}
\jump{\vec{W}}^T\vec{U}^* \leq \jump{\Phi}= \jump{\rho + \left|\vec{B}\right|^2},
\end{equation}
with the upwind flux taken to be
\begin{equation}
\vec{U}^* = \begin{bmatrix}
\rhoM\\[0.1cm]
\rhoM\voneM\\[0.1cm]
\rhoM\vtwoM\\[0.1cm]
\rhoM\vthrM\\[0.1cm]
E_{-}\\[0.1cm]
\BoneM\\[0.1cm]
\BtwoM\\[0.1cm]
\BthrM\\[0.1cm]
\end{bmatrix},\quad E_{-} = \frac{\rhoM}{2\betaM(\gamma-1)} + \frac{\rhoM}{2}\left(\voneM^2+\vtwoM^2+\vthrM^2\right) +  \frac{1}{2}\left(\BoneM^2+\BtwoM^2+\BthrM^2\right).
\end{equation}
We already have the jump in entropy variables \eqref{eq:entjumpMHD} so we can immediately compute the left hand side
\begin{equation}
\begin{aligned}
\jump{\vec{W}}^T\vec{U}^* =& \rhoM\left(\frac{\jump{\rho}}{\rho^{\ln}}+\frac{\jump{\beta}}{\beta^{\ln}(\gamma-1)} - 2\avg{\beta}\left(\avg{v_1}\jump{v_1}+\avg{v_2}\jump{v_2}+\avg{v_3}\jump{v_3}\right)-\jump{\beta}\left(\avg{v_1^2}+\avg{v_2^2}+\avg{v_3^2}\right)\right)\\
&+\rhoM\voneM\left(2\avg{v_1}\jump{\beta}+2\avg{\beta}\jump{v_1}\right) 
+ \rhoM\vtwoM\left(2\avg{v_2}\jump{\beta}+2\avg{\beta}\jump{v_2}\right) 
+ \rhoM\vthrM\left(2\avg{v_3}\jump{\beta}+2\avg{\beta}\jump{v_3}\right)\\
&-2E_{-}\jump{\beta} + \BoneM\left(2\avg{B_1}\jump{\beta}+2\avg{\beta}\jump{B_1}\right) + \BtwoM\left(2\avg{B_2}\jump{\beta}+2\avg{\beta}\jump{B_2}\right)\\
&+ \BthrM\left(2\avg{B_3}\jump{\beta}+2\avg{\beta}\jump{B_3}\right).
\end{aligned}
\end{equation}
Just as before we group together each of the jumps to make it easier to see how terms simplify. Further, we reuse \eqref{eq:avgProperty} in the momentum equations as well as add and subtract one from the density jump condition to facilitate later manipulations:
\begin{equation}
\begin{aligned}
\jump{\rho}:\quad& \frac{\rhoM}{\rho^{\ln}}\pm 1 = 1-\frac{1}{\rho^{\ln}}(\rho^{\ln} - \rhoM)\\[0.1cm]
\jump{v_1}:\quad& -2\avg{\beta}\avg{v_1}\rhoM + 2\avg{\beta}\rhoM\voneM =  -2\rhoM\avg{\beta}(\avg{v_1}-\voneM)=-\rhoM\avg{\beta}\jump{v_1}\\[0.1cm]
\jump{v_2}:\quad& -2\avg{\beta}\avg{v_2}\rhoM + 2\avg{\beta}\rhoM\vtwoM = -2\rhoM\avg{\beta}(\avg{v_2}-\vtwoM)=-\rhoM\avg{\beta}\jump{v_2}\\[0.1cm]
\jump{v_3}:\quad& -2\avg{\beta}\avg{v_3}\rhoM + 2\avg{\beta}\rhoM\vthrM = -2\rhoM\avg{\beta}(\avg{v_3}-\vthrM) = -\rhoM\avg{\beta}\jump{v_3}\\[0.1cm]
\jump{B_1}:\quad& 2\avg{\beta}\BoneM\\[0.1cm]
\jump{B_2}:\quad& 2\avg{\beta}\BtwoM\\[0.1cm]
\jump{B_3}:\quad& 2\avg{\beta}\BthrM\\[0.1cm]
\jump{\beta}:\quad& \rhoM\left(\frac{1}{\beta^{\ln}(\gamma-1)}-\avg{v_1^2}-\avg{v_2^2}-\avg{v_3^2}\right)+2\rhoM\voneM\avg{v_1}+2\rhoM\vtwoM\avg{v_2}+2\rhoM\vthrM\avg{v_3}\\
& -\left(\frac{\rhoM}{\betaM(\gamma-1)} + \rhoM\left(\voneM^2+\vtwoM^2+\vthrM^2\right) +  \left(\BoneM^2+\BtwoM^2+\BthrM^2\right)\right) \\
&+ 2\BoneM\avg{B_1}+ 2\BtwoM\avg{B_2}+ 2\BthrM\avg{B_3}\\
=&\frac{\rhoM}{\beta^{\ln}\betaM(\gamma-1)}(\betaM-\beta^{\ln})+\rhoM\left(-\avg{v_1^2}-\avg{v_2^2}-\avg{v_3^2}+2\voneM\avg{v_1}+2\vtwoM\avg{v_2}+2\vthrM\avg{v_3}\right.\\
&\left.-\voneM^2-\vtwoM^2-\vthrM^2\right) - \BoneM^2-\BtwoM^2-\BthrM^2+ 2\BoneM\avg{B_1}+ 2\BtwoM\avg{B_2}+ 2\BthrM\avg{B_3}
\end{aligned}
\end{equation}
We reuse the manipulation on each velocity field \eqref{eq:velManip} from the Euler case as well as some simple cancellation on the magnetic field terms to further simplify the $\jump{\beta}$ term when substituting $E_{-}$ to find
\begin{equation}
\jump{\beta}:\quad -\frac{\rhoM}{\beta^{\ln}\betaM(\gamma-1)}(\beta^{\ln}-\betaM)-\frac{\rhoM}{2}\left(\left(\jump{v_1}\right)^2+\left(\jump{v_2}\right)^2+\left(\jump{v_3}\right)^2\right)
+ \BoneM\BoneP+ \BtwoM\BtwoP+ \BthrM\BthrP.
\end{equation}
Now we revisit the complete left hand side of \eqref{eq:MHDEnt}
\begin{equation}
\begin{aligned}
\jump{\vec{W}}^T\vec{U}^* =& \jump{\rho}-\frac{\jump{\rho}}{\rho^{\ln}}(\rho^{\ln}-\rhoM)-\frac{\rhoM\jump{\beta}}{\beta^{\ln}\betaM(\gamma-1)}(\beta^{\ln}-\betaM) - \rhoM\left(\left(\jump{v_1}\right)^2+\left(\jump{v_2}\right)^2+\left(\jump{v_3}\right)^2\right)\left(\avg{\beta}+\frac{1}{2}\jump{\beta}\right)\\
&+\jump{\beta}\left(\BoneM\BoneP+ \BtwoM\BtwoP+ \BthrM\BthrP\right) + 2\avg{\beta}\left(\BoneM\jump{B_1}+\BtwoM\jump{B_2}+\BthrM\jump{B_3}\right).
\end{aligned}
\end{equation}
The first line simplifies as we reuse the results \eqref{eq:betaAvgProp} and \eqref{eq:logSimplify} from the temporal entropy stability proof from the Euler case to find
\begin{equation}
\begin{aligned}
\jump{\vec{W}}^T&\vec{U}^* = \jump{\rho}-\frac{\left(\jump{\rho}\right)^2}{2\rho^{\ln}}-\frac{\rhoM\left(\jump{\beta}\right)^2}{\beta^{\ln}\betaM(\gamma-1)} - \rhoM\betaP\left(\left(\jump{v_1}\right)^2+\left(\jump{v_2}\right)^2+\left(\jump{v_3}\right)^2\right)\\
&+\jump{\beta}\left(\BoneM\BoneP+ \BtwoM\BtwoP+ \BthrM\BthrP\right) + 2\avg{\beta}\left(\BoneM\jump{B_1}+\BtwoM\jump{B_2}+\BthrM\jump{B_3}\right).
\end{aligned}
\end{equation}

All that remains is the handle the magnetic field terms which requires adding zero in a clever way
\begin{equation}
\resizebox{\textwidth}{!}{$
\begin{aligned}
\jump{\vec{W}}^T\vec{U}^* =&\jump{\rho}-\frac{\left(\jump{\rho}\right)^2}{2\rho^{\ln}}-\frac{\rhoM\left(\jump{\beta}\right)^2}{\beta^{\ln}\betaM(\gamma-1)} - \rhoM\betaP\left(\left(\jump{v_1}\right)^2+\left(\jump{v_2}\right)^2+\left(\jump{v_3}\right)^2\right)\\
&+\jump{\beta}\left(\BoneM\BoneP+ \BtwoM\BtwoP+ \BthrM\BthrP\right) + 2\avg{\beta}\left(\BoneM\jump{B_1}+\BtwoM\jump{B_2}+\BthrM\jump{B_3}\right)\\
&\pm\avg{\beta}\left(\BoneP\jump{B_1}+\BtwoP\jump{B_2}+\BthrP\jump{B_3}\right)\\
=&\jump{\rho}-\frac{\left(\jump{\rho}\right)^2}{2\rho^{\ln}}-\frac{\rhoM\left(\jump{\beta}\right)^2}{\beta^{\ln}\betaM(\gamma-1)} - \rhoM\betaP\left(\left(\jump{v_1}\right)^2+\left(\jump{v_2}\right)^2+\left(\jump{v_3}\right)^2\right)\\
& + 2\avg{\beta}\left(\avg{B_1}\jump{B_1}+\avg{B_2}\jump{B_2}+\avg{B_3}\jump{B_3}\right)-\avg{\beta}\left(\left(\jump{B_1}\right)^2+\left(\jump{B_2}\right)^2+\left(\jump{B_3}\right)^2\right)\\
&+\jump{\beta}\left(\BoneM\BoneP+ \BtwoM\BtwoP+ \BthrM\BthrP\right)\\
=& \jump{\rho}-\frac{\left(\jump{\rho}\right)^2}{2\rho^{\ln}}-\frac{\rhoM\left(\jump{\beta}\right)^2}{\beta^{\ln}\betaM(\gamma-1)} - \rhoM\betaP\left(\left(\jump{v_1}\right)^2+\left(\jump{v_2}\right)^2+\left(\jump{v_3}\right)^2\right)\\
& + \avg{\beta}\left(\jump{B_1^2}+\jump{B_2^2}+\jump{B_3^2}\right)-\avg{\beta}\left(\left(\jump{B_1}\right)^2+\left(\jump{B_2}\right)^2+\left(\jump{B_3}\right)^2\right)\\
&+\jump{\beta}\left(\BoneM\BoneP+ \BtwoM\BtwoP+ \BthrM\BthrP\right)\\
=& \jump{\rho + \beta\ \left|\vec{B} \right|^2}-\frac{\left(\jump{\rho}\right)^2}{2\rho^{\ln}}-\frac{\rhoM\left(\jump{\beta}\right)^2}{\beta^{\ln}\betaM(\gamma-1)} - \rhoM\betaP\left(\left(\jump{v_1}\right)^2+\left(\jump{v_2}\right)^2+\left(\jump{v_3}\right)^2\right)\\
& - \jump{\beta}\left(\avg{B_1^2}+\avg{B_2^2}+\avg{B_3^2}\right)-\avg{\beta}\left(\left(\jump{B_1}\right)^2+\left(\jump{B_2}\right)^2+\left(\jump{B_3}\right)^2\right)\\
&+\jump{\beta}\left(\BoneM\BoneP+ \BtwoM\BtwoP+ \BthrM\BthrP\right)\\
=& \jump{\rho + \beta\ \left|\vec{B} \right|^2}-\frac{\left(\jump{\rho}\right)^2}{2\rho^{\ln}}-\frac{\rhoM\left(\jump{\beta}\right)^2}{\beta^{\ln}\betaM(\gamma-1)} - \rhoM\betaP\left(\left(\jump{v_1}\right)^2+\left(\jump{v_2}\right)^2+\left(\jump{v_3}\right)^2\right)\\
& - \left(\avg{\beta}+\frac{1}{2}\jump{\beta}\right)\left(\left(\jump{B_1}\right)^2+\left(\jump{B_2}\right)^2+\left(\jump{B_3}\right)^2\right)\\
=& \jump{\rho + \beta\ \left|\vec{B} \right|^2}-\frac{\left(\jump{\rho}\right)^2}{2\rho^{\ln}}-\frac{\rhoM\left(\jump{\beta}\right)^2}{\beta^{\ln}\betaM(\gamma-1)} - \rhoM\betaP\left(\left(\jump{v_1}\right)^2+\left(\jump{v_2}\right)^2+\left(\jump{v_3}\right)^2\right) - \betaP\left(\left(\jump{B_1}\right)^2+\left(\jump{B_2}\right)^2+\left(\jump{B_3}\right)^2\right)\\
\leq&\jump{\rho + \beta\ \left|\vec{B} \right|^2}.
\end{aligned}
$}
\end{equation}
Again, under the physical assumptions of positive density and temperature the upwind temporal state is entropy stable for the ideal MHD equations.

\subsection{Spatial part for Magnetohydrodynamics}\label{sec:App B3}

For the ideal magnetohydrodynamic (MHD) equations the entropy analysis is slightly more complicated due to the divergence-free constraint on the magnetic field
\begin{equation}
\vv{\nabla}_x\cdot\vv{B}=0.
\end{equation}
The divergence-free constraint must be incorporated into the ideal MHD equations as a non-conservative term in order for the discrete entropy analysis to mimic the continuous entropy analysis, e.g. \cite{Barth1999,bohm2018,Chandrashekar2015,Derigs2017,Liu2017}, as it can be that $\vv{\nabla}_x\cdot\vv{B}\ne 0$ for a numerical discretization.

The inclusion of this non-conservative term alters the derivation of the entropy conservative (or entropy stable) numerical flux in the three spatial directions. Moreover, the entropy flux potential $\psi$ as shown in \eqref{EntropyFunctional2} contains an influence from this non-conservative term and has the form
\begin{equation}\label{DiscreteEntropyConditionMHD}
\jump{\vec{W}}^{T}{\vec{F}}_i^{\text{EC}} = \jump{\Psi_i} 
\quad \text{with} \quad
\Psi_i =\vec{W}^{T} \vec{F}_i - {F}_i^{s} +\theta B_i,
\end{equation} 
for $i=1,2,3$ and
\begin{equation}\label{eq:thetaDef}
\theta = \vec{w}^T\boldsymbol\Upsilon
= 2\beta(\vv{v}\cdot\vv{B}).
\end{equation}
See \cite{Chandrashekar2015,Derigs2017} for complete details on the derivation of entropy stable spatial flux functions for the ideal MHD equations.

The DGSEM approximation on general curvilinear meshes discussed briefly in Section \ref{sec:DG} and fully presented in Gassner et al. \cite{Gassner2017} remains largely the same for the ideal MHD equations. The only issue is how to compute, at high-order, the volume contributions and surface coupling of the non-conservative term proportional to the divergence-free condition present in \eqref{eq:modifiedMHD}. We will present and briefly discuss the modified spatial operator and how it fits into the present analysis in this work. We note that complete details on the DG approximation for the MHD equations can be found in Bohm et al. \cite{bohm2018}.

The spatial operator from \eqref{SpatialPart} for the ideal MHD equations takes the form
\begin{equation}\label{SpaceOperatorMHD}
\begin{aligned}
A_{S}\!\left(\mathbf{U},\boldsymbol{\varphi}\right)
:=& \quad\frac{\Delta t}{2}\left\langle \Dprojection{N}\cdot\blockvec{\tilde{\mathbf{F}}}{}^{\text{EC}},\boldsymbol{\varphi}\right\rangle_{\!N\times M}+\frac{\Delta t}{2}\int\limits_{\partial E^{3},N}\left\langle \boldsymbol{\varphi}^{T}\left(\tilde{\mathbf{F}}^{*}-\left(\blockveccon{\mathbf{F}}\cdot\hat{n}\right)\right),1\right\rangle_{\!M}\dS\\[0.1cm]
&+\frac{\Delta t}{2}\left\langle \boldsymbol\Upsilon\Dprojection{N}\cdot\vv{\tilde{B}},\boldsymbol{\varphi}\right\rangle_{\!N\times M}+\frac{\Delta t}{2}\int\limits_{\partial E^{3},N}\left\langle \boldsymbol{\varphi}^{T}\left(\left(\boldsymbol\Upsilon\left(\vv{B}\cdot\hat{n}\right)\right)^{\!\!\Diamond}-\boldsymbol\Upsilon\left(\vv{{B}}\cdot\hat{n}\right)\right),1\right\rangle_{\!M}\dS,
\end{aligned}
\end{equation}
where we have additional contributions in the volume and at the surface of the non-conservative terms. The non-conservative volume contribution is
\begin{align}\label{NonConsDerivativeProjectionOperator}
\begin{split}
\boldsymbol\Upsilon\Dprojection{N}\cdot\vv{\tilde{B}}_{\sigma ijk}
:=2\sum_{m=0}^{N}\quad\,& 
\mathcal{D}_{im}\left(\boldsymbol\Upsilon_{ijk}\left(\vv{B}_{mjk}\cdot\avg{J\vv{a}^{1}}_{\left(i,m\right)jk}\right)\right) \\
\quad +  & \,\mathcal{D}_{jm}\left(\boldsymbol\Upsilon_{ijk}\left(\vv{B}_{imk}\cdot\avg{J\vv{a}^{2}}_{i(j,m)k}\right)\right) \\
\quad  +& \, \mathcal{D}_{km}\left(\boldsymbol\Upsilon_{ijk}\left(\vv{B}_{ijm}\cdot\avg{J\vv{a}^{3}}_{ij(k,m)}\right)\right), 
\end{split}
\end{align}
where $\sigma = 0,\ldots,M$. If we select the test function $\boldsymbol\phi$ to be the interpolant of the entropy variables, $\vec{W}$, then the volume contributions become the entropy flux at the boundary \cite{bohm2018}, i.e., 
\begin{equation}\label{EntropyVolumeContributionMHD}
\left\langle \Dprojection{N}\cdot\blockvec{\tilde{\mathbf{F}}}{}^{\text{EC}},\mathbf{W}\right\rangle_{\!N\times M} + \left\langle \boldsymbol\Upsilon\Dprojection{N}\cdot\left(\vv{\tilde{B}}\right),\vec{W}\right\rangle_{\!N\times M}=\int\limits_{\partial E^{3},N}\left\langle \vv{\tilde{F}}_{\hat{n}}^{s},1\right\rangle_{\!M}\dS.
\end{equation}
For the non-conservative surface contributions we define the interface coupling to be
\begin{equation}\label{NonConsSurface}
\left(\boldsymbol\Upsilon\left(\vv{B}\cdot\hat{n}\right)\right)^{\!\!\Diamond} = \boldsymbol\Upsilon^{-}\avg{\vv{B}}\cdot\hat{n},
\end{equation}
where the ``$-$'' denotes that the value is coming from the primary element. We typically suppress this notation unless it is necessary to make the manipulations clear. 
Then the last term in \eqref{SpaceOperatorMHD} becomes
\begin{equation}
\begin{aligned}
\int\limits_{\partial E^{3},N}\left\langle \vec{W}^{T}\left(\left(\boldsymbol\Upsilon\left(\vv{B}\cdot\hat{n}\right)\right)^{\!\!\Diamond}-\boldsymbol\Upsilon\left(\vv{{B}}\cdot\hat{n}\right)\right),1\right\rangle_{\!M}\dS &= \int\limits_{\partial E^{3},N}\left\langle \left(\vec{W}^{-}\right)^{T}\boldsymbol\Upsilon^{-}\left(\avg{\vv{B}}-\vv{{B}}\right)\cdot\hat{n},1\right\rangle_{\!M}\dS\\[0.1cm]
&=\frac{1}{2}\int\limits_{\partial E^3,N}\left\langle \theta^{-}\jump{\vv{B}}\cdot\hat{n},1\right\rangle_{\!M}\dS\\[0.1cm]
&=\frac{1}{2}\int\limits_{\partial E^3,N}\left\langle \theta\jump{\vv{B}}\cdot\hat{n},1\right\rangle_{\!M}\dS,
\end{aligned}
\end{equation}
where we use the result \eqref{eq:avgProperty} and the definition of $\theta$ from \eqref{eq:thetaDef}. Thus, the property \eqref{EntropyVolumeContribution} becomes
\begin{equation}\label{EntropyVolumeContributionMHD}
A_{S}\!\left(\blockvec{\mathbf{F}},\vec{W}\right):=\quad\frac{\Delta t}{2}\int\limits_{\partial E^{3},N}\left\langle \mathbf{W}^{T}\tilde{\mathbf{F}}_{\hat{n}}^{*}-\mathbf{W}^{T}\tilde{\mathbf{F}}_{\hat{n}}+\tilde{F}_{\hat{n}}^{s}+\frac{\theta}{2}\jump{\vv{B}}\cdot\hat{n},1\right\rangle_{\!M}\dS,
\end{equation}
in the case of the ideal MHD equations.

If we again introduce the contravariant flux notation \eqref{ContravariantSurfaceFlux} then summing over all the interior faces gives us
\begin{equation}\label{SurfaceEntropyConditionMHD}
\sum_{n=1}^{K_{T}}\underset{\text{faces}}{\sum_{\text{Interior}}}\frac{\Delta t^{n}}{2}\int\limits_{\partial E^{3},N}\left\langle \jump{\mathbf{W}^{n,k}}^{T}\tilde{\mathbf{F}}_{\hat{n}}^{n,k,*}-\jump{\left(\mathbf{W}^{n,k}\right)^{T}\tilde{\mathbf{F}}_{\hat{n}}^{n,k}}+\jump{\tilde{F}_{\hat{n}}^{s}}+\avg{\theta}\jump{\vv{B}}\cdot\hat{n},1\right\rangle_{\!M}\dS=0,
\end{equation}
due to the entropy stable numerical fluxes for the MHD equations in each Cartesian direction \eqref{DiscreteEntropyConditionMHD}. We note the additional surface term for the magnetic field components $\vv{B}$ comes from the orientation of the jump operator and that the physical normal vector $\hat{n}$ is defined uniquely at the interface \cite{bohm2018}. By summing over all space-time elements and applying the inequality \eqref{SurfaceEntropyCondition}, we obtain 
\begin{equation}\label{SpatialStabilityMHD}
\begin{aligned}
& 
\sum_{n=1}^{K_{T}}\sum_{k=1}^{K_{S}}A_{S}\!\left(\blockvec{\mathbf{F}}{}^{n,k},\mathbf{W}^{n,k}\right) \\ 
&=  \sum_{n=1}^{K_{T}}\underset{\text{faces}}{\sum_{\text{Boundary}}}\frac{\Delta t^{n}}{2}\int\limits_{\partial E^{3},N}\left\langle \left(\mathbf{W}^{n,k}\right)^{T}\tilde{\mathbf{F}}_{\hat{n}}^{n,k,*}-\left(\mathbf{W}^{n,k}\right)^{T}\tilde{\mathbf{F}}_{\hat{n}}+\tilde{F}_{\hat{n}}^{n,k,s}+\frac{\theta^{-}}{2}\left(\vv{B}^{\text{bndy}}-\vv{B}^{-}\right)\cdot\hat{n},1\right\rangle_{\!M}\dS  \\
& -\sum_{n=1}^{K_{T}}\underset{\text{faces}}{\sum_{\text{Interior}}}\frac{\Delta t^{n}}{2}\int\limits_{\partial E^{3},N}\left\langle\jump{\mathbf{W}^{n,k}}^{T}\tilde{\mathbf{F}}_{\hat{n}}^{n,k,*}-\jump{\left(\mathbf{W}^{n,k}\right)^{T}\tilde{\mathbf{F}}_{\hat{n}}^{n,k}}+\jump{\tilde{F}_{\hat{n}}^{n,k,s}}+\avg{\theta}\jump{\vv{B}}\cdot\hat{n},1\right\rangle_{\!M}\dS
\\
= & \quad\sum_{n=1}^{K_{T}}\underset{\text{faces}}{\sum_{\text{Boundary}}}\frac{\Delta t^{n}}{2}\int\limits_{\partial E^{3},N}\left\langle \left(\mathbf{W}^{n,k}\right)^{T}\tilde{\mathbf{F}}_{\hat{n}}^{n,k,*}-\left(\mathbf{W}^{n,k}\right)^{T}\tilde{\mathbf{F}}_{\hat{n}}+\tilde{F}_{\hat{n}}^{n,k,s}+\frac{\theta^{-}}{2}\left(\vv{B}^{\text{bndy}}-\vv{B}^{-}\right)\cdot\hat{n},1\right\rangle_{\!M}\dS, 
\end{aligned}
\end{equation}
where $\vv{B}^{\text{bndy}}$ is the magnetic field evaluated from an appropriate boundary state. So, we obtain a discrete entropy equality of a similar form as \eqref{SpatialStability2} up to the prescription of proper boundary conditions for the ideal MHD equations.

We note that the entropy preserving MHD fluxes are not enough to ensure a stable discretization for discontinuous solutions. A dissipation operator, like the one found in \cite{bohm2018}, must be added to the entropy preserving flux functions. Then, the equations \eqref{SurfaceEntropyConditionMHD} and \eqref{SpatialStabilityMHD} become inequalities. However, even with these inequalities a discrete entropy inequality of a similar form as \eqref{eq:finalResultEC} can be proven.    

\section{Proofs for the temporal entropy analysis}\label{sec:TimeProofs}
In this section, we apply the following identities which result from the properties of the SBP operator $\matx{Q}$  
\begin{align}
\sum_{i,j=0}^{K}\matx{Q}_{ij}\jump{a}_{\left(i,j\right)}=&
 \qquad   \sum_{i,j=0}^{n}\matx{B}_{ij}a_{j}=-\left.a\right|_{-1}^{\,1}, \label{SBP1}\\
\sum_{i,j=0}^{K}\matx{Q}_{ij}\jump{a}_{\left(i,j\right)}\avg{b}_{\left(i,j\right)}=& \qquad
\sum_{i,j=0}^{K}\matx{Q}_{ij}a_{i}b_{j}-\left.ab\right|_{-1}^{\,1}=-\sum_{i,j=0}^{K}\matx{Q}_{ij}a_{j}b_{i}, \label{SBP2} \\
\sum_{i,j=0}^{K}\matx{Q}_{ij}\jump{a}_{\left(i,j\right)}\avg{b}_{\left(i,j\right)}\avg{c}_{\left(i,j\right)} 
=&-\frac{1}{2}\sum_{i,j=0}^{K}\matx{Q}_{ij}a_{j}b_{i}c_{j}+\frac{1}{2}\sum_{i,j=0}^{n}\matx{Q}_{ij}a_{i}b_{i}c_{j}-\frac{1}{2}\sum_{i,j=0}^{n}\matx{Q}_{ij}a_{j}b_{i}c_{i}, \label{SBP3}
\end{align}
where $K=M;N$ and $\left\{ a\right\} _{i=0}^{K}$, $\left\{ b\right\} _{i=0}^{K}$ and $\left\{ c\right\} _{i=0}^{K}$ are  generic nodal values. These identities can be proven in a similar way as the discrete split forms in Lemma 1 in \cite{Gassner:2016ye}. Thus, we skip a proof in this paper.

\subsection{Proof of the inequality \eqref{TemporalStability2}}\label{sec:App C}
The coefficients of the SBP operator $\matx{Q}$ are given by $\matx{Q}_{\sigma\theta}=\omega_{\sigma}\matx{D}_{\sigma\theta}$ and the SBP property \eqref{SBP} supplies $2\matx{Q}_{\sigma\theta}=\matx{Q}_{\sigma\theta}-\matx{Q}_{\theta\sigma}+\matx{B}_{\sigma\theta}$. Thus, it follows  
\begin{align}\label{VolumeContributionTemporal1}
\begin{split}
\left\langle J\vec{\mathbb{D}}^{M}\mathbf{U}^{\text{EC}},\mathbf{W}\right\rangle_{\!N\times M}
=&\quad
\sum_{i,j,k=0}^{N}\omega_{ijk}J_{ijk}\sum_{\sigma,\theta=0}^{M}2\matx{Q}_{\sigma\theta}\mathbf{W}_{\sigma ijk}^{T}\mathbf{U}^{\#}\left(\mathbf{U}_{\sigma ijk},\mathbf{U}_{\theta ijk}\right) \\
=&\quad
\sum_{i,j,k=0}^{N}\omega_{ijk}J_{ijk}\sum_{\sigma,\theta=0}^{M}\matx{Q}_{\sigma\theta}\mathbf{W}_{\sigma ijk}^{T}\mathbf{U}^{\#}\left(\mathbf{U}_{\sigma ijk},\mathbf{U}_{\theta ijk}\right) \\
& -\sum_{i,j,k=0}^{N}\omega_{i,j,k}J_{ijk}\sum_{\sigma,\theta=0}^{M}\matx{Q}_{\theta\sigma}\mathbf{W}_{\sigma ijk}^{T}\mathbf{U}^{\#}\left(\mathbf{U}_{\sigma ijk},\mathbf{U}_{\theta ijk}\right) \\
& +\sum_{i,j,k=0}^{N}\omega_{ijk}J_{ijk}\sum_{\sigma,\theta=0}^{M}\matx{B}_{\sigma\theta}\mathbf{W}_{\sigma ijk}^{T}\mathbf{U}^{\#}\left(\mathbf{U}_{\sigma ijk},\mathbf{U}_{\theta ijk}\right).
\end{split}
\end{align}
The two-point state $\mathbf{U}^{\#}$ is symmetric, hence we obtain by the substitution $\sigma\longleftrightarrow \theta$ in the penultimate sum in equation \eqref{VolumeContributionTemporal1} 
\begin{align}\label{VolumeContributionTemporal2}
\begin{split}
\left\langle J\vec{\mathbb{D}}^{M}\mathbf{U}^{\text{EC}},\mathbf{W}\right\rangle_{\!N\times M}
=&\quad
\sum_{ijk=0}^{N}\omega_{ijk}J_{ijk}\sum_{\sigma,\theta=0}^{M}2\omega_{\sigma}\matx{D}_{\sigma\theta}\mathbf{W}_{\sigma ijk}^{T}\mathbf{U}^{\#}\left(\mathbf{U}_{\sigma ijk},\mathbf{U}_{\theta ijk}\right) \\
=&\quad
\sum_{ijk=0}^{N}\omega_{ijk}J_{ijk}\sum_{\sigma,\theta=0}^{M}\matx{Q}_{\sigma\theta}\jump{\mathbf{W}}_{\left(\sigma,\theta\right)ijk}^{T}\mathbf{U}^{\#}\left(\mathbf{U}_{\sigma ijk},\mathbf{U}_{\theta ijk}\right) \\
&+\sum_{ijk=0}^{N}\omega_{ijk}J_{ijk}\sum_{\sigma,\theta=0}^{M}\matx{B}_{\sigma\theta}\mathbf{W}_{\sigma ijk}^{T}\mathbf{U}^{\#}\left(\mathbf{U}_{\sigma ijk},\mathbf{U}_{\theta ijk}\right).
\end{split}
\end{align}
Next, since $\mathbf{U}^{\#}$ is consistent, it follows  
\begin{equation}\label{VolumeContributionTemporal3}
\sum_{ijk=0}^{N}\omega_{ijk}J_{ijk}\sum_{\sigma,\theta=0}^{M}\matx{B}_{\sigma\theta}\mathbf{W}_{\sigma ijk}^{T}\mathbf{U}^{\#}\left(\mathbf{U}_{\sigma ijk},\mathbf{U}_{\theta ijk}\right)=\left.\left\langle \mathbf{W}^{T}\mathbf{U},J\right\rangle_{\!N}\right|_{-1}^{\,1}.
\end{equation}
Furthermore, since $\mathbf{U}^{\#}$ satisfies the property \eqref{DiscreteEntropyConservation}, we obtain by the identity \eqref{SBP1}  
\begin{align}\label{VolumeContributionTemporal4}
\begin{split}
& \quad \sum_{ijk=0}^{N}\omega_{ijk}J_{ijk}\sum_{\sigma,\theta=0}^{M}\matx{Q}_{\sigma\theta}\jump{\mathbf{W}}_{\left(\sigma,\theta\right)ijk}^{T}\mathbf{U}^{\#}\left(\mathbf{U}_{\sigma ijk},\mathbf{U}_{\theta ijk}\right)  \\
=& \quad  \sum_{i,j,k=0}^{N}\omega_{ijk}J_{ijk}\sum_{\sigma,\theta=0}^{M}\matx{Q}_{\sigma\theta}\jump{\varPhi\left(\mathbf{U}\right)}_{\left(\sigma,\theta\right)ijk} \\
=&-\sum_{i,j,k=0}^{N}\omega_{ijk}J_{ijk}\left.\varPhi\left(\mathbf{U}\right)\right|_{-1}^{\, 1}
=-\left.\left\langle \varPhi\left(\mathbf{U}\right),J\right\rangle_{\!N}\right|_{-1}^{\,1}.
\end{split}
\end{align}
Next, by plugging the equations \eqref{VolumeContributionTemporal3} and \eqref{VolumeContributionTemporal4} in \eqref{VolumeContributionTemporal2}, we obtain the identity  
\begin{equation}\label{VolumeContributionTemporal11}
\left\langle J\vec{\mathbb{D}}^{M}\mathbf{U}^{\text{EC}},\mathbf{W}\right\rangle_{\!N\times M} 
= \left.\left\langle \mathbf{W}^{T}\mathbf{U},J\right\rangle_{\!N}\right|_{-1}^{\,1}-\left.\left\langle \varPhi\left(\mathbf{U}\right),J\right\rangle_{\!N}\right|_{-1}^{\,1}=\left.\left\langle s\left(\mathbf{U}\right),J\right\rangle_{\!N}\right|_{-1}^{\,1}.
\end{equation}
Finally, a summation of the temporal part \eqref{TemporalPart} over all space-time elements and the identity \eqref{VolumeContributionTemporal11} provide 
\begin{align}\label{SummationTemporal}
\begin{split}
\sum_{n=1}^{K_{T}}\sum_{k=1}^{K_{S}}A_{T}\!\left(\mathbf{U}^{n,k},\mathbf{W}^{n,k}\right)
=&  \quad \bar{S}\left(\mathbf{U}\left(T\right)\right)  
+  \sum_{k=1}^{K_{S}}\left.\left\langle \left(\mathbf{W}^{K_{T},k}\right)^{T}\left(\mathbf{U}^{K_{T},k,*}-\mathbf{U}^{K_{T},k}\right),J^{k}\right\rangle_{\!N}\right|_{\tau=1} \\
&  \quad \qquad \qquad -\sum_{n=2}^{K_{T}}\sum_{k=1}^{K_{S}}\left.\left\langle \jump{\mathbf{W}^{n,k}}^{T}\mathbf{U}^{n,k,*}-\jump{\varPhi\left(\mathbf{U}^{n,k}\right)},J^{k}\right\rangle_{\!N}\right|_{\tau=-1} \\
& -\bar{S}\left(\mathbf{U}\left(0\right)\right)-\sum_{k=1}^{K_{S}}\left.\left\langle \left(\mathbf{W}^{1,k}\right)^{T}\left(\mathbf{U}^{1,k,*}-\mathbf{U}^{1,k}\right),J^{k}\right\rangle_{\!N}\right|_{\tau=-1}, 
\end{split}
\end{align}
where the quantity $\bar{S}\left(\mathbf{U}\left(T\right)\right)$ is defined in \eqref{DisreteEntropyIntegral2} and the quantity $\bar{S}\left(\mathbf{U}\left(0\right)\right)$ is defined in \eqref{DisreteEntropyIntegral3}.

This completes the proof for the identity \eqref{TemporalStability2}.

\section{Proofs for the temporal kinetic energy analysis}\label{sec:TimeProofsKEP}\label{sec:App KEP}

\subsection{Proof of Theorem \ref{Therorem:KEP} (Kinetic energy preservation)}\label{sec:TimeProofsKEP1}
We choose $\boldsymbol{\varphi}=\mathbf{V}$ as test function in the equation \eqref{SpaceTimeDG1} and sum over all space-time elements to obtain
\begin{equation}\label{TemporalKineticStability1}
\sum_{n=1}^{K_{T}}\sum_{k=1}^{K_{S}}\left(A_{T}\!\left(\mathbf{U}^{n,k},\mathbf{V}^{n,k}\right)+A_{S}\!\left(\blockvec{\mathbf{F}}{}^{n,k},\mathbf{V}^{n,k}\right)\right)=0. 
\end{equation}

First, we investigate the temporal part of the space-time DGSEM. We proceed as in the proof of the identity \eqref{TemporalStability2} and obtain by the SBP property \eqref{SBP} 
\begin{align}\label{VolumeContTemporalKinetic1}
\begin{split}
\left\langle J\vec{\mathbb{D}}^{M}\mathbf{U}^{\text{KEP}},\mathbf{V}\right\rangle _{N\times M}
=&\quad
\sum_{ijk=0}^{N}\omega_{ijk}J_{ijk}\sum_{\sigma,\theta=0}^{M}2  \matx{Q}_{\sigma\theta}\mathbf{V}_{\sigma ijk}^{T}\mathbf{U}^{\text{KEP}}\left(\mathbf{U}_{\sigma ijk},\mathbf{U}_{\theta ijk}\right) \\
=&\quad
\sum_{ijk=0}^{N}\omega_{ijk}J_{ijk}\sum_{\sigma,\theta=0}^{M}\matx{Q}_{\sigma\theta}\jump{\mathbf{V}}_{\left(\sigma,\theta\right)ijk}^{T}\mathbf{U}^{\text{KEP}}\left(\mathbf{U}_{\sigma ijk},\mathbf{U}_{\theta ijk}\right) \\
&+\sum_{ijk=0}^{N}\omega_{ijk}J_{ijk}\sum_{\sigma,\theta=0}^{M}\matx{B}_{\sigma\theta}\mathbf{V}_{\sigma ijk}^{T}\mathbf{U}^{\text{KEP}}\left(\mathbf{U}_{\sigma ijk},\mathbf{U}_{\theta ijk}\right).
\end{split}
\end{align}
The definition of the quantity $\mathbf{V}$ provides  
\begin{align}\label{VolumeContTemporalKinetic2}
\begin{split}
& \quad \sum_{ijk=0}^{N}\omega_{ijk}J_{ijk}\sum_{\sigma,\theta=0}^{M}\matx{Q}_{\sigma\theta}\jump{\vec{V}}_{\left(\sigma,\theta\right)ijk}^{T}\mathbf{U}^{\text{KEP}}\left(\mathbf{U}_{\sigma ijk},\mathbf{U}_{\theta ijk}\right)  \\
=& -\sum_{ijk=0}^{N}\omega_{ijk}J_{ijk}\sum_{\sigma,\theta=0}^{M}\matx{Q}_{\sigma\theta}\frac{1}{2}\jump{\left|\vv{v}\right|^{2}}_{\left(\sigma,\theta\right)ijk}\mathbf{U}_{1}^{\text{KEP}}\left(\mathbf{U}_{\sigma ijk},\mathbf{U}_{\theta ijk}\right) \\
& +\sum_{ijk=0}^{N}\omega_{ijk}J_{ijk}\sum_{\sigma,\theta=0}^{M}\matx{Q}_{\sigma\theta}\jump{v_{1}}_{\left(\sigma,\theta\right)ijk}\mathbf{U}_{2}^{\text{KEP}}\left(\mathbf{U}_{\sigma ijk},\mathbf{U}_{\theta ijk}\right) \\
& +\sum_{ijk=0}^{N}\omega_{ijk}J_{ijk}\sum_{\sigma,\theta=0}^{M}\matx{Q}_{\sigma\theta}\jump{v_{2}}_{\left(\sigma,\theta\right)ijk}\mathbf{U}_{3}^{\text{KEP}}\left(\mathbf{U}_{\sigma ijk},\mathbf{U}_{\theta ijk}\right) \\
& +\sum_{ijk=0}^{N}\omega_{ijk}J_{ijk}\sum_{\sigma,\theta=0}^{M}\matx{Q}_{\sigma\theta}\jump{v_{3}}_{\left(\sigma,\theta\right)ijk}\mathbf{U}_{4}^{\text{KEP}}\left(\mathbf{U}_{\sigma ijk},\mathbf{U}_{\theta ijk}\right).
\end{split}
\end{align}
Next, it follows by the properties \eqref{DiscreteKineticFlux} of the state $\mathbf{U}^{\text{KEP}}$ and the identity \eqref{eq:jumpProperties} 
\begin{equation}\label{VolumeContTemporalKinetic3}
\sum_{ijk=0}^{N}\omega_{ijk}J_{ijk}\sum_{\sigma,\theta=0}^{M}\matx{Q}_{\sigma\theta}\jump{\mathbf{V}}_{\left(\sigma,\theta\right)ijk}^{T}\mathbf{U}^{\text{KEP}}\left(\mathbf{U}_{\sigma ijk},\mathbf{U}_{\theta ijk}\right)=0. 
\end{equation}
Furthermore, since $\mathbf{U}^{\text{KEP}}$ is consistent, we obtain
\begin{equation}\label{VolumeContTemporalKinetic4}
\sum_{ijk=0}^{N}\omega_{ijk}J_{ijk}\sum_{\sigma,\theta=0}^{M}\matx{B}_{\sigma\theta}\mathbf{V}_{\sigma ijk}^{T}\mathbf{U}^{\#}\left(\mathbf{U}_{\sigma ijk},\mathbf{U}_{\theta ijk}\right)=\left.\left\langle J\mathbf{U},\mathbf{V}\right\rangle _{N}\right|_{-1}^{\ \,1}.
\end{equation}
Next, we obtain by \eqref{VolumeContTemporalKinetic1}, \eqref{VolumeContTemporalKinetic3} and  \eqref{VolumeContTemporalKinetic4} 
\begin{equation}\label{VolumeContTemporalKinetic5}
A_{T}\!\left(\mathbf{U},\mathbf{V}\right)=\left.\left\langle J\mathbf{U}^{*},\mathbf{V}\right\rangle _{N}\right|_{-1}^{\,1}.  
\end{equation}
Since, the space-time DGSEM is applied with Dirichlet boundary conditions in time, a summation of the temporal part \eqref{TemporalPart} over all space-time elements and the \eqref{VolumeContTemporalKinetic5} provide  
\begin{align}\label{SummationTemporalKinetic1}
\begin{split}
\sum_{k=1}^{K_{T}}\sum_{k=1}^{K_{S}}A_{T}\left(\mathbf{U}^{n,k},\mathbf{V}^{n,k}\right) 
=&\quad \bar{K}\left(\mathbf{U}\left(T\right)\right)
-\sum_{k=2}^{K_{T}}\sum_{k=1}^{K_{S}}\left.\left\langle \jump{\mathbf{V}^{n,k}}^{T}\mathbf{U}^{n,k,*},J^{k}\right\rangle _{N}\right|_{\tau=-1} \\
& \qquad \qquad \qquad   -\sum_{k=1}^{K_{S}}\left.\left\langle \left(\mathbf{V}_{+}^{1,k}\right)^{T}\mathbf{U}^{1,k,*},J^{k}\right\rangle _{N}\right|_{\tau=-1}, 
\end{split}
\end{align}
where $\bar{K}\left(\mathbf{U}\left(T\right)\right)$ is given by \eqref{DicreteKineticIntegrals1}. The temporal numerical state functions $\mathbf{U}^{*}$ have the property \eqref{DiscreteKineticFlux} at the interior temporal boundary points. Thus, we obtain     
\begin{equation}\label{SummationTemporalKinetic2}
\sum_{k=2}^{K_{T}}\sum_{k=1}^{K_{S}}\left.\left\langle \jump{\mathbf{V}^{n,k}}^{T}\mathbf{U}^{n,k,*},J^{k}\right\rangle _{N}\right|_{\tau=-1}=0.  
\end{equation}
Furthermore, at the exterior temporal boundary points the temporal numerical state functions $\mathbf{U}^{*}$ satisfy the property  \eqref{EntropyPreservation1}. This provides the identity    
\begin{equation}\label{SummationTemporalKinetic3}
\left.\mathbf{U}^{1,k,*}\right|_{\tau=-1}=\left.\mathbf{U}_{-}^{1,k}\right|_{\tau=-1}=\left.\mathbf{u}^{1,k}\right|_{\tau=-1}, \qquad k=1,\dots,K_{S},
\end{equation}
where $\left.\mathbf{u}^{1,k}\right|_{\tau=-1}=\mathbf{u}^{k}\left(0\right)$ is the initial condition prescribed to the conservation law in the spatial cell $e_{k}$, $k=1,\dots,K_{S}$. Thus, we obtain 
\begin{equation}\label{SummationTemporalKinetic4}
-\sum_{k=1}^{K_{S}}\left.\left\langle
 \left(\mathbf{V}_{+}^{1,k}\right)^{T}\mathbf{U}^{1,k,*},J^{k}\right\rangle _{N}\right|_{\tau=-1}	
=	-\bar{K}\left(\mathbf{u}\left(0\right)\right)
 -\sum_{k=1}^{K_{S}}\left.\left\langle \jump{\mathbf{V}^{1,k}}^{T}\mathbf{u}^{1,k},J^{k}\right\rangle _{N}\right|_{\tau=-1}, 
\end{equation}
where $\bar{K}\left(\mathbf{u}\left(0\right)\right)$ is given by \eqref{DicreteKineticIntegrals2}. Therefore, the equation \eqref{SummationTemporalKinetic1} becomes
\begin{align}\label{SummationTemporalKinetic5}
\begin{split}
\sum_{k=1}^{K_{T}}\sum_{k=1}^{K_{S}}A_{T}\left(\mathbf{U}^{n,k},\mathbf{V}^{n,k}\right)  
=&   \quad  \bar{K}\left(\mathbf{U}\left(T\right)\right) 
-\bar{K}\left(\mathbf{u}\left(0\right)\right) \\
&   -\sum_{k=1}^{K_{S}}\left.\left\langle \jump{\mathbf{V}^{1,k}}^{T}\mathbf{u}^{1,k},J^{k}\right\rangle _{N}\right|_{\tau=-1}.
\end{split}
\end{align}  

Next, we consider the spatial part of the space-time DGSEM. In Appendix \ref{sec:TimeProofsKEP2}, we prove the identity 
\begin{equation}\label{KineticVolumeContribution}
\left\langle \Dprojection{N}\cdot\blockvec{\tilde{\mathbf{F}}}{}^{\text{KEP}},\mathbf{V}\right\rangle _{N\times M}=-\left\langle \vv{\nabla}_{\xi}\cdot\interpolation{N}{\left(\vv{\tilde{v}}\right)},p\right\rangle _{N\times M}+\int\limits _{\partial E^{3},N}\left\langle \tilde{F}_{\hat{n}}^{\kappa}+p\tilde{v}_{\hat{n}},1\right\rangle _{M}\dS, 
\end{equation}
where $\tilde{F}^{\kappa}_{\hat{n}}:=\vv{\tilde{F}}^{\kappa}\cdot\hat{n}$ and $\vv{\tilde{F}}^{\kappa}$ is the polynomial approximation of the kinetic energy flux $\vv{\tilde{f}}^{\kappa}$. 
It follows by \eqref{BlockvectorNormal}
\begin{align}\label{KineticBlockvectorNormal}
\begin{split}
\mathbf{V}^T\tilde{\mathbf{F}}_{\hat{n}} 
=& \quad\; \hat{s}n_{1}\left(-\frac{1}{2}\left|\vv{v}\right|^{2}
+\rho v_{1}^{3}+ \rho p v_{1}  +\rho v_{1} v_{2}^{2}+\rho v_{1} v_{3}^{2} \right) 
\\
&+\hat{s}n_{2}\left(-\frac{1}{2}\left|\vv{v}\right|^{2}
+\rho v_{1}^{2}v_{2} +\rho  v_{2}^{3}+ \rho P v_{2}+\rho v_{2} v_{3}^{2} \right) 
\\
&+\hat{s}n_{3}\left(-\frac{1}{2}\left|\vv{v}\right|^{2}
+\rho v_{1}^{2}v_{3} +\rho v_{2}^{2}v_{3}+\rho v_{3}^{3}+ \rho p v_{3}  \right) \\
=& \frac{1}{2}\left|\vv{v}\right|^{2}\left(\hat{s}\vv{n}\cdot\vv{v}\right)+p\left(\hat{s}\vv{n}\cdot\vv{v}\right)=\tilde{F}_{\hat{n}}^{\kappa}+p\tilde{v}_{\hat{n}}. 
\end{split}
\end{align}
Thus, we obtain by \eqref{KineticVolumeContribution} and \eqref{KineticBlockvectorNormal}  
\begin{align}\label{SpatialStabilityKinetic1}
\begin{split}
A_{S}\!\left(\blockvec{\mathbf{F}},\mathbf{V}\right) 
=& -\frac{\Delta t}{2}\left\langle \vv{\nabla}_{\xi}\cdot\interpolation{N}{\left(\vv{\tilde{v}}\right)},p\right\rangle _{N\times M}  \\
&+\frac{\Delta t}{2}\int\limits _{\partial E^{3},N}\left\langle \tilde{F}_{\hat{n}}^{\kappa}+p\tilde{v}_{\hat{n}},1\right\rangle _{M}\dS+\frac{\Delta t}{2}\int\limits _{\partial E^{3},N}\left\langle \mathbf{V}^{T}\tilde{\mathbf{F}}_{\hat{n}}^{*}-\tilde{F}_{\hat{n}}^{\kappa}-p\tilde{v}_{\hat{n}},1\right\rangle _{M} \\
=&-\frac{\Delta t}{2}\left\langle \vv{\nabla}_{\xi}\cdot\interpolation{N}{\left(\vv{\tilde{v}}\right)},p\right\rangle _{N\times M} 
+\frac{\Delta t}{2}\int\limits _{\partial E^{3},N}\left\langle \mathbf{V}^{T}\tilde{\mathbf{F}}_{\hat{n}}^{*},1\right\rangle _{M}.
\end{split}
\end{align}
Additionally, the conditions \eqref{Jameson2} and the properties  \eqref{eq:jumpProperties} of the jump operator provide
\begin{align}\label{KineticSurfaceFluxNormal}
\begin{split}
\jump{\mathbf{V}}^{T}\tilde{\mathbf{F}}_{\hat{n}}^{*} 
=& \quad\;  \hat{s}n_{1}\left(-\frac{1}{2}\jump{\left|\vv{v}\right|^{2}}\mathbf{F}_{1}^{1,\text{KEP}}+\jump{v_{1}}\mathbf{F}_{1}^{2,\text{KEP}}+\jump{v_{2}}\mathbf{F}_{1}^{3,\text{KEP}}+\jump{v_{3}}\mathbf{F}_{1}^{4,\text{KEP}}\right) \\
& +\hat{s}n_{2}\left(-\frac{1}{2}\jump{\left|\vv{v}\right|^{2}}\mathbf{F}_{2}^{1,\text{KEP}}+\jump{v_{1}}\mathbf{F}_{2}^{2,\text{KEP}}+\jump{v_{2}}\mathbf{F}_{2}^{3,\text{KEP}}+\jump{v_{3}}\mathbf{F}_{2}^{4,\text{KEP}}\right) \\
&+\hat{s}n_{3}\left(-\frac{1}{2}\jump{\left|\vv{v}\right|^{2}}\mathbf{F}_{3}^{1,\text{KEP}}+\jump{v_{1}}\mathbf{F}_{3}^{2,\text{KEP}}+\jump{v_{2}}\mathbf{F}_{3}^{3,\text{KEP}}+\jump{v_{1}}\mathbf{F}_{3}^{4,\text{KEP}}\right) \\ 
=& \quad\; \hat{s}n_{1}\mathbf{F}_{1}^{1,\text{KEP}}\left(-\frac{1}{2}\jump{\left|\vv{v}\right|^{2}}+\jump{v_{1}}\avg{v_{1}}+\jump{v_{1}}\avg{p}+\jump{v_{2}}\avg{v_{2}}+\jump{v_{3}}\avg{v_{3}}\right) \\
&+\hat{s}n_{2}\mathbf{F}_{1}^{1,\text{KEP}}\left(-\frac{1}{2}\jump{\left|\vv{v}\right|^{2}}+\jump{v_{1}}\avg{v_{1}}+\jump{v_{2}}\avg{v_{2}}+\jump{v_{2}}\avg{p}+\jump{v_{3}}\avg{v_{3}}\right) \\
&+\hat{s}n_{3}\mathbf{F}_{3}^{1,\text{KEP}}\left(-\frac{1}{2}\jump{\left|\vv{v}\right|^{2}}+\jump{v_{1}}\avg{v_{1}}+\jump{v_{2}}\avg{v_{2}}+\jump{v_{3}}\avg{v_{3}}+\jump{v_{3}}\avg{p}\right)\\
=& \avg{p}\left(\hat{s}\vv{n}\cdot\jump{\vv{v}}\right)=\avg{p}\jump{\tilde{v}_{\hat{n}}}.
\end{split}
\end{align}
Next, we sum the equation \eqref{SpatialStabilityKinetic1} over all space-time elements and apply the identity \eqref{KineticSurfaceFluxNormal}. This results in   
\begin{align}\label{SpatialStabilityKinetic2}
\begin{split}
 \quad\sum_{n=1}^{K_{T}}\sum_{k=1}^{K_{S}}A_{S}\!\left(\blockvec{\mathbf{F}}{}^{n,k},\mathbf{V}^{n,k}\right) 
= & -\sum_{n=1}^{K_{T}}\sum_{k=1}^{K_{S}}\frac{\Delta t^{n}}{2}\left\langle \vv{\nabla}_{\xi}\cdot\interpolation{N}{\left(\vv{\tilde{v}}^{n,k}\right)},p^{n,k}\right\rangle _{N\times M} \\
&+\sum_{n=1}^{K_{T}}\underset{\text{faces}}{\sum_{\text{Boundary}}}\frac{\Delta t^{n}}{2}\int\limits _{\partial E^{3},N}\left\langle \left(\mathbf{V}^{n,k}\right)^{T}\tilde{\mathbf{F}}_{\hat{n}}^{n,k,*},1\right\rangle _{M}\dS \\
&-\sum_{n=1}^{K_{T}}\underset{\text{faces}}{\sum_{\text{Interior}}}\frac{\Delta t^{n}}{2}\int\limits _{\partial E^{3},N}\left\langle \jump{\mathbf{V}^{n,k}}^{T}\tilde{\mathbf{F}}_{\hat{n}}^{n,k,*},1\right\rangle _{M}\dS \\
=&-\sum_{n=1}^{K_{T}}\sum_{k=1}^{K_{S}}\frac{\Delta t^{n}}{2}\left\langle \vv{\nabla}_{\xi}\cdot\interpolation{N}{\left(\vv{\tilde{v}}^{n,k}\right)},p^{n,k}\right\rangle _{N\times M} \\ 
&-\sum_{n=1}^{K_{T}}\underset{\text{faces}}{\sum_{\text{Interior}}}\frac{\Delta t^{n}}{2}\int\limits _{\partial E^{3},N}\left\langle \avg{p^{n,k}}\jump{\tilde{v}_{\hat{n}}^{n,k}},1\right\rangle _{M}\dS+\textbf{BC},
\end{split}
\end{align}
where 
\begin{equation}\label{SpatialStabilityKinetic3}
\textbf{BC}:=\sum_{n=1}^{K_{T}}\underset{\text{faces}}{\sum_{\text{Boundary}}}\frac{\Delta t^{n}}{2}\int\limits _{\partial E^{3},N}\left\langle \left(\mathbf{V}^{n,k}\right)^{T}\tilde{\mathbf{F}}_{\hat{n}}^{n,k,*},1\right\rangle _{M}\dS.
\end{equation}
The term $\textbf{BC}$ in \eqref{SpatialStabilityKinetic3} vanished, since the space-time DGSEM is applied with periodic boundary conditions in space. Finally, we substitute the results \eqref{SummationTemporalKinetic5} and  \eqref{SpatialStabilityKinetic2} in \eqref{TemporalKineticStability1}, use that $\textbf{BC}=0$ and rearrange terms. This results in the final equation  
\begin{align}\label{eq:finalResultKEP}
\begin{split}
\bar{K}\left(\mathbf{U}\left(T\right)\right)
&=
\bar{K}\left(\mathbf{U}\left(0\right)\right)
+\sum_{n=1}^{K_{T}}\sum_{k=1}^{K_{S}}\frac{\Delta t^{n}}{2}\left\langle \vv{\nabla}_{\xi}\cdot\interpolation{N}{\left(\vv{\tilde{v}}^{n,k}\right)},p^{n,k}\right\rangle _{N\times M} \\ 
& \quad \qquad \qquad \ \ +\sum_{n=1}^{K_{T}}\underset{\text{faces}}{\sum_{\text{Interior}}}\frac{\Delta t^{n}}{2}\int\limits _{\partial E^{3},N}\left\langle \avg{p^{n,k}}\jump{\tilde{v}_{\hat{n}}^{n,k}},1\right\rangle _{M}\dS \\
& \quad \qquad \qquad \ \ 
 +\sum_{k=1}^{K_{S}}\left.\left\langle \jump{\mathbf{V}^{1,k}}^{T}\mathbf{u}^{1,k},J^{k}\right\rangle _{N}\right|_{\tau=-1}. 
\end{split}
\end{align}

\subsection{Proof of the identity \eqref{KineticVolumeContribution}}\label{sec:TimeProofsKEP2}
The definition of the spatial derivative projection operator \eqref{SpatialDerivativeProjectionOperator} supplies   
\begin{align}\label{KEPVolumeContribution1}
\begin{split}
& \left\langle \Dprojection{N}\cdot\blockvec{\tilde{\mathbf{F}}}{}^{\text{KEP}},\mathbf{V}\right\rangle _{N\times M}  \\
=& \sum_{\sigma=0}^{M}\omega_{\sigma}
\sum_{i,j,k,m=0}^{N} \bigg\{\omega_{jk}2\matx{Q}_{im}\mathbf{V}_{\sigma ijk}^{T}\left(\blockvec{\mathbf{F}}{}^{\text{KEP}}\left(\mathbf{U}_{\sigma ijk},\mathbf{U}_{\sigma mjk}\right)\cdot\avg{J\vv{a}^{1}}_{\left(i,m\right)jk}\right) \\ 
& \qquad \qquad \qquad \quad   +\,\omega_{ik}2\matx{Q}_{jm}\mathbf{V}_{\sigma ijk}^{T}\left(\blockvec{\mathbf{F}}{}^{\text{KEP}}\left(\mathbf{U}_{\sigma ijk},\mathbf{U}_{\sigma imk}\right)\cdot\avg{J\vv{a}^{2}}_{i\left(j,m\right)k}\right)  \\ 
& \qquad \qquad \qquad \quad +\,\omega_{ij}2\matx{Q}_{km}\mathbf{V}_{\sigma ijk}^{T}\left(\blockvec{\mathbf{F}}{}^{\text{KEP}}\left(\mathbf{U}_{\sigma ijk},\mathbf{U}_{\sigma ijm}\right)\cdot\avg{J\vv{a}^{3}}_{ij\left(k,m\right)}\right)\bigg\}. 
\end{split}
\end{align}
The numerical flux functions satisfy the symmetry condition \eqref{FluxSymmetric}. Thus, by the same arguments as in the appendix \eqref{sec:App C} and the SBP property, the first sum on the the right hand side in \eqref{KEPVolumeContribution1} can be written as    
\begin{align}\label{KEPVolumeContribution2}
\begin{split}
& \quad \sum_{\sigma=0}^{M}\omega_{\sigma}\sum_{i,j,k,m=0}^{N}
\omega_{jk}2\matx{Q}_{im}\mathbf{V}_{\sigma ijk}^{T}\left(\blockvec{\mathbf{F}}{}^{\text{KEP}}\left(\mathbf{U}_{\sigma ijk},\mathbf{U}_{\sigma mjk}\right)\cdot\avg{J\vv{a}^{1}}_{\left(i,m\right)jk}\right) \\ 
=& \quad \sum_{\sigma=0}^{M}\omega_{\sigma}\sum_{i,j,k,m=0}^{N}\omega_{jk}\matx{Q}_{im}\jump{\mathbf{V}}_{\sigma ijk}^{T}\left(\blockvec{\mathbf{F}}{}^{\text{KEP}}\left(\mathbf{U}_{\sigma ijk},\mathbf{U}_{\sigma mjk}\right)\cdot\avg{J\vv{a}^{1}}_{\left(i,m\right)jk}\right) \\
&+\sum_{\sigma=0}^{M}\omega_{\sigma}\sum_{i,j,k,m=0}^{N}\omega_{jk}\matx{B}_{im}\mathbf{V}_{\sigma ijk}^{T}\left(\blockvec{\mathbf{F}}{}^{\text{KEP}}\left(\mathbf{U}_{\sigma ijk},\mathbf{U}_{\sigma mjk}\right)\cdot\avg{J\vv{a}^{1}}_{\left(i,m\right)jk}\right). 
\end{split}
\end{align}
For the last sum on the right hand side in \eqref{KEPVolumeContribution2}, it follows 
\begin{align}\label{KEPVolumeContribution4}
\begin{split}
& \quad \sum_{\sigma=0}^{M}\omega_{\sigma}\sum_{i,j,k,m=0}^{N}\omega_{jk}\matx{B}_{im}\mathbf{V}_{\sigma ijk}^{T}\left(\blockvec{\mathbf{F}}{}^{\text{KEP}}\left(\mathbf{U}_{\sigma ijk},\mathbf{U}_{\sigma mjk}\right)\cdot\avg{J\vv{a}^{1}}_{\left(i,m\right)jk}\right) \\
=& \quad \sum_{\sigma=0}^{M}\omega_{\sigma}\sum_{j,k=0}^{N}\omega_{jk}\left\{ \left(\vv{F}_{\sigma Njk}^{\kappa}+p_{\sigma Njk}\vv{v}_{\sigma Njk}\right)^{T}J\vv{a}_{Njk}^{1}-\left(\vv{F}_{\sigma0jk}^{\kappa}+p_{\sigma0jk}\vv{v}_{\sigma0jk}\right)^{T}J\vv{a}_{0jk}^{1}\right\} ,
\end{split}
\end{align}
since the numerical flux functions are consistent with $\blockvec{\mathbf{F}}$ and the equation 
\begin{equation}\label{KEPVolumeContribution3}
\mathbf{V}^{T}\left(\blockvec{\mathbf{F}}\cdot J\vv{a}^{1}\right)=\left(\vv{F}^{\kappa}+p\vv{v}\right)^{T} J\vv{a}^{1},
\end{equation}
is satisfied at the nodes. Moreover, we obtain by Jameson's conditions \eqref{Jameson} and the property \eqref{eq:jumpProperties} of the jump operator the identity 
\begin{align}\label{KEPVolumeContribution5}
\begin{split}
& \quad \jump{\mathbf{V}}_{\sigma ijk}^{T}\left(\blockvec{\mathbf{F}}{}^{\text{KEP}}\left(\mathbf{U}_{\sigma ijk},\mathbf{U}_{\sigma mjk}\right)\cdot\avg{J\vv{a}^{1}}_{\left(i,m\right)jk}\right) \\
=& \quad \sum_{s=1}^{3}\bigg\{-\frac{1}{2}\jump{\left|\vv{v}\right|^2}_{\sigma\left(i,m\right)jk}\mathbf{F}_{s}^{1,\text{KEP}}\left(\mathbf{U}_{\sigma ijk},\mathbf{U}_{\sigma mjk}\right)\avg{J{a}_{s}^{1}}_{\left(i,m\right)jk} \\ 
& \quad \qquad \quad  + \avg{v_{1}}_{\sigma\left(i,m\right)jk}\jump{v_{1}}_{\sigma\left(i,m\right)jk}\mathbf{F}_{s}^{1,\text{KEP}}\left(\mathbf{U}_{\sigma ijk},\mathbf{U}_{\sigma mjk}\right)\avg{J{a}_{s}^{1}}_{\left(i,m\right)jk} \\
& \quad \quad \qquad  +\avg{v_{2}}_{\sigma\left(i,m\right)jk}\jump{v_{2}}_{\sigma\left(i,m\right)jk}\mathbf{F}_{s}^{1,\text{KEP}}\left(\mathbf{U}_{\sigma ijk},\mathbf{U}_{\sigma mjk}\right)\avg{J{a}_{s}^{1}}_{\left(i,m\right)jk} \\ 
& \quad \quad \qquad  +\avg{v_{3}}_{\sigma\left(i,m\right)jk}\jump{v_{3}}_{\sigma\left(i,m\right)jk}\mathbf{F}_{s}^{1,\text{KEP}}\left(\mathbf{U}_{\sigma ijk},\mathbf{U}_{\sigma mjk}\right)\avg{J{a}_{s}^{1}}_{\left(i,m\right)jk} \\
& \quad \quad \qquad  
+ \jump{v_{s}}_{\sigma\left(i,m\right)jk}\left(p_{s1}^{\star}\right)_{\sigma\left(i,m\right)jk}\avg{J{a}_{s}^{1}}_{\left(i,m\right)jk} \bigg\} \\
=& \quad \jump{\vv{v}}_{\sigma\left(i,m\right)jk}^{T}\left(2\avg{p}_{\sigma\left(i,m\right)jk}\avg{J\vv{a}^{1}}_{\left(i,m\right)jk}-\avg{pJ\vv{a}^{1}}_{\sigma\left(i,m\right)jk}\right)
\end{split}
\end{align}
for all nodal values, since the equation \eqref{Jameson1} provides   
\begin{equation}\label{KEPVolumeContribution6}
\left(p_{s1}^{\star}\right)_{\sigma\left(i,m\right)jk}\avg{J{a}_{s}^{1}}_{\left(i,m\right)jk}=2\avg{p}_{\sigma\left(i,m\right)jk}\avg{J{a}_{s}^{1}}_{\left(i,m\right)jk}-\avg{pJ{a}_{s}^{1}}_{\sigma\left(i,m\right)jk},\quad s=1,2,3.
\end{equation}
By combining the identity \eqref{KEPVolumeContribution5} with \eqref{SBP2} and \eqref{SBP3}, we obtain          
\begin{align}\label{KEPVolumeContribution7}
\begin{split}
& \quad 
\sum_{\sigma=0}^{M}\omega_{\sigma}\sum_{i,j,k,m=0}^{N}\omega_{jk}\matx{Q}_{im}\jump{\mathbf{V}}_{\sigma ijk}^{T}\left(\blockvec{\mathbf{F}}{}^{\text{KEP}}\left(\mathbf{U}_{\sigma ijk},\mathbf{U}_{\sigma mjk}\right)\cdot\avg{J\vv{a}^{1}}_{\left(i,m\right)jk}\right)  \\
=& \quad 
\sum_{\sigma=0}^{M}\omega_{\sigma}\sum_{i,j,k,m=0}^{N}\omega_{jk}\matx{Q}_{im}
\jump{\vv{v}}_{\sigma\left(i,m\right)jk}^{T}\left(2\avg{p}_{\sigma\left(i,m\right)jk}\avg{J\vv{a}^{1}}_{\left(i,m\right)jk}-\avg{pJ\vv{a}^{1}}_{\sigma\left(i,m\right)jk}\right) \\ 
=& \quad \left\langle p\vv{v},\pderivative{\interpolation{N}{\left(J\vv{a}^{1}}\right)}{\xi^{1}}\right\rangle _{N\times M} 
-\left\langle \pderivative{\interpolation{N}{\left(\vv{v}^T J\vv{a}^{1}}\right)}{\xi^{1}},p\right\rangle _{N\times M}. 
\end{split}
\end{align}
Then, we substitute \eqref{KEPVolumeContribution4} and \eqref{KEPVolumeContribution7} in \eqref{KEPVolumeContribution2}. This results in the equation   
\begin{align}\label{KEPVolumeContribution8}
\begin{split}
& \quad \quad \sum_{\sigma=0}^{M}\omega_{\sigma}\sum_{i,j,k,m=0}^{N}
\omega_{jk}2\matx{Q}_{jm}\mathbf{V}_{\sigma ijk}^{T}\left(\blockvec{\mathbf{F}}{}^{\text{KEP}}\left(\mathbf{U}_{\sigma ijk},\mathbf{U}_{\sigma mjk}\right)\cdot\avg{J\vv{a}^{1}}_{\left(i,m\right)jk}\right) \\
=& \quad  \quad \left\langle p\vv{v},\pderivative{\interpolation{N}{\left(J\vv{a}^{1}}\right)}{\xi^{1}}\right\rangle _{N\times M} 
-\left\langle \pderivative{\interpolation{N}{\left(\vv{v}^T J\vv{a}^{1}}\right)}{\xi^{1}},p\right\rangle _{N\times M} \\
& + \quad  \sum_{\sigma=0}^{M}\omega_{\sigma}\sum_{j,k=0}^{N}\omega_{jk}\left\{ \left(\vv{F}_{\sigma Njk}^{\kappa}+p_{\sigma Njk}\vv{v}_{\sigma Njk}\right)^{T}J\vv{a}_{\sigma Njk}^{1}-\left(\vv{F}_{\sigma 0jk}^{\kappa}+p_{\sigma 0jk}\vv{v}_{\sigma 0jk}\right)^{T}J\vv{a}_{0jk}^{1}\right\}. 
\end{split}
\end{align}

In the same way the second and third sum on the right hand side in \eqref{KEPVolumeContribution1} are evaluated. Thus, we also have the identities     
\begin{align}\label{KEPVolumeContribution9}
\begin{split}
& \quad \quad \sum_{\sigma=0}^{M}\omega_{\sigma}\sum_{i,j,k,m=0}^{N}
\omega_{ik}2\matx{Q}_{jm}\mathbf{V}_{\sigma ijk}^{T}\left(\blockvec{\mathbf{F}}{}^{\text{KEP}}\left(\mathbf{U}_{\sigma ijk},\mathbf{U}_{\sigma mjk}\right)\cdot\avg{J\vv{a}^{2}}_{i\left(j,m\right)k}\right) \\
=& \quad  \quad \left\langle p\vv{v},\pderivative{\interpolation{N}{\left(J\vv{a}^{2}}\right)}{\xi^{2}}\right\rangle _{N\times M} 
-\left\langle \pderivative{\interpolation{N}{\left(\vv{v}^T J\vv{a}^{2}}\right)}{\xi^{2}},p\right\rangle _{N\times M} \\
& + \quad  \sum_{\sigma=0}^{M}\omega_{\sigma}\sum_{i,k=0}^{N}\omega_{ik}\left\{ \left(\vv{F}_{\sigma iNk}^{\kappa}+p_{\sigma iNk}\vv{v}_{\sigma iNk}\right)^{T}J\vv{a}_{iNk}^{2}-\left(\vv{F}_{\sigma i0k}^{\kappa}+p_{\sigma i0k}\vv{v}_{\sigma i0k}\right)^{T}J\vv{a}_{i0k}^{2}\right\} 
\end{split}
\end{align}
for the second sum on the right hand side in \eqref{KEPVolumeContribution1} and
\begin{align}\label{KEPVolumeContribution10}
\begin{split}
& \quad \quad \sum_{\sigma=0}^{M}\omega_{\sigma}\sum_{i,j,k,m=0}^{N}
\omega_{ij}2\matx{Q}_{km}\mathbf{V}_{\sigma ijk}^{T}\left(\blockvec{\mathbf{F}}{}^{\text{KEP}}\left(\mathbf{U}_{\sigma ijk},\mathbf{U}_{\sigma mjk}\right)\cdot\avg{J\vv{a}^{3}}_{ij\left(k,m\right)}\right) \\
=& \quad  \quad \left\langle p\vv{v},\pderivative{\interpolation{N}{\left(J\vv{a}^{3}}\right)}{\xi^{3}}\right\rangle _{N\times M} 
-\left\langle \pderivative{\interpolation{N}{\left(\vv{v}^T J\vv{a}^{3}}\right)}{\xi^{3}},p\right\rangle _{N\times M} \\
& + \quad  \sum_{\sigma=0}^{M}\omega_{\sigma}\sum_{i,j=0}^{N}\omega_{ij}\left\{ \left(\vv{F}_{\sigma ijN}^{\kappa}+p_{\sigma ikN}\vv{v}_{\sigma ijN}\right)^{T}J\vv{a}_{ijN}^{3}-\left(\vv{F}_{\sigma ij0}^{\kappa}+p_{\sigma ik0}\vv{v}_{\sigma ij0}\right)^{T}J\vv{a}_{ij0}^{3}\right\} 
\end{split}
\end{align}
for the third sum on the right hand side in \eqref{KEPVolumeContribution1}. 

It follows that  
\begin{align}\label{KEPVolumeContribution11}
\begin{split}
 \sum_{s=1}^{3}\left\langle \pderivative{\interpolation{N}{\left(\vv{v}^T J\vv{a}^{s}}\right)}{\xi^{s}},p\right\rangle _{N\times M}=
\left\langle \vv{\nabla}_{\xi}\cdot\interpolation{N}{\left(\vv{\tilde{v}}\right)},p\right\rangle _{N\times M}
\end{split}
\end{align}
and by the definition of the discrete surface integrals \eqref{DiscreteSpatialSurface}  
\begin{align}\label{KEPVolumeContribution12}
\begin{split}
& \int\limits _{\partial E^{3},N}\left\langle \tilde{F}_{\hat{n}}^{\kappa}+p\tilde{v}_{\hat{n}},1\right\rangle _{M}\dS \\
= &  
 \sum_{\sigma=0}^{M}\omega_{\sigma}\bigg[\sum_{j,k=0}^{N}\omega_{jk}\left\{ \left(\vv{F}_{\sigma Njk}^{\kappa}+p_{\sigma Njk}\vv{v}_{\sigma Njk}\right)^{T}J\vv{a}_{Njk}^{1}-\left(\vv{F}_{\sigma 0jk}^{\kappa}+p_{\sigma 0jk}\vv{v}_{\sigma 0jk}\right)^{T}J\vv{a}_{0jk}^{1}\right\}  \\
&  \quad  \quad \quad  +  \sum_{i,k=0}^{N}\omega_{ik}\left\{ \left(\vv{F}_{\sigma iNk}^{\kappa}+p_{\sigma iNk}\vv{v}_{\sigma iNk}\right)^{T}J\vv{a}_{\sigma iNk}^{2}-\left(\vv{F}_{\sigma i0k}^{\kappa}+p_{\sigma i0k}\vv{v}_{\sigma i0k}\right)^{T}J\vv{a}_{i0k}^{2}\right\}  \\
& \quad  \quad \quad  +  \sum_{i,j=0}^{N}\omega_{ij}\left\{ \left(\vv{F}_{\sigma ijN}^{\kappa}+p_{\sigma ijN}\vv{v}_{\sigma ijN}\right)^{T}J\vv{a}_{ijN}^{3}-\left(\vv{F}_{\sigma ij0}^{\kappa}+p_{\sigma ij0}\vv{v}_{\sigma ij0}\right)^{T}J\vv{a}_{ij0}^{3}\right\} \bigg].
\end{split}
\end{align}
Moreover, we note that the contravariant coordinate vectors are discretized by \eqref{DiscreteContravariantVectors} and thus the discrete metric identities \eqref{DiscreteMetricIdentities} are satisfied \cite{kopriva2006metric}. Hence, it follows 
\begin{align}\label{KEPVolumeContribution13}
\begin{split}
&\sum_{s=1}^{3}\left\langle p\vv{v},\pderivative{\interpolation{N}{\left(J\vv{a}^{s}}\right)}{\xi^{s}}\right\rangle _{N\times M} \\ 
=&\sum_{\sigma=0}^{M}\omega_{\sigma}\bigg[\sum_{i,j,k=0}^{N}\omega_{ijk}p_{\sigma ijk}\left(v_{1}\right)_{\sigma ijk}\left\{ \sum_{m=0}^{N}\matx{D}_{im}\left(J{a}_{1}^{1}\right)_{mjk}+\matx{D}_{jm}\left(J{a}_{1}^{2}\right)_{imk}+\matx{D}_{km}\left(J{a}_{1}^{3}\right)_{ijm}\right\}  \\
& \qquad  \qquad \qquad    +\omega_{ijk}p_{\sigma ijk}\left(v_{2}\right)_{\sigma ijk}\left\{ \sum_{m=0}^{N}\matx{D}_{im}\left(J{a}_{2}^{1}\right)_{mjk}+\matx{D}_{jm}\left(J{a}_{2}^{2}\right)_{imk}+\matx{D}_{km}\left(J{a}_{2}^{3}\right)_{ijm}\right\}  \\
& \qquad  \qquad \qquad +\omega_{ijk}p_{\sigma ijk}\left(v_{3}\right)_{\sigma ijk}\left\{ \sum_{m=0}^{N}\matx{D}_{im}\left(J{a}_{3}^{1}\right)_{mjk}+\matx{D}_{jm}\left(J{a}_{3}^{2}\right)_{imk}+\matx{D}_{km}\left(J{a}_{3}^{3}\right)_{ijm}\right\} \bigg]=0.
\end{split}
\end{align}

Therefore, by plugging \eqref{KEPVolumeContribution8}, \eqref{KEPVolumeContribution9} and \eqref{KEPVolumeContribution10} in \eqref{KEPVolumeContribution1}, we obtain the desired identity 
\begin{equation}
\left\langle \Dprojection{N}\cdot\blockvec{\tilde{\mathbf{F}}}{}^{\text{KEP}},\mathbf{V}\right\rangle _{N\times M}=-\left\langle \vv{\nabla}_{\xi}\cdot\interpolation{N}{\left(\vv{\tilde{v}}\right)},p\right\rangle _{N\times M}+\int\limits _{\partial E^{3},N}\left\langle \tilde{F}_{\hat{n}}^{\kappa}+p\tilde{v}_{\hat{n}},1\right\rangle _{M}\dS. 
\end{equation}

\section{Entropy conservative and kinetic energy preserving flux of Ranocha}\label{sec:App E}

Here we state the entropy conservative and kinetic energy preserving (ECKEP) flux for the Euler equations recently developed by Ranocha \cite{ranocha2018}. The flux satisfies the entropy conservative spatial condition \eqref{EntropyFunctional2} as well as the KEP conditions of Jameson \eqref{Jameson2} on Cartesian meshes. The flux can be made entropy stable by adding a matrix dissipation of the form \eqref{CartesianFluxesES} where complete details are given in Gassner et al. \cite{Gassner2017}. We rewrite the form of the ECKEP originally presented by Ranocha \cite{ranocha2018} through the use of the identity
\begin{equation}
\frac{1}{4}\jump{a}\jump{b} = \avg{ab} - \avg{a}\avg{b}.
\end{equation}
We do so such that the ECKEP flux only involves arithmetic and logarithmic means. The flux in the $x-$direction is given by
\begin{equation}
\resizebox{\textwidth}{!}{$
\vec{F}_1^{\text{ECKEP}} = \begin{bmatrix}
\rho^{\ln}\avg{v_1}\\[0.15cm]
\rho^{\ln}\avg{v_1}^2 + \avg{p}\\[0.15cm]
\rho^{\ln}\avg{v_1}\avg{v_2}\\[0.15cm]
\rho^{\ln}\avg{v_1}\avg{v_3}\\[0.15cm]
\left(\rho^{\ln}\left(\avg{v_1}^2+\avg{v_2}^2+\avg{v_3}^2-\frac{1}{2}\left(\avg{v_1^2}+\avg{v_2^2}+\avg{v_3^2}\right)\right)+\frac{\rho^{\ln}}{(\gamma-1)\left(\frac{\rho}{p}\right)^{\ln}}\right)\avg{v_1} + 2\avg{p}\avg{v_1}-\avg{p v_1}\\[0.15cm]
\end{bmatrix}.
$}
\end{equation}
The ECKEP flux in the other Cartesian directions is then easily recovered by rotation.
\end{document}